\documentclass{amsart}
\usepackage{color}

\allowdisplaybreaks[3]

\usepackage{amsfonts}
\usepackage{epsf,latexsym,amsfonts,amsbsy,mathrsfs}
\usepackage{amsmath, amssymb, amsthm,amscd,amsxtra}
\usepackage[]{graphicx,subfigure}
\usepackage{epstopdf}
\usepackage{float}
\usepackage{placeins}
\usepackage{amsmath}
\usepackage{amssymb}
\usepackage{multirow}
\usepackage{float}
\usepackage{stmaryrd}
\usepackage{epsfig}
\usepackage{tikz}
\usepackage{cite}
\usetikzlibrary{positioning} 
\usepackage{xcolor}
\usepackage{marginnote}
\usepackage{sidenotes}
\usepackage{booktabs}
\usepackage{multirow}
\usepackage{arydshln}

\usepackage{changes}
\usepackage{lipsum}
\usepackage[colorlinks=true, linkcolor=blue, citecolor=blue]{hyperref}
\usepackage{cleveref}
\allowdisplaybreaks

\newtheorem{theorem}{Theorem}[section]
\newtheorem{lemma}[theorem]{Lemma}
\newtheorem{corollary}[theorem]{Corollary}

\crefname{theorem}{Theorem}{Theorems}
\crefname{lemma}{Lemma}{Lemmas}
\crefname{corollary}{Corollary}{Corollary}
\crefname{assumption}{Assumption}{Assumption}

\AddToHook{env/theorem/begin}{\crefalias{theorem}{theorem}}
\AddToHook{env/lemma/begin}{\crefalias{theorem}{lemma}}
\AddToHook{env/corollary/begin}{\crefalias{theorem}{corollary}}
\AddToHook{env/assumption/begin}{\crefalias{theorem}{assumption}}

\theoremstyle{definition}

\newtheorem{assumption}[theorem]{Assumption}

\theoremstyle{remark}
\newtheorem{remark}[theorem]{Remark}

\numberwithin{equation}{section}

\newcommand{\ignore}[1]{}

\newcommand{\Be}{{\boldsymbol{e}}}

\newcommand{\Bn}{{\boldsymbol{n}}}

\newcommand{\Bq}{{\boldsymbol{q}}}
\newcommand{\Br}{{\boldsymbol{r}}}

\newcommand{\Bv}{{\boldsymbol{v}}}
\newcommand{\Bw}{{\boldsymbol{w}}}

\newcommand{\BH}{{\boldsymbol{H}}}

\newcommand{\BV}{{\boldsymbol{V}}}

\begin{document}

    \title[Physics-Driven Multiscale Reduction For Helmholtz Equation]
{A Physics-Driven Multiscale Method for High-Frequency Helmholtz Equations}\thanks{The work was supported by National Key Research and Development Project of China (Grant No. 2024YFA1012600). Huangxin Chen was supported by National Natural Science Foundation of China (Grant No. 12471345, 12526102) and Fujian Provincial Natural Science Foundation of China (Grant No. 2026J011002). Shubin Fu was supported by the National Natural Science Foundation of China (Grant No. 12301514). Zhiyu Tan was supported by the National Natural Science Foundation of China (Grant No. 12401542) and the Natural Science Foundation of Fujian Province of China (Grant No. 2026J001044).}

	\author{Huangxin Chen}
	\address{School of Mathematical Sciences and Fujian Provincial Key Laboratory on Mathematical Modeling and
		High Performance Scientific Computing, Xiamen University, Fujian, 361005, China}
	\email{chx@xmu.edu.cn}
    
    \author{Jianhui Chen}
	\address{School of Mathematical Sciences and Fujian Provincial Key Laboratory on Mathematical Modeling and
		High Performance Scientific Computing, Xiamen University, Fujian, 361005, China}
	\email{chenjianhui1@stu.xmu.edu.cn}
    
    \author{Shubin Fu}
\thanks{Corresponding author: Shubin Fu (\texttt{sfu@eitech.edu.cn}).}
\address{School of Mathematical Sciences, Eastern Institute of Technology,
Ningbo, Zhejiang 315200, China}
\email{sfu@eitech.edu.cn}

	\author{Zhiyu Tan}
	\address{School of Mathematical Sciences and Fujian Provincial Key Laboratory on Mathematical Modeling and
		High Performance Scientific Computing, Xiamen University, Fujian, 361005, China}
	\email{zhiyutan@xmu.edu.cn}
    
    \pagestyle{myheadings}
    \thispagestyle{plain}
    \markboth{Physics-Driven Multiscale Reduction For The Helmholtz Equation}{Huangxin Chen, Jianhui Chen, Shubin Fu, Zhiyu Tan}
	
	

	\date{}
	
\begin{abstract}
We develop and analyze a physics-driven multiscale method for model reduction
of the Helmholtz equation at high frequency. The method constructs a reduced
approximation space from local wave responses within a mixed finite element
formulation with velocity elimination. These responses are computed by
solving Helmholtz problems with boundary excitations on oversampled regions,
giving snapshots that incorporate the oscillatory behavior and propagation
characteristics of the underlying medium. Local generalized eigenvalue
problems then select modes from this physically informed space to form a
reduced global system with sparse interelement coupling. To establish
stability and convergence of this reduction, we relate the local spectral
approximation to approximation of the global adjoint problem. Starting
from an inf--sup estimate for the fine-grid formulation, this argument
yields a multiscale inf--sup condition governed by fine-grid resolution,
oversampling, and the first omitted local eigenvalues. The resulting
stability establishes well-posedness and quasi-optimality of the multiscale
approximation in a discrete energy norm. Together with local spectral
estimates, it gives an a priori error bound that explicitly quantifies the
roles of the coarse and fine mesh sizes, oversampling depth, spectral
truncation, and wavenumber. This analysis provides a quantitative link
between local space reduction and the stability and accuracy of the
global approximation. Numerical experiments in two and three dimensions
demonstrate accurate wavefield approximation with substantial reductions
in global dimension and solution cost for homogeneous and strongly
heterogeneous media, including perfectly matched layer truncations.
\end{abstract}
	
	\maketitle
	
	\section{Introduction}

The Helmholtz equation models time-harmonic wave propagation and scattering
in acoustics, seismic imaging, and optical devices
\cite{VirieuxOperto2009,GaoFuChung2018,GiammatteoHeinleinLedererSchlottbom2025}.
Its numerical approximation becomes particularly demanding when high
frequency and material heterogeneity occur together. The mesh must resolve
both the wavelength and the fine scales of the medium, while the pollution
effect in standard low-order finite element methods can require refinement
beyond wavelength resolution
\cite{IhlenburgBabuska1995,IhlenburgBabuska1997,BabuskaSauter1997}.
The resulting large, indefinite systems are also challenging to solve
\cite{GanderZhang2019}. In imaging and inversion, this cost is compounded by
the need to compute wavefields for many sources at a fixed medium and
frequency. A reduced approximation space that retains the medium-dependent
wave information and can be reused across sources is therefore of
particular interest.

Two closely related questions govern the construction of such a space:
how to represent oscillatory solutions efficiently, and how to preserve
stability when these functions are coupled globally. These questions have
motivated several developments in Helmholtz discretization. Generalized
and partition-of-unity finite element methods incorporate problem-adapted
enrichment functions into the approximation space
\cite{BabuskaIhlenburgPaikSauter1995,MelenkBabuska1996}.
Interior-penalty discontinuous Galerkin and continuous interior-penalty
methods modify the coupling to improve stability and control pollution
\cite{FengWu2009,Wu2014}, while hybridizable discontinuous Galerkin methods
admit wavenumber-explicit stability and error analyses
\cite{ChenLuXu2013}. Trefftz methods build the local wave equation directly
into the approximation functions, with plane-wave discontinuous Galerkin
methods providing an important example
\cite{HiptmairMoiolaPerugia2011,HiptmairMoiolaPerugia2016}.
For general heterogeneous media, however, prescribed plane waves must be
replaced by functions adapted to the local coefficients. This leads
naturally to multiscale spaces generated by localized numerical solves.

Multiscale methods transfer fine-scale information into a smaller globally
coupled problem. Early approaches to heterogeneous Helmholtz equations
include subgrid Green-function methods, multiscale asymptotic
constructions, and heterogeneous multiscale methods
\cite{OberaiPinsky1998,CaoCuiZhu2002,CiarletStohrer2014,OhlbergerVerfurth2018}.
Homogenization-based constructions exploit an effective description or a
two-scale structure. For media without explicit scale separation, local
fine-grid problems provide another means of incorporating the coefficients
into the approximation space, as in the multiscale finite element method
\cite{HouWu1997}. The generalized multiscale finite element method
(GMsFEM) develops this approach through a snapshot--spectral construction:
local solution spaces are first sampled through fine-scale problems and
then compressed by local generalized eigenvalue problems
\cite{EfendievGalvisHou2013,ChungEfendievHou2016}. The selected eigenmodes
retain medium-dependent wave information while reducing the number of
global unknowns. This construction has been applied to heterogeneous wave
propagation and Helmholtz problems
\cite{ChungEfendievLeung2014,FuGao2017}, with further developments using
edge-based, wavelet-based, and scattering-oriented spaces
\cite{FuLiCrasterGuenneau2021,FuChungLi2021,KalachikovaVasilyevaHarrisChung2023}.

For the indefinite Helmholtz operator, the analysis of this reduction must
connect local approximation to global stability. Operator-adapted
multiscale methods provide several ways of establishing this connection.
Localized orthogonal decomposition (LOD) constructs coarse functions
through localized fine-scale corrections and yields stable, quasi-optimal
approximations under resolution and oversampling conditions
\cite{GallistlPeterseim2015,Peterseim2017,BrownGallistlPeterseim2017,PeterseimVerfurth2020}.
Multi-resolution and super-localized variants further improve the
localization of the fine-scale responses
\cite{HauckPeterseim2022,FreeseHauckPeterseim2024}.
From a spectral perspective, multiscale spectral generalized finite element
methods construct locally optimal spaces and establish wavenumber-explicit
convergence estimates \cite{MaAlberScheichl2023}; exponentially convergent
local approximation spaces have also been developed for heterogeneous
Helmholtz equations \cite{ChenHouWang2023}. A related approach is approximate
component mode synthesis, which separates local Helmholtz problems from a
global interface problem and uses local eigenmodes to reduce the latter
\cite{GiammatteoHeinleinSchlottbom2025,GiammatteoHeinleinLedererSchlottbom2025}.
Related operator-adapted constructions have also been used to build
coarse spaces for two-level domain decomposition methods, including
local spectral spaces and LOD spaces
\cite{MaAlberScheichlZhang2025,LuXuZhengZou2025}.

These developments identify local approximation and stable global coupling
as complementary ingredients of Helmholtz reduction. Their analyses depend
on the particular local solution spaces, localization procedures, and
variational formulations employed. In a snapshot--spectral method, the
local eigenproblem controls the modes discarded from each snapshot space;
the reduced Helmholtz problem couples the retained modes through an
indefinite sesquilinear form. An analysis of this construction must therefore
relate the local spectral estimates to the inf--sup condition of that form
on the assembled space. Oversampling adds a further step to this argument,
since the spectral approximation is constructed on an enlarged region and
then restricted to its target coarse element. Establishing this connection
provides a mathematical basis for choosing local spectral dimensions and
oversampling regions in the reduced model.

The global coupling of the local spaces is an essential part of this
question. Conforming multiscale methods commonly multiply local functions
by a partition of unity \cite{FuGao2017,FuLiCrasterGuenneau2021}; the
resulting overlap influences both the approximation estimates and the
connectivity of the reduced matrix. Mixed formulations offer a different
structure. Their scalar variable is discontinuous, whereas the flux belongs
to an $H(\mathrm{div})$-conforming space, allowing normal-flux continuity and
local conservation. Mixed multiscale methods have been developed for
heterogeneous elliptic problems
\cite{ChenHou2003,ChungEfendievLee2015,AlberMaScheichl2025}, and
hybrid-mixed formulations have been analyzed for heterogeneous Helmholtz
equations \cite{ChaumontFreletValentin2020}.
With suitable quadrature, the flux variables can be eliminated locally,
leaving only the scalar wave field in the global system and allowing local
flux recovery \cite{WheelerYotov2006}. This velocity-eliminated structure
has been incorporated into GMsFEM for heterogeneous flow
\cite{ChenChungHeSun2020}. For Helmholtz model reduction, it permits local
wave spaces to be coupled through a scalar formulation without extending
their supports by a partition of unity.

In this paper, we develop and analyze a velocity-eliminated mixed GMsFEM
for the Helmholtz equation. Starting from a mass-lumped lowest-order
Raviart--Thomas discretization, we generate snapshots by solving local
Helmholtz problems with mixed Dirichlet--impedance boundary conditions on
oversampled regions. Varying the Dirichlet excitation along the local
boundary samples wave fields entering the region, while impedance
conditions provide an absorbing closure on the remaining boundary.
Each snapshot therefore incorporates the oscillatory behavior and
medium-dependent wave interactions through the governing equation.
This gives the local approximation space its physics-driven character:
the subsequent spectral reduction selects modes from these local wave
responses. The selected modes are then restricted to individual coarse elements. The
regions used to compute the basis functions thus overlap, but their
supports in the reduced model do not. Consequently, oversampling enriches
the information available to each local space while the reduced matrix
retains coupling only between neighboring coarse elements. This
distinction is important for reducing both the dimension and the
factorization cost of the global system.

This boundary-driven construction also separates local wave propagation
from the response to the volume source. Local particular solutions
account for the forcing, and the total approximation
is the sum of this source correction and the field determined by the
reduced Galerkin problem. The correction enters both the reduced
right-hand side and the reconstruction, leaving the reduced matrix
unchanged. This decomposition gives the method its source-independent
approximation space and also underlies the stability and error analysis.

The main analytical contribution is a stability and convergence analysis
that relates this local spectral construction to the global indefinite
problem. We first establish an inf--sup bound for the velocity-eliminated
fine-grid formulation using continuous adjoint regularity and discrete
consistency estimates. We then transfer stability to the multiscale space
through an adjoint approximation argument. On each oversampled region, the
fine-grid adjoint solution is decomposed into a source-driven component and
a locally homogeneous component represented by the snapshots. Spectral
approximation of the latter, together with an estimate of the former,
gives a bound on the adjoint approximation error. This bound yields an
explicit criterion for the multiscale inf--sup condition in terms of
fine-grid resolution, oversampling, and the neglected local eigenmodes.
The local spectra therefore enter not only the approximation estimate but
also the stability of the reduced Galerkin formulation.

With this stability result, Galerkin orthogonality gives quasi-optimality
of the total, source-corrected approximation in the discrete
energy norm: its error is bounded by the best approximation available in
the corrected multiscale space, multiplied by a stability factor.
A corresponding decomposition of the fine-grid solution
reduces the best-approximation estimate to local spectral truncation on
the oversampled regions. Combining the two arguments yields an a priori
error estimate that quantifies the dependence on the coarse scale $H$,
fine-grid spacing $h$, oversampling depth $m$, first omitted local
eigenvalues, and wavenumber $k$. In this way, the analysis links the
selection of local modes to stability and error control of the reduced
solution.

The resulting method has a natural offline--online structure. At fixed
medium and frequency, the snapshots, selected modes, and reduced
factorization are constructed once; each new source requires local
particular solutions and a solve with the same reduced operator.
The independent local constructions can be carried out in parallel,
whereas the globally coupled computation uses only the retained spectral
coefficients. This makes the reduction of global dimension and sparse
coarse coupling directly relevant to repeated-source wave simulation.
We assess this reduction in two- and three-dimensional
experiments with homogeneous, Marmousi, and Overthrust media, examining
the effects of oversampling and local spectral dimension on accuracy
alongside the cost of the global solve. Polynomial boundary traces further
reduce the snapshot construction cost in three dimensions. The experiments
include impedance and global PML truncations and show that accurate
fine-grid wavefield approximation can be obtained with substantially
smaller global systems and lower solution costs.

The remainder of the paper is organized as follows. The fine-grid formulation is presented in \Cref{sec-2}, followed by the
local space construction in \Cref{sec-3}. Stability and approximation
estimates are proved in \Cref{sec-4,sec-5}, and \Cref{sec-6} reports the
numerical experiments.



	\section{A mixed finite element method for the Helmholtz equation}\label{sec-2}In this section we state the problem under consideration and present a mixed finite element method with velocity elimination. The key idea is to apply a suitable quadrature rule to the flux mass term, which permits local elimination of the discrete flux variable. As a result, the global system involves only the scalar wave field, and the flux can be recovered locally once
the scalar solution has been obtained. This velocity-elimination strategy is motivated by the mixed generalized multiscale finite element framework originally developed for single-phase flow problems in \cite{ChenChungHeSun2020,HeChungChenChen2021}. 
\subsection{The continuous problem}
	Let $\Omega\subset \mathbb{R}^2$ be a bounded convex polygonal domain with a Lipschitz continuous boundary $\partial\Omega$. We consider the Helmholtz equation as
\begin{subequations}
\label{Hel_imp}
    \begin{align}
\label{Hel}
    -\Delta u-k^{2}u&=f\quad\text{ in }\Omega,\\
\label{imp}
    \frac{\partial u}{\partial \Bn}-ik u&=g\quad\text{ on }\partial\Omega.
\end{align}
\end{subequations}
where $k>0$ is a constant wavenumber, $i = \sqrt{-1}$ denotes the imaginary unit, and $\Bn$ is the unit outward normal vector on $\partial\Omega$.

By introducing  $\Bq:=-\nabla u$ the auxiliary flux variable,  the second-order problem \eqref{Hel_imp} can be reformulated in an equivalent mixed form, i.e.,
\begin{subequations}
\label{mix_hel_imp}
    \begin{align}
\label{mix_Hel}
    \Bq+\nabla u&=0 \quad\text{ in }\Omega,\\
    \nabla\cdot\Bq-k^2u&=f \quad\text{ in }\Omega,\\
    \label{mix_imp}
    -\Bq\cdot\Bn-ik u&=g \quad\text{ on }\partial\Omega.
\end{align}
\end{subequations}
Let
\[
\BV := \BH(\mathrm{div},\Omega;\mathbb{C})\ \ \mbox{and}\ \ 
W := L^2(\Omega;\mathbb{C})
\]
be the standard Sobolev spaces (see
\cite{BrezziFortin1991,GiraultRaviart1986}).
For a subdomain $D\subset\Omega$, $(\cdot,\cdot)_D$ denotes the
$L^2(D)$ inner product, linear in its first argument, with the Euclidean
dot product used for vector fields. Similarly,
$\langle a,b\rangle_{\partial D}=\int_{\partial D}a\overline b\,ds$
is defined for $a,b\in L^2(\partial D)$.
Throughout, the real part, imaginary part, and complex conjugate of a
complex number $z$ are denoted by $\Re z$, $\Im z$, and $\overline z$,
respectively.

Now, the variational form of \eqref{mix_hel_imp} is stated as follows: find \((\Bq,u)\in \BV\times W\) such that
\begin{subequations}
\label{dis_mix}
    \begin{align}
\label{dis_mix_Hel}
    (\Bq,\Bv)_\Omega-(u,\nabla\cdot \Bv)_\Omega+\langle \frac{i}{k}\Bq\cdot \Bn,\Bv\cdot \Bn\rangle_{\partial\Omega}
    &=-\langle \frac{i}{k}g,\Bv\cdot \Bn\rangle_{\partial\Omega}, \qquad &\forall\Bv \in \BV,\\
    \label{dis_mix_imp}
    -(\nabla\cdot\Bq,p)_\Omega+(k^2u,p)_\Omega&=-(f,p)_\Omega, \qquad&\forall p\in W.
\end{align}
\end{subequations}

\subsection{A mixed finite element method with velocity elimination}
For clarity of exposition, we restrict the discussion to the two-dimensional case, and the three-dimensional case follows analogously. We also restrict our discussion to shape-regular square meshes, and the arguments and results presented in this paper remain valid for shape-regular rectangular meshes.

Let $\mathcal{T}_h$ be a shape-regular Cartesian partition of
$\Omega$ into square elements.
The set of all edges is denoted by $\mathcal{E}_h$, with
$\mathcal{E}_h^0$ and $\mathcal{E}_h^B$ denoting, respectively, the
sets of interior edges and edges lying on
$\partial\Omega$. 

We recall the standard \(RT_0-P_0\) mixed finite element spaces on $\mathcal{T}_h$ here and denote them by
\[
\BV_h
=
\left\{
\Bv\in \BV:\ \Bv|_t\in RT_0(t),\ \forall t\in\mathcal{T}_h
\right\}
\quad\text{and}\quad
W_h
=
\left\{
q\in W:\ q|_t\in P_0,\ \forall t\in\mathcal{T}_h
\right\},
\]
where \(P_0\) denotes the polynomial space of zero degree. For each edge $e\in\mathcal E_h$, let $\Bv_e\in\BV_h$
denote the Raviart--Thomas basis function associated with $e$ (see Figure \ref{two_grid}),
normalized by 
\begin{align*}
    \int_{e'} \Bv_e\cdot\Bn_{e'}\,ds
=
\delta_{ee'},
\quad
e'\in\mathcal E_h .
\end{align*}
The family
$\{\Bv_e\}_{e\in\mathcal E_h}$ forms a basis of $\BV_h$.
Consequently, for any $\Bw,\Bv\in\BV_h$, there exist unique
coefficients $\{w_e\}_{e\in\mathcal E_h}$ and
$\{v_e\}_{e\in\mathcal E_h}$ such that
\begin{align*}
    \Bw
=
\sum_{e\in\mathcal E_h} w_e\Bv_e\ \ \mbox{and}\ \ 
\Bv
=
\sum_{e\in\mathcal E_h} v_e\Bv_e .
\end{align*}

We also introduce the standard jump operator $\llbracket \cdot \rrbracket_e$ across the edges. For any $u_h\in W_h$, (1) if $e\in\mathcal{E}_h^0$ is an interior edge shared by two elements $t_1$ and $t_2$ (see Figure \ref{two_grid}), then, for an arbitrary but fixed
ordering of $t_1$ and $t_2$, with $\Bn_e$ pointing from $t_1$ to $t_2$, $\llbracket u_h \rrbracket_e
    =
    u_h|_{t_1}-u_h|_{t_2}$; 
(2) if $e\in\mathcal{E}_h^B$ is a boundary edge and $t$ is the unique
element adjacent to $e$, with $e\subset\partial t$, $\llbracket u_h \rrbracket_e = u_h|_t$. 

\begin{figure}[H]
\centering
\includegraphics[width=0.5\textwidth]{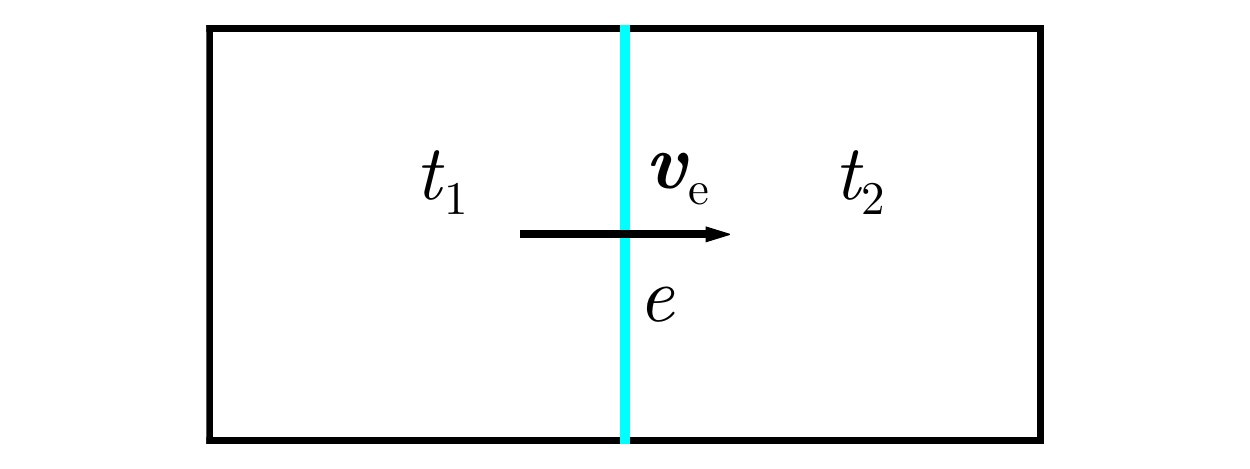}
\caption{\small Two elements $t_1$ and $t_2$ in $\mathcal{T}_h$ sharing an edge $e$ in $\mathcal{E}_h$, $\Bv_e$ denotes the basis function for velocity associated with $e$.}
\label{two_grid}
\end{figure}

The mixed finite element discretization of
~\eqref{mix_hel_imp}
is to find \((\Bq_h,u_h)\in\BV_h\times W_h\) such that
    \begin{align*}
(\Bq_h,\Bv_h)_{\Omega}
-(u_h,\nabla\cdot\Bv_h)_\Omega
+\langle
\frac{i}{k}\Bq_h\cdot\Bn,
\Bv_h\cdot\Bn
\rangle_{\partial\Omega}
&=
-\langle
\frac{i}{k}g,
\Bv_h\cdot\Bn
\rangle_{\partial\Omega},
\\
-(\nabla\cdot\Bq_h,p_h)_\Omega
+(k^2u_h,p_h)_\Omega
&=
-(f,p_h)_\Omega ,
\end{align*}
for all
\(\Bv_h\in\BV_h\) and \(p_h\in W_h\). 
In the implementation, the
inner product \((\Bq_h,\Bv_h)_\Omega\) is usually approximated by
some quadrature rules \cite{HeChungChenChen2021,RussellWheeler1983,KlausenWinther2006,WheelerYotov2006,HouSunChen2014}. The resulting formulation reads
\begin{subequations}
\label{weak_form_t}
    \begin{align}
\label{Gms_1}
(\Bq_h,\Bv_h)_{Q,\Omega}
-(u_h,\nabla\cdot\Bv_h)_\Omega
+\langle
\frac{i}{k}\Bq_h\cdot\Bn,
\Bv_h\cdot\Bn
\rangle_{\partial\Omega}
&=
-\langle
\frac{i}{k}g,
\Bv_h\cdot\Bn
\rangle_{\partial\Omega},
\\
\label{Gms_2}
-(\nabla\cdot\Bq_h,p_h)_\Omega
+(k^2u_h,p_h)_\Omega
&=
-(f,p_h)_\Omega ,
\end{align}
\end{subequations}
where \((\Bq_h,\Bv_h)_{Q,\Omega}\) denotes the quadrature approximation
of \((\Bq_h,\Bv_h)_\Omega\). 
Note that applying proper quadrature rules to \eqref{weak_form_t}
allows the flux variable $\Bq_h$ to be eliminated. This elimination is
described in detail below.
Following \cite{HeChungChenChen2021}, we use the elementwise
trapezoidal quadrature rule 
$$(\Bw,\Bv)_{Q,t}
=
\frac{1}{2}\sum_{e\in\partial t}
w_e \overline{v_e},\quad \forall\ t\in \mathcal{T}_h.
$$ 
The corresponding global quadrature form is then defined by
\begin{align}
(\Bw,\Bv)_{Q,\Omega}
=
\sum_{t\in\mathcal T_h}(\Bw,\Bv)_{Q,t}
=
\frac{1}{2}
\sum_{t\in\mathcal T_h}
\sum_{e\in\partial t}
w_e\overline{v_e}.
\label{tra}
\end{align}

\begin{remark}
According to the results in \cite{WheelerYotov2006}, the norm $\|\cdot\|_{Q,\Omega}$ induced by \eqref{tra} is equivalent to the $L^2$-norm $\|\cdot\|_{0,\Omega}$ on the space $\BV_h$.
\end{remark}

For the discrete flux $\Bq_h\in\BV_h$, we write 
$$\Bq_h
=
\sum_{e\in\mathcal E_h} q_e \Bv_e,$$ 
where $q_e$ denotes the coefficient of $\Bq_h$ associated with the edge $e$.

For an interior edge $e\in\mathcal{E}_h^0$, setting $\Bv_h=\Bv_e$ in \eqref{Gms_1}, the first term on the
left-hand side of \eqref{Gms_1} becomes
\begin{align}
\label{w5.1}
(\Bq_h,\Bv_e)_{Q,\Omega}
=
(\Bq_h,\Bv_e)_{Q,t_1}
+
(\Bq_h,\Bv_e)_{Q,t_2}
=
\frac{1}{2}q_e
+
\frac{1}{2}q_e
=
q_e,
\end{align}
 and the second term in the right-hand side of \eqref{Gms_1} turns into
\begin{align}
(u_h,\nabla\cdot \Bv_e)_\Omega
&=
(u_h,\nabla\cdot \Bv_e)_{t_1}
+
(u_h,\nabla\cdot \Bv_e)_{t_2}
\nonumber\\
&=
\langle u_h|_{t_1},\Bv_e\cdot \mathbf{n}_e\rangle_e
-
\langle u_h|_{t_2},\Bv_e\cdot \mathbf{n}_e\rangle_e
\label{w5.2}\\
&=
u_h|_{t_1}-u_h|_{t_2} = \llbracket u_h \rrbracket_e\nonumber.
\end{align} 
    
Combining \eqref{w5.1} and \eqref{w5.2}, we obtain
\begin{align}
\label{q_u_jump}
    q_e=\llbracket u_h\rrbracket_e,
\quad
\forall e\in \mathcal{E}_h^0 .
\end{align}

For a boundary edge $e\in\mathcal{E}_h^B$, we take $\Bv_h=\Bv_e$ in
\eqref{Gms_1}. Then
\begin{align*}
\frac{1}{2}q_e - u_h|_t
+ \frac{i}{k}q_e
\langle \Bv_e\cdot\Bn,\Bv_e\cdot\Bn\rangle_e +\frac{i}{k}
\langle g,\Bv_e\cdot\Bn\rangle_e
=0,
\end{align*}
where $t$ is the unique element such that $e\subset\partial t$.
A simple calculation yields
$$
q_e
= \frac{2k h}{k h+2i}u_h|_t -\frac{2hi}{k h+2i}g_e^{\rm av}
= \frac{2k h}{k h+2i}\llbracket u_h \rrbracket_e -\frac{2hi}{k h+2i}g_e^{\rm av},
$$
where $g_e^{\rm av}:=\frac{1}{|e|}\int_e g\,ds$ 
denotes the average of $g$ over the edge $e$. 
For later use, we introduce the coefficients
\begin{align*}
\alpha:=\frac{2k h}{k h+2i},
\qquad
\beta:=\frac{2hi}{k h+2i}, 
\end{align*}
which implies that
\begin{align}
\label{q_u_bjump}
q_e
=
\alpha \llbracket u_h\rrbracket_e
-
\beta g_e^{\rm av},\quad \forall e\in \mathcal{E}_h^B.
\end{align}

Therefore, by \eqref{q_u_jump} and \eqref{q_u_bjump}, we have
$$
\Bq_h
= \sum_{e\in\mathcal E_h} q_e \Bv_e
= \sum_{e\in\mathcal E_h^0} \llbracket u_h\rrbracket_e \Bv_e + \sum_{e\in\mathcal E_h^B} (\alpha \llbracket u_h\rrbracket_e
- \beta g_e^{\rm av}) \Bv_e.
$$
Hence, the flux variable $\Bq_h$ can be eliminated locally in \eqref{weak_form_t} and recovered locally once the scalar solution $u_h$ has been computed.

In order to compute the scalar solution $u_h$, we derive the reduced problem only involving the scalar variable $u_h$ now.  

For any edge $e\in \mathcal{E}_h$ and element $t\in \mathcal{T}_h$ satisfying $e\subset \partial t$, we define a scalar $S_{et}=\Bv_e\cdot\Bn_{et}|e|$, where $\Bn_{et}$ is the unit normal vector on $e$ pointing outward of $t$. For example, under the situation shown in Figure~\ref{two_grid} , $S_{et_1}=1$ and $S_{et_2}=-1$. 

By the relationship of $q_e$ and $u_h$ as shown in \eqref{q_u_jump} and \eqref{q_u_bjump}, the first term in the left-hand side of ~\eqref{Gms_2} becomes
\begin{align*}
-(\nabla\cdot \Bq_h,p_h)_\Omega
&=
-\sum_{t\in \mathcal{T}_h}(\nabla\cdot \Bq_h,p_h)_t
=
-\sum_{t\in \mathcal{T}_h}\sum_{e\in \partial t}q_e \overline{p_h}|_t S_{et}
=
-\sum_{e\in \mathcal E_h} q_e\overline{\llbracket p_h \rrbracket_e}
\nonumber\\
&=
-\sum_{e\in \mathcal{E}_h^0}\llbracket u_h \rrbracket_e \,\overline{\llbracket p_h \rrbracket_e}
-\frac{2k h}{k h+2i}\sum_{e\in \mathcal{E}_h^{B}} \llbracket u_h \rrbracket_e \,\overline{\llbracket p_h \rrbracket_e}
+\frac{2i}{k h+2i}\langle g,p_h\rangle_{\partial\Omega}.
\end{align*}
Thus, we obtain the sesquilinear form $a_h(\cdot,\cdot)$ as
\begin{align*}
a_h(r,q)
= \sum_{e\in \mathcal{E}_h^0}\llbracket r \rrbracket_e \,\overline{\llbracket q \rrbracket_e}
 + \frac{2k h}{k h+2i}\sum_{e\in \mathcal{E}_h^{B}} \llbracket r \rrbracket_e \,\overline{\llbracket q \rrbracket_e}-(k^2 r,q)_\Omega,
\quad
\forall r,q\in W_h,
\end{align*}
Consequently, the original mixed discrete formulation \eqref{weak_form_t}
can be reduced to the following problem involving only the pressure variable $u_h$:
Find $u_h\in W_h$ such that
\begin{align}
\label{bf}
a_h(u_h,p_h)
= (f,p_h)_\Omega + \frac{2i}{k h+2i}\langle g,p_h\rangle_{\partial\Omega},
\qquad
\forall p_h\in W_h.
\end{align}

\begin{remark}
The well-posedness of \eqref{bf} is analyzed in Section~\ref{sec-4}.
\end{remark}

	\section{The multiscale space and the multiscale approximation}
	\label{sec-3}
	In this section, we construct the multiscale space $W_{\rm ms}$ and the approximation $u_f+u_{\rm ms}$ of $u_h$, where $u_{\rm ms}\in W_{\rm ms}$. We first introduce the local snapshot spaces, which are constructed from a sufficiently rich collection of basis functions obtained by solving local Helmholtz problems subject to mixed boundary conditions. The snapshot spaces are then reduced by solving suitable local spectral problems, from which the dominant modes are selected. These spectral problems are designed so that the dominant modes capture the essential local features of the solution and provide an efficient reduction of the snapshot spaces. The linear span of the selected modes defines the multiscale space $W_{\rm ms}$.

To define the multiscale space, we need to introduce a proper coarse-grid $\mathcal{T}_H$ with respect to the fine-grid $\mathcal{T}_h$. We construct the coarse mesh $\mathcal{T}_H$ by agglomerating fine-grid
elements here. More precisely, each coarse element
$T_i\in\mathcal{T}_H$ is a connected union of elements of
$\mathcal{T}_h$, i.e., $T_i=\bigcup_{k=1}^{N_i} t_k,
\ \forall t_k\in\mathcal{T}_h,$ 
where $N_i$ denotes the number of fine-grid elements contained in
$T_i$. In the simplest setting, $\mathcal{T}_H$ is generated by a
uniform coarsening of $\mathcal{T}_h$, so that the coarse elements are
also rectangular. For each coarse element $T_i$, we denote by $T_i^m$ its $m$-layer
oversampling region, where $m\geq0$ and $T_i^0=T_i$.
The region $T_i^m$ is constructed recursively by enlarging $T_i$ with
$m$ layers of neighboring fine-grid elements.
Figure~\ref{two_grid_pdf} illustrates a multiscale mesh containing two
neighboring coarse elements $T_i$ and $T_j$, together with the
oversampling region $T_i^m$ for $m=2$.
The common interface between $T_i$ and $T_j$ is denoted by $E_{i,j}$,
while an interior fine-grid edge shared by two neighboring fine
elements is denoted by $e$. For an oversampling region $T_i^m$, let
$\mathcal{E}_{T_i^m}$ denote the set of all fine-grid edges contained
in $\overline{T_i^m}$.
We further denote by $\mathcal{E}_{T_i^m}^0$ the set of interior
fine-grid edges of $T_i^m$ and by $\mathcal{E}_{T_i^m}^B$ the set of
fine-grid edges lying on $\partial T_i^m$.
Finally, let $N_t$ and $N_T$ denote the total numbers of elements in
$\mathcal{T}_h$ and $\mathcal{T}_H$, respectively.

\begin{figure}[H]
\centering
\includegraphics[width=0.5\textwidth]{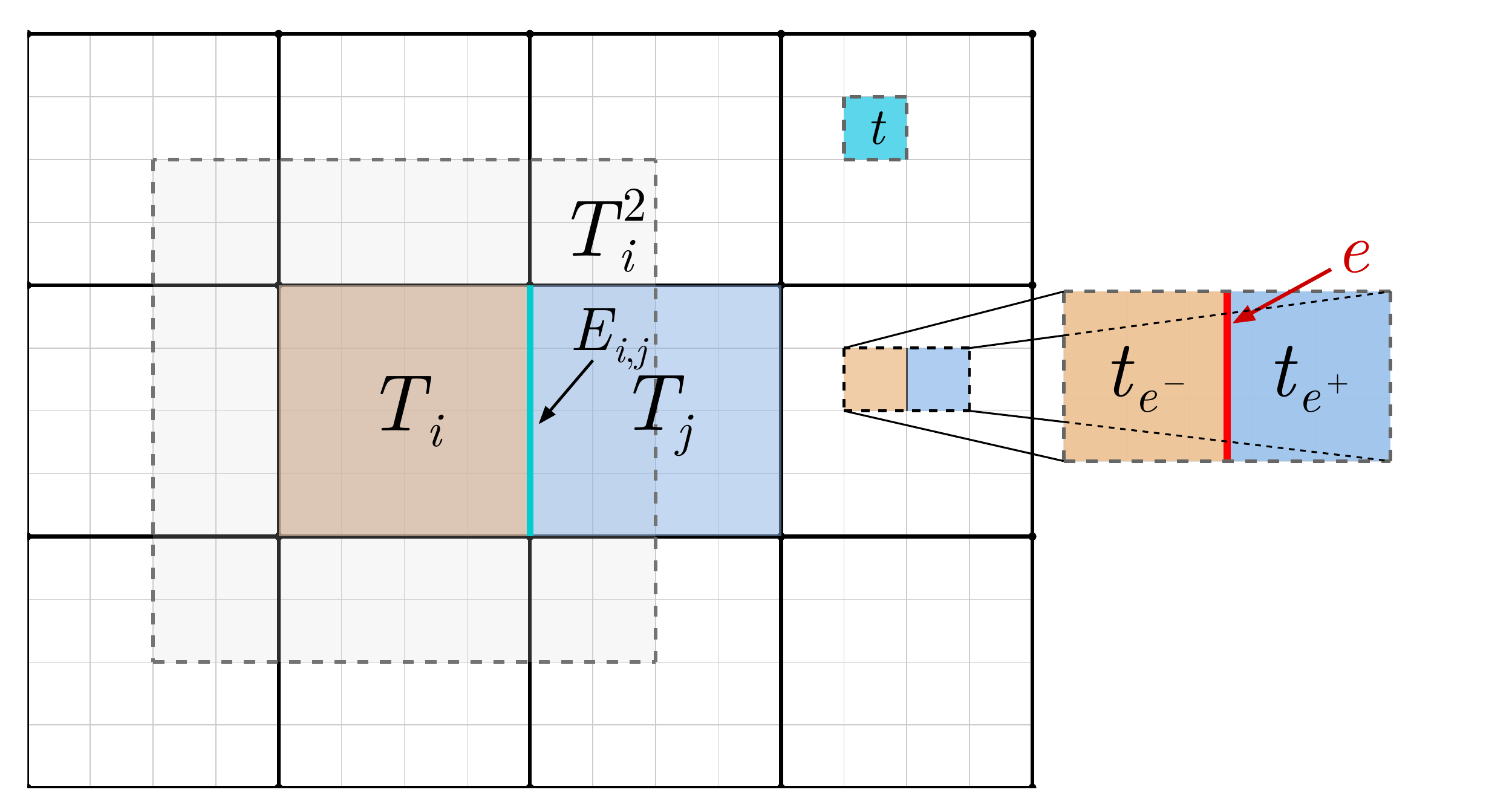}
\caption{\small 
An illustration of the two-scale mesh, including a fine element $t$, 
coarse elements $T_i$ and $T_j$, and the oversampling region $T_i^m$ with $m=2$. 
The interface between coarse elements is denoted by $E_{i,j}$, 
while $e$ represents a fine-scale edge shared by two neighboring fine elements 
$t_{e^-}$ and $t_{e^+}$.}
\label{two_grid_pdf}
\end{figure}


\subsection{ The snapshot space}

Let $T_i\in\mathcal T_H$ be a coarse element and let $T_i^m$ denote its
associated oversampling region. We construct snapshots from fine-grid
boundary excitations. For each fine-grid
edge $e_i^j\subset\partial T_i^m$, consider the local Helmholtz problem
\begin{subequations}
\label{local_snap1}
\begin{align}
\boldsymbol{\Phi}_i^j+\nabla\psi_i^j
&=0
&&\text{in }T_i^m,\\
-k^2\psi_i^j+\nabla\cdot\boldsymbol{\Phi}_i^j
&=0
&&\text{in }T_i^m,\\
\psi_i^j
&=1
&&\text{on }e_i^j,\\
\label{local2}
\frac{\partial\psi_i^j}{\partial n}
-ik\psi_i^j
&=0
&&\text{on }\partial T_i^m\setminus e_i^j.
\end{align}
\end{subequations}

We discretize \eqref{local_snap1}
on the fine grid using the velocity-eliminated mixed finite element method.
The resulting discrete solution is denoted by
$(\boldsymbol{\Phi}_i^j,\psi_i^j)
\in\mathbf V_h(T_i^m)\times W_h(T_i^m)$,
where $\mathbf V_h(T_i^m)$ and $W_h(T_i^m)$ denote the restrictions of
$\mathbf V_h$ and $W_h$, respectively, to $T_i^m$.
The following result establishes that each prescribed boundary excitation
determines a unique snapshot function and the proof is given in Appendix~\ref{proof_A}.
\begin{lemma}
\label{well_posed_of_snap}
For any edgewise constant
$\zeta$ on $e_i^j$, the discrete problem of finding
$(\boldsymbol\Phi_{\zeta,h},\psi_{\zeta,h})
\in
\boldsymbol V_h(T_i^m)\times W_h(T_i^m)$ such that
\begin{align*}
(\boldsymbol\Phi_{\zeta,h},\boldsymbol v_h)_{Q,T_i^m}
-(\psi_{\zeta,h},\nabla\!\cdot\boldsymbol v_h)_{T_i^m}
+\langle\zeta,\boldsymbol v_h\cdot\boldsymbol n\rangle_{e_i^j}
+\frac{i}{k}
\langle
\boldsymbol\Phi_{\zeta,h}\cdot\boldsymbol n,
\boldsymbol v_h\cdot\boldsymbol n
\rangle_{\partial T_i^m\setminus e_i^j}
&=0,\\
(\nabla\!\cdot\boldsymbol\Phi_{\zeta,h},p_h)_{T_i^m}
-(k^2\psi_{\zeta,h},p_h)_{T_i^m}
&=0,
\end{align*}
for all
$(\boldsymbol v_h,p_h)
\in
\boldsymbol V_h(T_i^m)\times W_h(T_i^m)$
admits a unique solution.
\end{lemma}

Let $ \mathcal{E}_{T_i^m}^B
    =
    \left\{e_i^j:\ 1\leq j\leq J_i\right\},$ where $J_i$ denotes the number of fine-grid edges on $\partial T_i^m$.
The full local snapshot space over the oversampling region $T_i^m$ is
defined by $W_{i,m}^{\mathrm{snap}}
    =
    \operatorname{span}
    \left\{
        \psi_i^j:\ 1\leq j\leq J_i
    \right\}.$ To construct snapshot functions associated with the target coarse element
$T_i$, we restrict each oversampled snapshot to $T_i$ and extend the
restriction by zero to the rest of $\Omega$. More precisely, we define
\begin{align*}
\psi_i^{j,\mathrm{snap}}(x)
=
\begin{cases}
    \psi_i^j(x), & x\in T_i,\\
    0, & x\in\Omega\setminus T_i,
\end{cases}
\qquad
1\leq j\leq J_i.
\end{align*}
The corresponding local snapshot space associated with $T_i$ is then
defined by
\begin{align*}
    W_i^{\mathrm{snap}}
    =
    \operatorname{span}
    \left\{
        \psi_i^{j,\mathrm{snap}}:\ 1\leq j\leq J_i
    \right\}.
\end{align*}
Consequently, the global snapshot space is
\begin{align*}
    W_{\mathrm{snap}}
    =
    \bigoplus_{i=1}^{N_T} W_i^{\mathrm{snap}}
    =
    \operatorname{span}
    \left\{
        \psi_i^{j,\mathrm{snap}}:
        1\leq j\leq J_i,\;
        1\leq i\leq N_T
    \right\}.
\end{align*}
The above construction uses one local solve for each fine-grid boundary
degree of freedom and hence provides a rich snapshot space. 
\begin{remark}
Polynomial boundary traces reduce the number of local problems by
exciting an entire coarse face at a time. For an edge \((d=2)\) or a face \((d=3)\)
\(F\subset\partial T_i^m\), let
\(\{q_{F,r}\}_{r=1}^{(p+1)^{d-1}}\) be a basis of the tensor-product
polynomial trace space \(\mathbb Q_p(F)\).  For each trace \(q_{F,r}\), we compute
\((\boldsymbol\Phi_{i,F}^r,\psi_{i,F}^r)\) from
\begin{subequations}\label{polynomial-snapshot}
\begin{align}
\boldsymbol\Phi_{i,F}^r+\nabla\psi_{i,F}^r&=0
&&\text{in }T_i^m,\\
-k^2\psi_{i,F}^r+\nabla\!\cdot\boldsymbol\Phi_{i,F}^r&=0
&&\text{in }T_i^m,\\
\psi_{i,F}^r&=q_{F,r}
&&\text{on }F,\\
\frac{\partial\psi_{i,F}^r}{\partial n}-ik\psi_{i,F}^r&=0
&&\text{on }\partial T_i^m\setminus F.
\end{align}
\end{subequations}
The resulting alternative snapshot space is
\[
 W_{i,m}^{\mathrm{snap},p}
 :=\operatorname{span}\{\psi_{i,F}^r:
 F\subset\partial T_i^m,\ 1\le r\le(p+1)^{d-1}\}.
\]

Thus, the two constructions differ in the choice of boundary trace
space: the fine-grid trace construction resolves the boundary data at the
fine-grid level, whereas $W_{i,m}^{\mathrm{snap},p}$ employs a
low-dimensional polynomial trace space. In particular, for a fixed
boundary entity $F$, all polynomial traces share the same local matrix
and can be treated as multiple right-hand sides. Consequently, the
polynomial construction substantially reduces the number of local
factorizations, especially in three dimensions, while the subsequent
spectral reduction can be applied in the same manner as for the
snapshot space.
\end{remark}

\begin{remark}
For a globally PML-truncated problem, the same construction is applied to
the stretched fine-grid operator. We first enlarge the oversampled region
by an auxiliary local PML, solve the boundary-excited problem there, discard
the auxiliary-layer degrees of freedom, and restrict the resulting snapshots
to the target coarse element. The auxiliary layer supplies the absorbing
boundary treatment for local construction, while the global PML remains
part of the fine and reduced operators.
\end{remark}
\subsection{The multiscale space}

We compress $W_{i,m}^{\mathrm{snap}}$ through a local spectral problem
to obtain a low-dimensional coarse space. More precisely, we seek an eigenpair
$(\lambda,u)\in\mathbb{R}\times W_{i,m}^{\mathrm{snap}}$, with $u\neq0$,
such that
\begin{align}
\label{spe_pro}
    a_i(u,v)
    =
    \lambda s_i(u,v),
    \qquad
    \forall v\in W_{i,m}^{\mathrm{snap}},
\end{align}
where
\begin{align*}
a_i(u,v)
&=
\sum_{e\in \mathcal{E}_{T_i^m}^0}
 h\llbracket u\rrbracket_e \overline{\llbracket v\rrbracket_e}+\sum_{e\in \mathcal{E}_{T_i^m}^B\setminus\mathcal{E}_h^B}k h^2
 \llbracket u\rrbracket_e \overline{\llbracket v\rrbracket_e}+\sum_{e\in \mathcal{E}_{T_i^m}^B\cap\mathcal{E}_h^B}h|\alpha|
 \llbracket u\rrbracket_e \overline{\llbracket v\rrbracket_e},\\
s_i(u,v)
&=
\frac{m+1}{H^2}
\left(
\sum_{e\in \mathcal{E}_{T_i^m}\setminus\mathcal{E}_h^B}
\llbracket u\rrbracket_e \overline{\llbracket v\rrbracket_e}+\sum_{e\in \mathcal{E}_{T_i^m}\cap\mathcal{E}_h^B}|\alpha|
\llbracket u\rrbracket_e \overline{\llbracket v\rrbracket_e}
+(k^2u,v)_{T_i^m}
\right).
\end{align*}
The spectral reduction for the alternative snapshot space
$W_{i,m}^{\mathrm{snap},p}$ is defined analogously, with
$W_{i,m}^{\mathrm{snap}}$ in \eqref{spe_pro} replaced by
$W_{i,m}^{\mathrm{snap},p}$.

Let $r_i=\dim W_{i,m}^{\mathrm{snap}}\leq J_i$, and choose a basis
$\{\chi_i^j\}_{j=1}^{r_i}$ of this space. We arrange the eigenvalues
of \eqref{spe_pro} in nondecreasing order, $\lambda_1^{(i)}
    \leq
    \lambda_2^{(i)}
    \leq
    \cdots
    \leq
    \lambda_{r_i}^{(i)}.$ Here, $\lambda_k^{(i)}$ denotes the $k$th eigenvalue associated with the
oversampling region $T_i^m$. In the adjoint and primal approximation estimates,
we write $\lambda_{j,1}^{(i)}$ and $\lambda_{j,2}^{(i)}$, respectively, for
this same eigenvalue $\lambda_j^{(i)}$. Thus $\lambda_{j,1}^{(i)}
=\lambda_{j,2}^{(i)}=\lambda_j^{(i)}$; when $i$ is fixed or appears under a
minimum over coarse elements, we suppress the superscript $(i)$.
Let $Z_k^{(i)}
    =
    \left(
        Z_{k,1}^{(i)},
        \ldots,
        Z_{k,r_i}^{(i)}
    \right)^{\top}$ be the corresponding eigenvector, where $Z_{kj}^{(i)}$ denotes its
$j$th component. Using the coefficients of $Z_k^{(i)}$, we construct
the corresponding spectral basis function by $ \tilde{\psi_i^k}
    =
    \sum_{j=1}^{r_i}
    Z_{k,j}^{(i)}
    \chi_i^{j},
    \quad
    1\leq k\leq r_i.$ We choose these eigenfunctions to be
$s_i$-orthonormal. Their restrictions give the representation $ W_i^{\mathrm{snap}}
    =
    \operatorname{span}
    \left\{
        \psi_i^{k,\mathrm{ms}}:
        1\leq k\leq r_i
    \right\},$
    where
\begin{align*}
\psi_i^{k,\mathrm{ms}}
=
\begin{cases}
    \tilde{\psi_i^k}(x), & x\in T_i,\\
    0, & x\in\Omega\setminus T_i,
\end{cases}
\qquad
1\leq k\leq r_i.
\end{align*}
    To obtain a reduced multiscale space, we retain the first $l_i\leq r_i$ eigenfunctions and define
    \begin{align*}
         W_i^{\mathrm{ms}}
    =
    \operatorname{span}
    \left\{
        \psi_i^{k,\mathrm{ms}}:
        1\leq k\leq l_i
    \right\}.
    \end{align*}
The global multiscale space is given by
\begin{align*}
    W_{\mathrm{ms}}
=
\bigoplus_{i=1}^{N_T} W_i^{\mathrm{ms}}
=
\operatorname{span}
\left\{
\psi_i^{k,\mathrm{ms}}:\ 1\leq k\leq l_i,\ 1\leq i\leq N_T
\right\}.
\end{align*}
When $l_i=r_i$, the local spectral tail vanishes; we use the convention
$\lambda_{r_i+1}^{(i)}=+\infty$ in the error estimates below.

\subsection{The multiscale approximation}
Note that in the previous definition of the multiscale space, we do not involve any information from the source term $f$. To account for the source term $f$ in \eqref{Hel_imp}, we introduce a
source corrector. For each coarse element $T_i\in\mathcal T_H$, we consider
the following local source problem: find
$(\boldsymbol\Phi_{f,i}^m,\psi_{f,i}^m)$ such that
\begin{subequations}\label{eq:local-source-corrector}
\begin{align}
\boldsymbol\Phi_{f,i}^m+\nabla\psi_{f,i}^m
&=0
&&\text{in }T_i^m,
\\
-k^2\psi_{f,i}^m
+\nabla\!\cdot\boldsymbol\Phi_{f,i}^m
&=f
&&\text{in }T_i^m,
\\
\frac{\partial\psi_{f,i}^m}{\partial n}
-ik\psi_{f,i}^m
&=0
&&\text{on }\partial T_i^m.
\end{align}
\end{subequations}

In the actual computation, the local problem
\eqref{eq:local-source-corrector} is discretized on the fine grid using
the velocity-eliminated mixed finite element method. More precisely, the
resulting discrete solution belongs to
$\boldsymbol V_h(T_i^m)\times W_h(T_i^m)$.
For notational simplicity, we suppress the subscript $h$ and continue to
denote the discrete solution by
$(\boldsymbol\Phi_{f,i}^m,\psi_{f,i}^m)$.

The global source corrector $u_f\in W_h$ is then assembled by restricting
each local oversampled corrector to its target coarse element:
\begin{align*}
    \left.u_f\right|_{T_i}
:=
\left.\psi_{f,i}^m\right|_{T_i},
\qquad
\forall T_i\in\mathcal T_H.
\end{align*}
With all the ingredients at hand, we define $u_{\rm ms}\in W_{\rm ms}$ so that $u_f+u_{\rm ms}$ approximates $u_h$:
\begin{align}
\label{GMs}
    a_h(u_{\rm ms},p_h)=(f,p_h)_\Omega-a_{h}(u_f,p_h)+\frac{2i}{k h+2i}\langle g,p_h\rangle_{\partial\Omega},\quad\forall p_h\in W_{\rm ms},
\end{align}
which is actually the velocity-eliminated mixed GMsFEM formulation of the Helmholtz equation.

\begin{remark}
The stability and well-posedness of \eqref{GMs} are analyzed in Section~\ref{sec-4}.
\end{remark}
\begin{remark}
   Oversampling plays an important role in the local multiscale
construction. Enlarging the computational region from $T_i$ to
$T_i^m$ allows the local basis functions to capture information from a
broader neighborhood and, at the same time, mitigates the effect of the
artificial boundary conditions on the approximation restricted to the
target coarse element $T_i$. The influence of the oversampling size is
examined numerically in \Cref{sec-6} through a comparison of the
relative errors obtained with different choices of $T_i^m$.

\end{remark}

    \section{Inf--sup and stability analyses}
    \label{sec-4}

    We establish the inf--sup condition of the multiscale method and the stability properties of the discrete solution in this section. To this end, we first introduce the related norms used in the analysis, and then recall the wave-number-explicit stability and regularity estimates for the continuous Helmholtz problem, which provide the key ingredients for the subsequent discrete analysis. Based on these estimates, we establish the inf--sup stability of the fine-grid discretization and characterize the stability of the discrete solution. Finally, we transfer this stability to the multiscale discretization and prove a corresponding inf--sup condition on the multiscale spaces under suitable resolution and approximation assumptions. These results guarantee the well-posedness of the discrete problems and provide the foundation for the error analysis developed in the subsequent section.


Throughout the analysis, $k\geq k_0>0$, with $k_0$ fixed.
We begin by introducing the $k$-weighted energy norm
\begin{equation*}
    \|u\|_{E,k}^2
    :=
    \|\nabla u\|_{0,\Omega}^2
    +
    k\|u\|_{0,\partial\Omega}^2
    +
    k^2\|u\|_{0,\Omega}^2,
    \qquad
    u\in H^1(\Omega),
\end{equation*}
together with the mesh-dependent norm
\begin{equation*}
    \|u_h\|_{D}^2
    :=
    \sum_{e\in\mathcal{E}_h^0}
    |\llbracket u_h\rrbracket_e|^2
    +
    |\alpha|
    \sum_{e\in\mathcal{E}_h^B}
    |\llbracket u_h\rrbracket_e|^2
    +
    k^2\|u_h\|_{0,\Omega}^2,
    \qquad
    u_h\in W_h,
\end{equation*}
where
$|\alpha| = \frac{2k h}{\sqrt{(k h)^2+4}}$.
In addition, we define the localized mesh-dependent norm
\begin{equation*}
    \|u_h\|_{D,T_i^m}^2
    :=
    \sum_{e\in\mathcal{E}_{T_i^m}^0}
    |\llbracket u_h\rrbracket_e|^2
    +
    \sum_{e\in\mathcal{E}_{T_i^m}^B}
    \alpha_e
    |\llbracket u_h\rrbracket_e|^2
    +
    k^2\|u_h\|_{0,T_i^m}^2,
    \qquad
    u_h\in W_h(T_i^m)
\end{equation*}
with
\begin{equation*}
    \alpha_e
    :=
    \begin{cases}
        1, & e\in\mathcal{E}_h^0,\\[1mm]
        |\alpha|, & e\in\mathcal{E}_h^B.
    \end{cases}
\end{equation*}

Before establishing the discrete inf--sup stability, we give the following wave-number-explicit
stability estimate for the continuous Helmholtz problem \cite{ChenLuXu2013,CummingsFeng2006}.
\begin{lemma}\label{continue_stability} 
Let $(\Bq,u)$ solve \eqref{mix_hel_imp} with $f\in L^2(\Omega)$,
$g=0$, and $k\geq k_0>0$. Then $u\in H^2(\Omega)$ and
\begin{align*}
k\|u\|_{0,\Omega}
+\|u\|_{1,\Omega}+k^{\frac{1}{2}}\|u\|_{0,\partial\Omega}
+k^{-1}\|\boldsymbol q\|_{1,\Omega}
\le 
C_{\rm stab}\|f\|_{0,\Omega}.
\end{align*}
Here $C_{\rm stab}$ depends on $\Omega$ and $k_0$, but not on $k$.
\end{lemma}



To facilitate the stability and error analysis, we further introduce the
auxiliary sesquilinear form
\begin{align*}
    b_h(v,w)
:=
a_h(v,w)+2k^2(v,w)_\Omega,
\qquad
\forall v,w\in W_h.
\end{align*}
The corresponding energy norm is denoted by $\|\cdot\|_D$.
The following lemma summarizes the properties of $a_h(\cdot,\cdot)$
and $b_h(\cdot,\cdot)$ needed in the subsequent analysis.

\begin{lemma}\label{a_and_b}
The following statements hold.
\begin{enumerate}
\item The discrete scheme \eqref{GMs} is consistent:
\begin{equation*}
a_h(u_h-(u_f+u_{\rm ms}),v_{\rm ms})=0
\qquad \forall v_{\rm ms}\in W_{\rm ms}.
\end{equation*}

\item The following continuity estimate holds:
\begin{equation*}\label{eq:lemma37-continuity}
|a_h(w,v)| \le \|w\|_D\,\|v\|_D,\quad|b_h(w,v)| \le \|w\|_D\,\|v\|_D
\qquad \forall w,v\in W_h.
\end{equation*}
\item The sesquilinear form \(b_h(\cdot,\cdot)\) is coercive on \(W_h\):
\begin{equation*}
|b_h(v,v)| \ge \frac{1}{\sqrt{2}}\|v\|_D^2
\qquad \forall v\in W_h.
\end{equation*}
\end{enumerate}
\end{lemma}
\begin{proof}
    (1)
    Since \(W_{\rm ms}\subset W_h\), for any \(v_{\rm ms}\in W_{\rm ms}\), in (\ref{bf}) and (\ref{GMs}), we take
\(p_h=v_{\rm ms}\) in the fine-scale discrete problem. Hence, $a_h(u_h,v_{\rm ms})
=
(f,v_{\rm ms})_\Omega+\frac{2i}{k h+2i}
\langle g,v_{\rm ms}\rangle_{\partial\Omega}.$ 
On the other hand, by the definition of the multiscale solution \(u_{\rm ms}\),
\[
a_h(u_{\rm ms},v_{\rm ms})
=
(f,v_{\rm ms})_\Omega-a_h(u_f,v_{\rm ms})+\frac{2i}{k h+2i}
\langle g,v_{\rm ms}\rangle_{\partial\Omega}.
\]
Subtracting the two identities gives $a_h(u_h-(u_f+u_{\rm ms}),v_{\rm ms})=0,
\quad \forall v_{\rm ms}\in W_{\rm ms}.$\\
(2) 
For any \(u,v\in W_h\), it follows immediately from the Cauchy--Schwarz inequality that
\begin{align*}
|a_h(u,v)| &\le \|u\|_D\|v\|_D,
\qquad
|b_h(u,v)| \le \|u\|_D\|v\|_D.
\end{align*}

\noindent
\textnormal{(3)} By definition of $\alpha$, we have $|\alpha|
=
\frac{2k h}{\sqrt{(k h)^2+4}}. $  Since \(k h>0\), there exists a unique
\(\theta\in(0,\pi/2)\) such that
$\alpha
=
|\alpha| e^{-\mathrm{i}\theta},$  where $\cos\theta
=
\frac{k h}{\sqrt{(k h)^2+4}},$ $
\sin\theta=
\frac{2}{\sqrt{(k h)^2+4}}. $ 
Define the phase-rotated sesquilinear form $\widetilde b_h(w,v)
:=
e^{\mathrm{i}\theta/2}b_h(w,v).$ For any \(v\in W_h\), using
\(\alpha=|\alpha| e^{-\mathrm{i}\theta}\), we have
\begin{align*}
\operatorname{Re}\widetilde b_h(v,v)
&=
\operatorname{Re}
\left[
e^{\mathrm{i}\theta/2}
\left(
\sum_{e\in\mathcal E_h^0}|\llbracket v\rrbracket_e|^2
+
\alpha
\sum_{e\in\mathcal E_h^B}|\llbracket v\rrbracket_e|^2
+
k^2\|v\|_{0,\Omega}^2
\right)
\right]
\\
&=
\operatorname{Re}
\left[
e^{\mathrm{i}\theta/2}
\left(
\sum_{e\in\mathcal E_h^0}|\llbracket v\rrbracket_e|^2
+
k^2\|v\|_{0,\Omega}^2
\right)
\right.\left.
+
|\alpha| e^{-\mathrm{i}\theta/2}
\sum_{e\in\mathcal E_h^B}|\llbracket v\rrbracket_e|^2
\right]
\\
&=
\cos\left(\frac{\theta}{2}\right)
\left(
\sum_{e\in\mathcal E_h^0}|\llbracket v\rrbracket_e|^2
+
|\alpha|
\sum_{e\in\mathcal E_h^B}|\llbracket v\rrbracket_e|^2
+
k^2\|v\|_{0,\Omega}^2
\right)
\\
&=
\cos\left(\frac{\theta}{2}\right)
\|v\|_D^2.
\end{align*}
Set $c_\theta
:=
\cos\left(\frac{\theta}{2}\right).$ 
By the half-angle identity , we have $c_\theta
=
\left(
\frac{1+\cos\theta}{2}
\right)^{1/2}
\nonumber
=
\left[
\frac12
\left(
1+
\frac{k h}
{\sqrt{(k h)^2+4}}
\right)
\right]^{1/2}. $ 
Since \(0<\theta<\pi/2\), it follows that $\frac{1}{\sqrt2}
<
c_\theta
<
1.$ 
Consequently,
\begin{align}
\operatorname{Re}\widetilde b_h(v,v)
&=
c_\theta\|v\|_D^2
\geq
\frac{1}{\sqrt2}\|v\|_D^2,
\qquad
\forall v\in W_h.
\label{eq:rotated-b-coercivity}
\end{align}
Therefore, the rotated sesquilinear form
\(\widetilde b_h(\cdot,\cdot)\) is coercive on \(W_h\) with respect
to \(\|\cdot\|_D\). 
Moreover, since $|\widetilde b_h(v,v)|
=
|b_h(v,v)|,$ we deduce from \eqref{eq:rotated-b-coercivity} that
\begin{align}
|b_h(v,v)|
=
|\widetilde b_h(v,v)|
\geq
\operatorname{Re}\widetilde b_h(v,v)
\nonumber
\geq
\frac{1}{\sqrt2}\|v\|_D^2,
\qquad
\forall v\in W_h.
\end{align}
Hence \(b_h(\cdot,\cdot)\) is uniformly coercive  with
respect to the norm \(\|\cdot\|_D\).
\end{proof}




\subsection{Inf--Sup Condition for the Fine-Grid Discretization}
The goal of this subsection is to establish the well-posedness of the global fine-grid problem \eqref{bf} by proving an appropriate inf--sup condition.

For arbitrary $w_h\in W_h$, let $z=z(w_h)$ solve
\begin{subequations}
\label{eq:continuous-adjoint}
\begin{align}
-\Delta z-k^2z&=2k^2w_h &&\text{ in }\Omega,
\label{eq:continuous-adjoint-a}\\
\partial_nz+ik z&=0  &&\text{ on }\partial\Omega.
\label{eq:continuous-adjoint-b}
\end{align}
\end{subequations}
Because $k\|w_h\|_{0,\Omega}\leq\|w_h\|_D$, \cref{continue_stability} immediately yields
\begin{align}
\|z\|_{E,k}
&\leq
2\sqrt{3}\,C_{\rm stab}\,
k\|w_h\|_D,
\label{eq:z-energy-w}\\
|z|_{2,\Omega}
&\leq
2C_{\rm stab}k^2\|w_h\|_D.
\label{eq:z-H2-w}
\end{align} 
These two estimates characterize the wave-number dependence of the 
auxiliary solution and will be used in the consistency analysis below. 
To relate the continuous auxiliary  solution $z$ to the discrete space 
$W_h$, we introduce the piecewise-constant $L^2$-projection 
$P_h^0:L^2(\Omega)\rightarrow W_h$ defined by
\begin{align*}
v_t^{\rm av}
:= 
(P_h^0v)|_t
=
\frac{1}{|t|}
\int_t v\,dx,
\qquad \forall t\in\mathcal T_h.
\end{align*}
Since $P_h^0$ is the $L^2$-orthogonal projection onto $W_h$, for every $w_h\in W_h$,
\begin{align}
(w_h,v)_\Omega=(w_h,P_h^0v)_\Omega.
\label{eq:mass-exact}
\end{align}
For later use, we also denote the average of $v$ over an edge
$e\in\mathcal E_h$ by $v_e^{\rm av}
:=
\frac{1}{|e|}
\int_e v\,ds.$

To estimate the projection-induced consistency terms while keeping
their dependence on the discretization parameters explicit, we collect some technical results here for later use and their proofs can be found in Appendix~\ref{proof_B} --Appendix~\ref{proof_D}.
\begin{lemma}
\label{cell-edge-average-H1}
Let $t\in \mathcal{T}_h$ and let $e\subset\partial t$ be one of its edges,
with unit normal $\Bn_e$. For every $v\in H^1(t)$,
\begin{align*}
|v_t^{\rm av}-v_e^{\rm av}|
&\leq \frac{1}{\sqrt{3}}\|\partial_{n_e}v\|_{0,t}
\leq \frac{1}{\sqrt{3}}\|\nabla v\|_{0,t}.
\end{align*}
\end{lemma}


\begin{lemma}
\label{lem:flux}
For every $v\in H^2(\Omega)$, the following estimate holds:
\begin{align*}
    \bigl(
\sum_{e\in\mathcal{E}_h^0}
\left|
\llbracket P_h^0v\rrbracket_e
+\int_e\partial_{n_e}v\,ds
\right|^2
\bigr)^{1/2}
\le
2C_{\rm bh}
h|v|_{2,\Omega}.
\end{align*}
\end{lemma}

\begin{lemma}
\label{lem:boundary-average}
For every \(v\in H^2(\Omega)\), we have
\begin{align}
\sum_{e\in\mathcal{E}_h^B}
|v_e^{\rm av}-v_{t(e)}^{\rm av}|^2
&\le 
h\|\partial_n v\|_{0,\partial\Omega}^2
+h^2|v|_{2,\Omega}^2.
\label{eq:boundary-average}
\end{align}
If $v$ additionally satisfies the impedance condition \eqref{eq:continuous-adjoint-b}, then
\begin{align}
\sum_{e\in\mathcal{E}_h^B}
|v_e^{\rm av}-v_{t(e)}^{\rm av}|^2
&\le
 k^2 h\|v\|_{0,\partial\Omega}^2
+h^2|v|_{2,\Omega}^2.
\label{eq:boundary-average-imp}
\end{align}
\end{lemma}

Combining the preceding local estimates with the regularity properties of
functions in $H^2(\Omega)$, and using the fact that the solution $z(w_h)$
of \eqref{eq:continuous-adjoint} belongs to $H^2(\Omega)$, we obtain the
following stability estimate for $P_h^0 z$.

\begin{lemma}
\label{lem:P0-D}
Let $z\in H^2(\Omega)$ be the solution to problem \eqref{eq:continuous-adjoint}. Then
\begin{align*}
    \|P_h^0z\|_D\le C_{T,k h}k\|w_h\|_D,
\end{align*} where $C_{T,k h}=2\sqrt{3}C_{\rm stab}(k h+1)\sqrt{C_I+\frac43 k h+2}.$
\end{lemma}

\begin{proof}
We estimate separately the interior-edge, physical-boundary, and mass contributions to the $D$-norm. 

\emph{(i) Interior-edge.} By \cref{lem:flux} and the triangle inequality,
\begin{align*}
\left(\sum_{e\in\mathcal{E}_h^0}|\llbracket P_h^0z\rrbracket _e|^2\right)^{1/2}
&\le
2C_{\rm bh}h|z|_{2,\Omega}
+
\left(
\sum_{e\in\mathcal E_h^0}
\left|
\int_e\partial_{n_e}z\,ds
\right|^2
\right)^{1/2}.
\end{align*}
For each edge, Cauchy--Schwarz inequality gives $\left|\int_e\partial_{n_e}z\,ds\right|^2
\leq h\|\partial_{n_e}z\|_{0,e}^2
\leq h\|\nabla z\|_{0,e}^2.$ Applying  scaled trace inequality\cite[Lemma~1.49]{DiPietroErn2012} elementwise and observing that every interior edge is counted at most twice in the element-boundary sums yields $\sum_{e\in\mathcal{E}_h^0}\left|\int_e\partial_{n_e}z\,ds\right|^2
\le C_{\rm tr}(\|\nabla z\|_{0,\Omega}^2+h^2|z|_{2,\Omega}^2).$ Hence, for a constant $C_I$ depending only on $C_{\rm bh}$ and $C_{\rm tr}$,
\begin{align}
\sum_{e\in\mathcal{E}_h^0}|\llbracket P_h^0z\rrbracket_e|^2
&\leq C_I\left(\|\nabla z\|_{0,\Omega}^2+h^2|z|_{2,\Omega}^2\right).
\label{eq:D-int-bound}
\end{align}
\emph{(ii) Physical-boundary.} By Cauchy--Schwarz inequality, we have 
$$
|z_e^{\rm av}|^2
=\frac1{h^2}\left|\int_ez\,ds\right|^2
\leq \frac1{h^2}|e|\int_e|z|^2\,ds
=\frac1h\|z\|_{0,e}^2.
$$ 
Combining with \cref{cell-edge-average-H1}, for each boundary edge $e\subset\partial t$, 
\begin{align*}
|z_{t(e)}^{\rm av}|^2
\leq 2|z_e^{\rm av}|^2+2|z_{t(e)}^{\rm av}-z_e^{\rm av}|^2
\le \frac{2}{h}\|z\|_{0,e}^2+\frac{2}{3}\|\partial_{n_e}z\|_{0,t}^2.
\end{align*}
Each cell has at most two edges normal to either coordinate direction,
so summing over physical boundary edges gives
\begin{align}
|\alpha|\sum_{e\in\mathcal{E}_h^B}|z_{t(e)}^{\rm av}|^2
&\le \frac{2|\alpha|}{h}\|z\|_{0,\partial\Omega}^2
+\frac{2|\alpha|}{3}
\sum_{e\in\mathcal{E}_h^B}\|\partial_{n_e}z\|_{0,t(e)}^2\notag\\
&\le
\frac{4k}{\sqrt{(k h)^2+4}}\|z\|_{0,\partial\Omega}^2
+\frac{8k h}{3\sqrt{(k h)^2+4}}\|\nabla z\|_{0,\Omega}^2.
\label{physical-b}
\end{align}

\emph{(iii) Mass contribution.}
Since $P_h^0$ is the $L^2$-orthogonal projection, its contractivity gives
\begin{align}
k^2\|P_h^0z\|_{0,\Omega}^2
&\leq k^2\|z\|_{0,\Omega}^2.
\label{eq:D-mass-bound}
\end{align}

Combining \eqref{eq:D-int-bound}, \eqref{physical-b} and \eqref{eq:D-mass-bound}, we have $\|P_h^0z\|_D^2
\le (C_I+\frac43 k h+2) 
\|z\|_{E,k}^2+C_Ih^2|z|_{2,\Omega}^2.$  
Combining \eqref{eq:z-energy-w}, \eqref{eq:z-H2-w}   we obtain $\|P_h^0z\|_D
\le C_{T,k h}k\|w_h\|_D,$ where $C_{T,k h}=2\sqrt{3}C_{\rm stab}(k h+1)\sqrt{C_I+\frac43 k h+2}$.
\end{proof}

By the stability of $P_h^0z$ in the above lemma, we have the following crucial estimate.

\begin{lemma}
\label{lem:cell-identity}
Let $z$ solve \eqref{eq:continuous-adjoint}. 
Then
\begin{align*}
\left|a_h(w_h,P_h^0z)-2k^2\|w_h\|_{0,\Omega}^2\right|
\le C_{\rm con,k h}k^2h\|w_h\|_D^2,
\end{align*}
where $C_{\rm con,k h}
:=4C_{\rm stab}C_{\rm bh}
+\frac{C_{T,k h}}{2}
+2\sqrt{3}C_{\rm stab}(2+k h
). $ 

\end{lemma}

\begin{proof}
Take the complex conjugate of \eqref{eq:continuous-adjoint-a}, multiply it on each cell $t$ by the constant value $w_h|_t$, and integrate. This gives
\begin{align*}
2k^2\int_t|w_h|^2
&=-w_h|_t\int_{\partial t}\partial_{n_t}\overline z\,ds
-k^2w_h|_t\int_t\overline z\,dx.
\end{align*}
Summing over all elements, the contributions from the two cells adjacent to an interior edge combine as
\begin{align*}
-w_{t^-}\int_e\partial_{n_e}\overline z\,ds
+w_{t^+}\int_e\partial_{n_e}\overline z\,ds
&=-\llbracket w_h\rrbracket_e\overline{\int_e\partial_{n_e}z\,ds}.
\end{align*}
$\partial_nz+ik z=0$ implies $\partial_n\overline z-ik\overline z=0, $ 
and therefore $-w_h\int_e\partial_n\overline z\,ds
=-ik h\,\llbracket w_h\rrbracket_e\overline{z_e^{\rm av}}. $ 
Combining the interior and boundary contributions, we obtain
\begin{align}
2k^2\|w_h\|_{0,\Omega}^2
&=-\sum_{e\in\mathcal{E}_h^0}\llbracket w_h\rrbracket_e
\overline{\int_e\partial_{n_e}z\,ds}
-ik h\sum_{e\in\mathcal{E}_h^B}\llbracket w_h\rrbracket_e\overline{z_e^{\rm av}}
-k^2(w_h,z)_\Omega.
\label{eq:cell-identity}
\end{align}
By definition of $a_h(\cdot,\cdot)$, \eqref{eq:mass-exact} and \eqref{eq:cell-identity}, we have
\begin{align*}
 a_h(w_h,P_h^0z)-2k^2\|w_h\|_{0,\Omega}^2&=\sum_{e\in \mathcal{E}_h^0}\llbracket w_h \rrbracket_e \,\overline{\llbracket P_h^0z \rrbracket_e} 
+\alpha\sum_{e\in \mathcal{E}_h^{B}} \llbracket w_h \rrbracket_e \,\overline{\llbracket P_h^0z \rrbracket_e}-k^2( w_h,P_h^0z)_\Omega-2k^2\|w_h\|_{0,\Omega}^2\\
&=:R_1+R_2,
\end{align*}
where 
\begin{align*}
R_1
:=\sum_{e\in\mathcal{E}_h^0}\llbracket w_h\rrbracket_e
\overline{\left(\llbracket P_h^0z\rrbracket_e+\int_e\partial_{n_e}z\,ds\right)},\quad
R_2
:=\sum_{e\in\mathcal{E}_h^B}\llbracket w_h\rrbracket_e
\left(\alpha\overline{P_h^0z|_{t(e)}}+ik h\,\overline{z_e^{\rm av}}\right).
\end{align*}
By \cref{lem:flux} and \eqref{eq:z-H2-w},
\begin{align}
|R_1|
&\le 4C_{\rm stab}C_{\rm bh}k^2h\|w_h\|_D^2.
\label{eq:RI-bound}
\end{align}
For $R_2,$ we split it as
\begin{align*}
R_2
&=(\alpha+i k h)\sum_{e\in\mathcal{E}_h^B}\llbracket w_h\rrbracket_e\overline{P_h^0z|_{t(e)}}
+i k h\sum_{e\in\mathcal{E}_h^B}\llbracket w_h\rrbracket_e
\overline{(z_e^{\rm av}-P_h^0z|_{t(e)})}
\nonumber\\
&=:R_{2,1}+R_{2,2}.
\end{align*}
%
Using  \cref{lem:P0-D},
\begin{align}
|R_{2,1}|
&\leq \frac {k h}2
\left(|\alpha|\sum_{e\in\mathcal{E}_h^B}|\llbracket w_h\rrbracket_e|^2\right)^{1/2}
\left(|\alpha|\sum_{e\in\mathcal{E}_h^B}|P_h^0z|_{t(e)}|^2\right)^{1/2}
\nonumber\\
&\le \frac{C_{T,k h}k^2h}{2}\|w_h\|_D^2.
\label{eq:RB1-bound}
\end{align}
By \cref{lem:boundary-average}, \eqref{eq:z-energy-w} and \eqref{eq:z-H2-w}, we have
\begin{align*}
\sum_{e\in\mathcal{E}_h^B}|z_e^{\rm av}-P_h^0z|_{t(e)}|^2
&\le 12C_{\rm stab}^2(1+k h)k^3h
\|w_h\|_D^2.
\end{align*}
Using the Cauchy--Schwarz inequality gives
\begin{align}
|R_{2,2}|
&=|i k h\sum_{e\in\mathcal{E}_h^B}\llbracket w_h\rrbracket_e
\overline{(z_e^{\rm av}-P_h^0z|_{t(e)})}|\notag\\
&\leq \frac{k h}{\sqrt{|\alpha|}}\left(|\alpha|\sum_{e\in\mathcal{E}_h^B}|\llbracket w_h\rrbracket_e|^2\right)^{\frac{1}{2}}\left(\sum_{e\in\mathcal{E}_h^B}|z_e^{\rm av}-P_h^0z|_{t(e)}|^2\right)^{\frac{1}{2}}\notag\\
&\le \sqrt{6}C_{\rm stab}((4+(k h)^2)^\frac{1}{2}(1+k h))^{\frac{1}{2}} k^2h\|w_h\|_D^2\\
&\le  \sqrt{6}C_{\rm stab}(2+k h) k^2h\|w_h\|_D^2.
\label{eq:RB2-bound}
\end{align}
Combining \eqref{eq:RI-bound}, \eqref{eq:RB1-bound}, and \eqref{eq:RB2-bound}, we finish the proof.
\end{proof}
We can now establish the inf--sup stability of the fine-grid discretization. 

\begin{theorem}
\label{thm:main-infsup}
Assume that $C_{\rm con,k h}k^2h\leq \frac{\sqrt{2}}{4}.$ Then the sesquilinear
form $a_h(\cdot,\cdot)$ satisfies the discrete inf--sup condition
\begin{equation}
\beta_h
:=
\inf_{0\neq w_h\in W_h}
\sup_{0\neq v_h\in W_h}
\frac{|a_h(w_h,v_h)|}
{\|w_h\|_D\|v_h\|_D}
\ge
\frac{\sqrt{2}}{4(1+C_{T,k h}k)}.
\label{eq:fine-grid-infsup}
\end{equation}
\end{theorem}

\begin{proof}
Fix $0\neq w_h\in W_h$. Let $z$ be the solution of
\eqref{eq:continuous-adjoint} and set
$v_h^\star:=w_h+P_h^0z$. By \cref{lem:cell-identity},
\begin{align*}
a_h(w_h,P_h^0z)
=
2k^2\|w_h\|_{0,\Omega}^2+R_1+R_2,
\quad
|R_1+R_2|
\leq
C_{\rm con,k h}k^2h\|w_h\|_D^2.
\end{align*}
Hence, under the assumption  $C_{\rm con,k h}k^2h\le\frac{\sqrt{2}}{4},$ 
\begin{align}
|a_h(w_h,v_h^\star)|
&=
|b_h(w_h,w_h)+R_1+R_2|
\nonumber\\
&\geq
|b_h(w_h,w_h)|-|R_1+R_2|
\nonumber\\
&\geq
\left(\frac{1}{\sqrt{2}}-C_{\rm con,k h}k^2h\right)
\|w_h\|_D^2
\geq
\frac{\sqrt{2}}{4}\|w_h\|_D^2.
\label{eq:numerator}
\end{align}
On the other hand, \cref{lem:P0-D} gives
\begin{align}
\|v_h^\star\|_D
&\leq
\|w_h\|_D+\|P_h^0z\|_D
\le
(1+C_{T,k h}k)\|w_h\|_D.
\label{eq:denominator}
\end{align}
Combining \eqref{eq:numerator} and \eqref{eq:denominator}, we obtain
\begin{align}
\sup_{0\neq v_h\in W_h}
\frac{|a_h(w_h,v_h)|}{\|v_h\|_D}
&\geq
\frac{|a_h(w_h,v_h^\star)|}{\|v_h^\star\|_D}
\geq
\frac{\sqrt{2}}{4(1+C_{T,k h}k)}\|w_h\|_D.
\end{align}
Dividing by $\|w_h\|_D$ and taking the infimum over all
$0\neq w_h\in W_h$ yields \eqref{eq:fine-grid-infsup}, completing the proof.
\end{proof}

\begin{remark}
The assumption in \cref{thm:main-infsup} takes the form $k^2 h \le C,$ where $C>0$ is independent of $k$ and $h$. Such a mesh-resolution
condition is consistent with classical stability and quasi-optimality
analyses for low-order Helmholtz finite element discretizations
\cite{IhlenburgBabuska1995,Sauter2006}. Refined pre-asymptotic error
estimates can be found in \cite{Wu2014}. 
For the present $RT_0$--$P_0$ scheme, this condition follows directly
from the discrete inf--sup argument.
\end{remark}

As a direct consequence of the fine-grid inf--sup condition, we obtain
the following stability estimate for the discrete solution.

\begin{corollary}
\label[corollary]{fine-grid-stability}
Under the assumptions of \cref{thm:main-infsup}, let
$u_h\in W_h$ be the solution of \eqref{bf}. Then
\begin{align*}
    \|u_h\|_D
\leq
2\sqrt{2}\left[(C_{T,k h}+k^{-1})\|f\|_{0,\Omega}
+
(C_{T,k h}k^{1/2}+k^{-\frac{1}{2}})\|g\|_{0,\partial\Omega}\right].
\end{align*}
\end{corollary}

\begin{proof} The well-posedness of the discrete problem follows from Lemma \ref{a_and_b} and Theorem \ref{thm:main-infsup} directly and we focus on the stability estimate in the following.

Using \cref{thm:main-infsup} and \eqref{bf}, we have
\begin{align}
\frac{\sqrt{2}}{4(1+C_{T,k h}k)}\|u_h\|_D
\leq
\sup_{0\neq p_h\in W_h}
\frac{|a_h(u_h,p_h)|}{\|p_h\|_D}
\leq\sup_{0\neq p_h\in W_h}
\frac{|(f,p_h)_\Omega
+
\frac{2i}{k h+2i}
(g,p_h)_{\partial\Omega}|}{\|p_h\|_D}.
\label{eq:uh-stability-1}
\end{align}
Recall that $\left|\frac{2i}{k h+2i}\right|
=
\frac{|\alpha|}{k h},
$ and $
|\alpha|\leq k h.$
Since $p_h$ is piecewise constant and
$k\|p_h\|_{0,\Omega}\leq\|p_h\|_D$, we have $\|p_h\|_{0,\partial\Omega}^2
=
h\sum_{e\in\mathcal E_h^B}
|\llbracket p_h\rrbracket_e|^2
\leq
\frac{h}{|\alpha|}\|p_h\|_D^2.$ Therefore, we have
\begin{align}
|(f,p_h)_\Omega
+
\frac{2i}{k h+2i}
(g,p_h)_{\partial\Omega}|
&\leq
\|f\|_{0,\Omega}\|p_h\|_{0,\Omega}
+
\frac{|\alpha|}{k h}
\|g\|_{0,\partial\Omega}
\|p_h\|_{0,\partial\Omega}
\nonumber\\
&\leq
\left(
k^{-1}\|f\|_{0,\Omega}
+
\frac{\sqrt{|\alpha|}}{k\sqrt h}
\|g\|_{0,\partial\Omega}
\right)\|p_h\|_D
\nonumber\\
&\leq
\left(
k^{-1}\|f\|_{0,\Omega}
+
k^{-1/2}\|g\|_{0,\partial\Omega}
\right)\|p_h\|_D.
\label{eq:uh-rhs-estimate}
\end{align}
Substituting \eqref{eq:uh-rhs-estimate} into
\eqref{eq:uh-stability-1} yields
\begin{align*}
\|u_h\|_D
\leq
2\sqrt{2}\left[(C_{T,k h}+k^{-1})\|f\|_{0,\Omega}
+
(C_{T,k h}k^{1/2}+k^{-\frac{1}{2}})\|g\|_{0,\partial\Omega}\right].
\end{align*}
\end{proof}


\subsection{Inf--Sup Condition for the Multiscale Discretization}

We now turn to the corresponding stability analysis of the multiscale
discretization. The key step is to approximate the fine-grid adjoint
solution in the multiscale space.

For each $w_{\rm ms}\in W_{\rm ms}$, consider the fine-grid adjoint
problem: find $\eta_h=\eta_h(w_{\rm ms})\in W_h$ such that
\begin{equation}
a_h(p_h,\eta_h)
=
2k^2(p_h,w_{\rm ms})_\Omega,
\qquad
\forall p_h\in W_h.
\label{eq:ideal-adjoint}
\end{equation}
The fine-grid inf--sup condition guarantees the unique solvability of
\eqref{eq:ideal-adjoint}. Let $\Br_h$ denote the flux associated with
$\eta_h$ through the fine-grid mixed formulation.

To localize the adjoint solution, we first identify the discrete trace
induced by restricting the global mixed equations to an oversampling
region. For an artificial boundary edge
$e\subset\partial T_i^m\setminus\partial\Omega$, let
$t_e^-,t_e^+\in\mathcal T_h$ denote the two fine-grid elements adjacent
to $e$, with $t_e^-\subset T_i^m$ and
$t_e^+\subset\Omega\setminus T_i^m$, and define
\[
\{\eta_h\}_e
:=
\frac12
\left(
\eta_h|_{t_e^-}+\eta_h|_{t_e^+}
\right).
\]
Throughout the local problems, $\Bn_i$ denotes the outward unit normal
of $T_i^m$. By the interior-edge flux relation,
\begin{equation}
\{\eta_h\}_e
=
\eta_h|_{t_e^-}
-\frac h2(\Br_h\cdot\Bn_i)|_e.
\label{eq:numerical-interface-traces}
\end{equation}
Thus, $\{\eta_h\}_e$ is the numerical trace inherited from the global
mixed discretization and should not be confused with the one-sided
boundary value of a function defined only on $T_i^m$.

With this convention, the restriction
$(\Br_h,\eta_h)|_{T_i^m}$ satisfies
\begin{subequations}
\label{eq:adjoint-local-restriction}
\begin{align}
(\Bv_h,\Br_h)_{Q,T_i^m}
-(\nabla\cdot\Bv_h,\eta_h)_{T_i^m}
&+
\frac{i}{k}
\left\langle
\Bv_h\cdot\Bn_i,\Br_h\cdot\Bn_i
\right\rangle_{\partial T_i^m\cap\partial\Omega}
\nonumber\\
&+
\left\langle
\Bv_h\cdot\Bn_i,\{\eta_h\}
\right\rangle_{\partial T_i^m\setminus\partial\Omega}
=0,
\label{eq:adjoint-local-restriction-a}
\\
(p_h,\nabla\cdot\Br_h)_{T_i^m}
-(p_h,k^2\eta_h)_{T_i^m}
&=
(p_h,2k^2w_{\rm ms})_{T_i^m},
\label{eq:adjoint-local-restriction-b}
\end{align}
\end{subequations}
for all
$(\Bv_h,p_h)\in\BV_h(T_i^m)\times W_h(T_i^m)$.
Indeed, the second equation follows from the associated fine-grid
mixed adjoint system element by element. The relations on interior
and physical boundary edges in the first equation are inherited
directly from the global mixed system. On an artificial boundary
edge, setting $r_e=h(\Br_h\cdot\Bn_i)|_e$, the corresponding local
edge equation becomes
\[
\frac12\overline{r_e}
-\overline{\eta_h|_{t_e^-}}
+\overline{\{\eta_h\}_e}
=0
\]
by \eqref{eq:numerical-interface-traces}. Hence
\eqref{eq:adjoint-local-restriction} is precisely the local discrete
problem satisfied by the restriction of the global adjoint solution.

The following scale-resolution condition will be used for the local
stability estimates and the subsequent multiscale inf--sup analysis.

\begin{assumption}[Scale Resolution (SR) Assumption]
\label{assumption}
For the local estimates, $\mathcal T_H$ is a uniform square
coarsening of the Cartesian fine grid, with coarse side length $H$. 
We assume that the coarse mesh size $H$ is sufficiently small such that
\[
k(H+2mh)<1.
\]
\end{assumption}

We now decompose the restriction of the global adjoint solution into
a source-driven particular component and a locally homogeneous
component carrying the inherited boundary data. For each
$T_i\in\mathcal T_H$, let
$(\Br_{h,T_i^m}^{\mathrm I},
\eta_{h,T_i^m}^{\mathrm I})
\in\BV_h(T_i^m)\times W_h(T_i^m)$
be the solution of
\begin{subequations}
\label{adjoint-decom-I}
\begin{align}
(\Bv_h,\Br_{h,T_i^m}^{\mathrm I})_{Q,T_i^m}
-(\nabla\cdot\Bv_h,\eta_{h,T_i^m}^{\mathrm I})_{T_i^m}
&=0,
\label{eq:adjoint-decom-I-a}
\\
(p_h,\nabla\cdot\Br_{h,T_i^m}^{\mathrm I})_{T_i^m}
-(p_h,k^2\eta_{h,T_i^m}^{\mathrm I})_{T_i^m}
&=
(p_h,2k^2w_{\rm ms})_{T_i^m},
\label{eq:adjoint-decom-I-b}
\end{align}
\end{subequations}
for all
$(\Bv_h,p_h)\in\BV_h(T_i^m)\times W_h(T_i^m)$.
This component accounts solely for the local response to the volume
source $2k^2w_{\rm ms}$, with homogeneous Dirichlet data imposed
through the mixed formulation. Its well-posedness is established in
\Cref{first_comp} below.

We define the second component as the remainder
\[
(\Br_{h,T_i^m}^{\mathrm{II}},
\eta_{h,T_i^m}^{\mathrm{II}})
:=
(\Br_h,\eta_h)|_{T_i^m}
-
(\Br_{h,T_i^m}^{\mathrm I},
\eta_{h,T_i^m}^{\mathrm I}).
\]
Subtracting \eqref{adjoint-decom-I} from
\eqref{eq:adjoint-local-restriction} yields
\begin{subequations}
\label{adjoint-decom-II}
\begin{align}
(\Bv_h,\Br_{h,T_i^m}^{\mathrm{II}})_{Q,T_i^m}
-(\nabla\cdot\Bv_h,\eta_{h,T_i^m}^{\mathrm{II}})_{T_i^m}
&+
\frac{i}{k}
\left\langle
\Bv_h\cdot\Bn_i,\Br_h\cdot\Bn_i
\right\rangle_{\partial T_i^m\cap\partial\Omega}
\nonumber\\
&+
\left\langle
\Bv_h\cdot\Bn_i,\{\eta_h\}
\right\rangle_{\partial T_i^m\setminus\partial\Omega}
=0,
\label{eq:adjoint-decom-II-a}
\\
(p_h,\nabla\cdot\Br_{h,T_i^m}^{\mathrm{II}})_{T_i^m}
-(p_h,k^2\eta_{h,T_i^m}^{\mathrm{II}})_{T_i^m}
&=0,
\label{eq:adjoint-decom-II-b}
\end{align}
\end{subequations}
for all
$(\Bv_h,p_h)\in\BV_h(T_i^m)\times W_h(T_i^m)$.
Hence the second component is locally homogeneous in the entire
pressure space $W_h(T_i^m)$ and contains only the boundary information
inherited from the global adjoint solution. It is this component that
will subsequently be related to the local snapshot space.

Restricting the two local components to the target coarse element, we
define, for $s\in\{\mathrm I,\mathrm{II}\}$,
\begin{equation}
\Br_h^s|_{T_i}
:=
\Br_{h,T_i^m}^s|_{T_i},
\qquad
\eta_h^s|_{T_i}
:=
\eta_{h,T_i^m}^s|_{T_i}.
\label{decom1.2}
\end{equation}
Consequently,
\begin{equation}
\Br_h
=
\Br_h^{\mathrm I}+\Br_h^{\mathrm{II}},
\qquad
\eta_h
=
\eta_h^{\mathrm I}+\eta_h^{\mathrm{II}}.
\label{adjoint-global-decomposition}
\end{equation}

We first estimate the source-driven component.

\begin{lemma}
\label{first_comp}
Let Assumption~\ref{assumption} hold. Then
\eqref{adjoint-decom-I} is uniquely solvable and
\begin{equation}
\left(
\|\Br_{h,T_i^m}^{\mathrm I}\|_{Q,T_i^m}^{2}
+
k^2\|\eta_{h,T_i^m}^{\mathrm I}\|_{0,T_i^m}^{2}
\right)^{1/2}
\le
\frac{2\sqrt5}{3}\,
k^2(H+2mh)
\|w_{\rm ms}\|_{0,T_i^m}.
\label{eq:first-comp-energy}
\end{equation}
Consequently,
\begin{equation}
\|\eta_{h,T_i^m}^{\mathrm I}\|_{D,T_i^m}
\le
\frac{2\sqrt5}{3}\,
k^2(H+2mh)
\|w_{\rm ms}\|_{0,T_i^m}.
\label{eq:first-comp-D}
\end{equation}
\end{lemma}

\begin{proof}
We write $\Br_{h,T_i^m}^{\mathrm I}
=
\sum_{e\in\mathcal E_{T_i^m}}
r_{h,T_i^m,e}^{\mathrm I}\Bv_e,$ where $r_{h,T_i^m,e}^{\mathrm I}$ denotes the coefficient associated
with the edge $e$. Testing
\eqref{eq:adjoint-decom-I-a} with $\Bv_e$ and conjugating the resulting
scalar identity gives
\[
r_{h,T_i^m,e}^{\mathrm I}
=
\begin{cases}
\llbracket\eta_{h,T_i^m}^{\mathrm I}\rrbracket_e,
&
e\in\mathcal E_{T_i^m}^{0},
\\[1mm]
2\,\eta_{h,T_i^m}^{\mathrm I}|_{t_e},
&
e\in\mathcal E_{T_i^m}^{B},
\end{cases}
\]
where, for $e\in\mathcal E_{T_i^m}^{B}$,
$t_e\in\mathcal T_h$ is the unique fine-grid element contained in
$T_i^m$ such that $e\subset\partial t_e$.

Using the local trapezoidal quadrature rule, an interior edge receives
two contributions of weight $1/2$, whereas a boundary edge receives
only one. Therefore,
\begin{align}
\|\Br_{h,T_i^m}^{\mathrm I}\|_{Q,T_i^m}^{2}
&=
\sum_{e\in\mathcal E_{T_i^m}^{0}}
\left|
\llbracket\eta_{h,T_i^m}^{\mathrm I}\rrbracket_e
\right|^{2}
+
2
\sum_{e\in\mathcal E_{T_i^m}^{B}}
\left|
\eta_{h,T_i^m}^{\mathrm I}|_{t_e}
\right|^{2}.
\label{q_u_condition}
\end{align}

We next derive an estimate for the $L^2$-norm. Let
$v_h\in W_h(T_i^m)$ and denote its values on the fine cells by
$v_{j,r}$, where $1\le j\le n_x$ and $1\le r\le n_y$.
Since both side lengths of $T_i^m$ are bounded by $H+2mh$,
$n_xh,n_yh\le H+2mh$. Along each row,
\[
2v_{j,r}
=
v_{1,r}+v_{n_x,r}
+
\sum_{\ell=1}^{j-1}
(v_{\ell+1,r}-v_{\ell,r})
-
\sum_{\ell=j}^{n_x-1}
(v_{\ell+1,r}-v_{\ell,r}),
\]
and hence the Cauchy--Schwarz inequality gives
\[
|v_{j,r}|^2
\le
|v_{1,r}|^2+|v_{n_x,r}|^2
+
\frac{n_x-1}{2}
\sum_{\ell=1}^{n_x-1}
|v_{\ell+1,r}-v_{\ell,r}|^2.
\]
Multiplying by $h^2$, summing over all cells, repeating the argument
in the vertical direction, and averaging the two estimates yields
\begin{align}
\|v_h\|_{0,T_i^m}^{2}
&\le
\frac{(H+2mh)^2}{4}
\sum_{e\in\mathcal E_{T_i^m}^{0}}
|\llbracket v_h\rrbracket_e|^2+
\frac{h(H+2mh)}{2}
\sum_{e\in\mathcal E_{T_i^m}^{B}}
|v_h|_{t_e}|^2
\notag\\
&\le
\frac{(H+2mh)^2}{4}
\left[
\sum_{e\in\mathcal E_{T_i^m}^{0}}
|\llbracket v_h\rrbracket_e|^2
+
2\sum_{e\in\mathcal E_{T_i^m}^{B}}
|v_h|_{t_e}|^2
\right].
\label{eq:adjoint-dirichlet-poincare}
\end{align}
Applying \eqref{eq:adjoint-dirichlet-poincare} to
$\eta_{h,T_i^m}^{\mathrm I}$ and using
\eqref{q_u_condition}, we obtain
\begin{equation}
\|\eta_{h,T_i^m}^{\mathrm I}\|_{0,T_i^m}
\le
\frac{H+2mh}{2}
\|\Br_{h,T_i^m}^{\mathrm I}\|_{Q,T_i^m}.
\label{eq:first-comp-L2-flux}
\end{equation}
Taking
$\Bv_h=\Br_{h,T_i^m}^{\mathrm I}$ and
$p_h=\eta_{h,T_i^m}^{\mathrm I}$ in
\eqref{adjoint-decom-I}, and taking real parts, gives
\begin{align}
\|\Br_{h,T_i^m}^{\mathrm I}\|_{Q,T_i^m}^{2}
-
k^2\|\eta_{h,T_i^m}^{\mathrm I}\|_{0,T_i^m}^{2}
&=
\operatorname{Re}
\bigl(
\eta_{h,T_i^m}^{\mathrm I},
2k^2w_{\rm ms}
\bigr)_{T_i^m}
\notag\\
&\le
2k^2
\|w_{\rm ms}\|_{0,T_i^m}
\|\eta_{h,T_i^m}^{\mathrm I}\|_{0,T_i^m}.
\label{q_u}
\end{align}
Combining \eqref{eq:first-comp-L2-flux} and \eqref{q_u} gives
\[
\left(
1-\frac{k^2(H+2mh)^2}{4}
\right)
\|\Br_{h,T_i^m}^{\mathrm I}\|_{Q,T_i^m}^{2}
\le
k^2(H+2mh)
\|w_{\rm ms}\|_{0,T_i^m}
\|\Br_{h,T_i^m}^{\mathrm I}\|_{Q,T_i^m}.
\]
Assumption~\ref{assumption} implies $k(H+2mh)<1$; hence
$1-k^2(H+2mh)^2/4\ge3/4$, and therefore
\begin{align*}
\|\Br_{h,T_i^m}^{\mathrm I}\|_{Q,T_i^m}
\le
\frac43
k^2(H+2mh)
\|w_{\rm ms}\|_{0,T_i^m}.
\end{align*}
Furthermore, \eqref{eq:first-comp-L2-flux} gives
\[
k\|\eta_{h,T_i^m}^{\mathrm I}\|_{0,T_i^m}
\le
\frac{k(H+2mh)}{2}
\|\Br_{h,T_i^m}^{\mathrm I}\|_{Q,T_i^m}
\le
\frac12
\|\Br_{h,T_i^m}^{\mathrm I}\|_{Q,T_i^m}.
\]
Consequently,
\begin{align*}
\left(
\|\Br_{h,T_i^m}^{\mathrm I}\|_{Q,T_i^m}^{2}
+
k^2\|\eta_{h,T_i^m}^{\mathrm I}\|_{0,T_i^m}^{2}
\right)^{1/2}
&\le
\frac{\sqrt5}{2}
\|\Br_{h,T_i^m}^{\mathrm I}\|_{Q,T_i^m}
\\
&\le
\frac{2\sqrt5}{3}
k^2(H+2mh)
\|w_{\rm ms}\|_{0,T_i^m}.
\end{align*}
Finally, using the definition of the local $D$-norm,
\[
\|\eta_{h,T_i^m}^{\mathrm I}\|_{D,T_i^m}^{2}
=
\sum_{e\in\mathcal E_{T_i^m}^{0}}
\left|
\llbracket\eta_{h,T_i^m}^{\mathrm I}\rrbracket_e
\right|^2
+
k^2
\|\eta_{h,T_i^m}^{\mathrm I}\|_{0,T_i^m}^{2},
\]
and \eqref{q_u_condition}, we have
\[
\|\eta_{h,T_i^m}^{\mathrm I}\|_{D,T_i^m}^{2}
\le
\|\Br_{h,T_i^m}^{\mathrm I}\|_{Q,T_i^m}^{2}
+
k^2
\|\eta_{h,T_i^m}^{\mathrm I}\|_{0,T_i^m}^{2},
\]
which proves \eqref{eq:first-comp-D}.

For a zero right-hand side, the same argument gives
$\Br_{h,T_i^m}^{\mathrm I}=0$ and
$\eta_{h,T_i^m}^{\mathrm I}=0$.
Since \eqref{adjoint-decom-I} is a finite-dimensional square system,
this injectivity proves its unique solvability.
\end{proof}

We next consider the second component
$(\Br_{h,T_i^m}^{\mathrm{II}},
\eta_{h,T_i^m}^{\mathrm{II}})$.
The following lemma shows that this component can be represented by
functions in the local snapshot space.

\begin{lemma}
\label{thm:adjoint-II-snapshot-representation}
For each $T_i\in\mathcal T_H$, let
$(\Br_{h,T_i^m}^{\mathrm{II}},
\eta_{h,T_i^m}^{\mathrm{II}})
\in \BV_h(T_i^m)\times W_h(T_i^m)$
be the solution to \eqref{adjoint-decom-II}.
Then there exist coefficients $c_{i,j}\in\mathbb C$,
$j=1,\ldots,J_i$, such that
\begin{align*}
    \Br_{h,T_i^m}^{\mathrm{II}}
=
\sum_{j=1}^{J_i} c_{i,j}\boldsymbol\Phi_i^j,
\qquad
\eta_{h,T_i^m}^{\mathrm{II}}
=
\sum_{j=1}^{J_i} c_{i,j}\psi_i^j.
\end{align*}
In particular, $\eta_{h,T_i^m}^{\mathrm{II}}
\in W_{i,m}^{\rm snap}.$
\end{lemma}

\begin{proof}
By \eqref{adjoint-decom-II},  $(\Br_{h,T_i^m}^{\mathrm{II}},
\eta_{h,T_i^m}^{\mathrm{II}})$ satisfies
\begin{subequations}
\label{eq:adjII-original-system}
\begin{align}
(\Br_{h,T_i^m}^{\mathrm{II}},\Bv_h)_{Q,T_i^m}
-(\eta_{h,T_i^m}^{\mathrm{II}},
\nabla\cdot\Bv_h)_{T_i^m}
-\frac{i}{k}
\left\langle
\Br_h\cdot\Bn_i,\Bv_h\cdot\Bn_i
\right\rangle_{\partial T_i^m\cap\partial\Omega}
+\left\langle
\{\eta_h\},\Bv_h\cdot\Bn_i
\right\rangle_{\partial T_i^m\setminus\partial\Omega}
&=0,
\label{eq:adjII-original-system-a}
\\
(\nabla\cdot\Br_{h,T_i^m}^{\mathrm{II}},p_h)_{T_i^m}
-(k^2\eta_{h,T_i^m}^{\mathrm{II}},p_h)_{T_i^m}
&=0,
\label{eq:adjII-original-system-b}
\end{align}
\end{subequations}
for all $(\Bv_h,p_h)\in \BV_h(T_i^m)\times W_h(T_i^m)$.
Define
\begin{align*}
\gamma_{\Br,\eta}=
\begin{cases}
-\Br_h\cdot\Bn_i,
& \text{on }\partial T_i^m\cap\partial\Omega,
\\[0.6em]
-ik\{\eta_h\},
& \text{on }\partial T_i^m\setminus\partial\Omega.
\end{cases}
\end{align*}
Then \eqref{eq:adjII-original-system} can be rewritten as
\begin{subequations}
\label{eq:adjII-xi-system}
\begin{align}
(\Br_{h,T_i^m}^{\mathrm{II}},\Bv_h)_{Q,T_i^m}
-(\eta_{h,T_i^m}^{\mathrm{II}},
\nabla\cdot\Bv_h)_{T_i^m}
+\frac{i}{k}
\left\langle
\gamma_{\Br,\eta},\Bv_h\cdot\Bn_i
\right\rangle_{\partial T_i^m}
&=0,
\label{eq:adjII-xi-system-a}
\\
(\nabla\cdot\Br_{h,T_i^m}^{\mathrm{II}},p_h)_{T_i^m}
-(k^2\eta_{h,T_i^m}^{\mathrm{II}},p_h)_{T_i^m}
&=0.
\label{eq:adjII-xi-system-b}
\end{align}
\end{subequations}

We now convert \eqref{eq:adjII-xi-system} into an equivalent
full-boundary impedance problem. Define $g_i
:=
\gamma_{\Br,\eta}-\Br_{h,T_i^m}^{\mathrm{II}}\cdot\Bn_i
\text{ on }\partial T_i^m.$ Let
$(\widetilde{\Br}_i^{\mathrm{II}},
\widetilde{\eta}_i^{\mathrm{II}})
\in\BV_h(T_i^m)\times W_h(T_i^m)$
be the solution of
\begin{subequations}
\label{eq:adjII-aux-impedance}
\begin{align}
(\widetilde{\Br}_i^{\mathrm{II}},\Bv_h)_{Q,T_i^m}
-(\widetilde{\eta}_i^{\mathrm{II}},
\nabla\cdot\Bv_h)_{T_i^m}
+\frac{i}{k}
\left\langle
\widetilde{\Br}_i^{\mathrm{II}}\cdot\Bn_i+g_i,
\Bv_h\cdot\Bn_i
\right\rangle_{\partial T_i^m}
&=0,
\label{eq:adjII-aux-impedance-a}
\\
(\nabla\cdot\widetilde{\Br}_i^{\mathrm{II}},p_h)_{T_i^m}
-(k^2\widetilde{\eta}_i^{\mathrm{II}},p_h)_{T_i^m}
&=0.
\label{eq:adjII-aux-impedance-b}
\end{align}
\end{subequations}
Set $\Be_r
:=
\widetilde{\Br}_i^{\mathrm{II}}
-\Br_{h,T_i^m}^{\mathrm{II}},
\quad
e_\eta
:=
\widetilde{\eta}_i^{\mathrm{II}}
-\eta_{h,T_i^m}^{\mathrm{II}}.$ Subtracting \eqref{eq:adjII-xi-system} from
\eqref{eq:adjII-aux-impedance} therefore gives
\begin{align*}
(\Be_r,\Bv_h)_{Q,T_i^m}
-(e_\eta,\nabla\cdot\Bv_h)_{T_i^m}
+\frac{i}{k}
\left\langle
\Be_r\cdot\Bn_i,\Bv_h\cdot\Bn_i
\right\rangle_{\partial T_i^m}
&=0,
\\
(\nabla\cdot\Be_r,p_h)_{T_i^m}
-(k^2e_\eta,p_h)_{T_i^m}
&=0.
\end{align*}
The uniqueness of the homogeneous local impedance problem yields
$\Be_r=0$ and $e_\eta=0$. Hence
\begin{align}
    \label{equivalent-solutions}
    \widetilde{\Br}_i^{\mathrm{II}}
=
\Br_{h,T_i^m}^{\mathrm{II}},
\quad
\widetilde{\eta}_i^{\mathrm{II}}
=
\eta_{h,T_i^m}^{\mathrm{II}}.
\end{align}
We next decompose $g_i$ as 
\begin{equation}
g_i=\sum_{j=1}^{J_i}g_{i,j},
\qquad
\operatorname{supp}g_{i,j}\subset e_i^j.
\label{eq:adjII-edge-decomposition}
\end{equation}
For each $j$, let
$(\Br_{i,j},\eta_{i,j})
\in\BV_h(T_i^m)\times W_h(T_i^m)$
be the unique solution of
\begin{subequations}
\label{eq:adjII-edge-impedance}
\begin{align}
(\Br_{i,j},\Bv_h)_{Q,T_i^m}
-(\eta_{i,j},\nabla\cdot\Bv_h)_{T_i^m}
+\frac{i}{k}
\left\langle
\Br_{i,j}\cdot\Bn_i+g_{i,j},
\Bv_h\cdot\Bn_i
\right\rangle_{\partial T_i^m}
&=0,
\label{eq:adjII-edge-impedance-a}
\\
(\nabla\cdot\Br_{i,j},p_h)_{T_i^m}
-(k^2\eta_{i,j},p_h)_{T_i^m}
&=0.
\label{eq:adjII-edge-impedance-b}
\end{align}
\end{subequations}
Summing \eqref{eq:adjII-edge-impedance} over
$j=1,\ldots,J_i$ and using \eqref{eq:adjII-edge-decomposition}, 
$\left(
\sum_{j=1}^{J_i}\Br_{i,j},
\sum_{j=1}^{J_i}\eta_{i,j}
\right)$ 
satisfies \eqref{eq:adjII-aux-impedance}.
By uniqueness and \eqref{equivalent-solutions}, we have 
\begin{align}
\label{deco}
    \Br_{h,T_i^m}^{\mathrm{II}}
=
\sum_{j=1}^{J_i}\Br_{i,j},
\quad
\eta_{h,T_i^m}^{\mathrm{II}}
=
\sum_{j=1}^{J_i}\eta_{i,j}.
\end{align}
For each $j$, define $c_{i,j}
:=
\left.
\frac{i}{k}
\left(
\Br_{i,j}\cdot\Bn_i+g_{i,j}
\right)
\right|_{e_i^j}. $ Since $\Br_{i,j}\cdot\Bn_i$ and $g_{i,j}$ are constant on
$e_i^j$, $c_{i,j}$ is a scalar. Moreover,
$g_{i,j}=0$ on $\partial T_i^m\setminus e_i^j$, and hence
\eqref{eq:adjII-edge-impedance-a} becomes
\begin{align*}
(\Br_{i,j},\Bv_h)_{Q,T_i^m}
-(\eta_{i,j},\nabla\cdot\Bv_h)_{T_i^m}
+\left\langle
c_{i,j},\Bv_h\cdot\Bn_i
\right\rangle_{e_i^j}
+\frac{i}{k}
\left\langle
\Br_{i,j}\cdot\Bn_i,
\Bv_h\cdot\Bn_i
\right\rangle_{\partial T_i^m\setminus e_i^j}
=0.
\end{align*}
Together with \eqref{eq:adjII-edge-impedance-b},
this is precisely the discrete snapshot problem with boundary value
$c_{i,j}$ on $e_i^j$. Since
$(\boldsymbol\Phi_i^j,\psi_i^j)$ denotes the corresponding snapshot
pair with unit boundary value on $e_i^j$, the linearity and uniqueness
of the local snapshot problem imply
\begin{equation}
\Br_{i,j}
=
c_{i,j}\boldsymbol\Phi_i^j,
\qquad
\eta_{i,j}
=
c_{i,j}\psi_i^j.
\label{eq:adjII-edge-snapshot-representation}
\end{equation}
Combining \eqref{deco} and
\eqref{eq:adjII-edge-snapshot-representation} gives $\eta_{h,T_i^m}^{\mathrm{II}}
=
\sum_{j=1}^{J_i}c_{i,j}\psi_i^j.$
\end{proof}
Combining the estimate for the source-driven component
$\eta_h^{\mathrm{I}}$ with the snapshot-space representation of
$\eta_h^{\mathrm{II}}$ established above, and applying the local spectral
approximation on each oversampling region, we obtain the following global
approximation property for the fine-grid adjoint solution. The proof is given in Appendix~\ref{proof_E}.
\begin{lemma}
\label{eta_h-eta_ms}
Under Assumption~\ref{assumption}, let $\eta_h\in W_h$ be the solution of \eqref{eq:ideal-adjoint}. 
Then there exists $\eta_{\rm ms}\in W_{\rm ms}$ such that
\begin{align*}
    \|\eta_h-\eta_{\rm ms}\|_D
\le \theta_{\rm ms}\|w_{\rm ms}\|_D,
\end{align*}
where
\begin{align*}
\theta_{\rm ms}
:={}&
\frac{2}{3}
\sqrt{10N_{\rm ov}(m+1)}
\,k(H+2mh)
\\
&+
\left[
\frac{
2C_{R,m}N_{\rm ov}H^2
}{
\Lambda_{1,m}
}
\left(
C_{Q,kh}(k^2h+k)
+
\frac{20}{9}
C_{\alpha,kh}
k^2(H+2mh)^2h
\right)
\right]^{1/2}.
\end{align*}
Here $C_{R,m}$ is the restriction constant in
\eqref{eq:restriction-constant}, and $N_{\rm ov}$ bounds cell and edge
overlaps, including artificial-boundary cell incidences.
The remaining constants are
\begin{align*}
\Lambda_{1,m}
:=
\min_{1\le i\le N_T}\lambda_{l_i+1,1},
\quad
C_{Q,k h}
:=
32\bigl(C_{T,k h}+k^{-1}\bigr)^2, \text{ and }C_{\alpha,kh}
:=
\max\{1,kh\}.
\end{align*}
\end{lemma}

Combining the above adjoint approximation property with the fine-grid stability yields the following inf--sup condition on the multiscale space.
\begin{theorem}
\label{thm:ms-infsup}
Under the hypotheses of \cref{thm:main-infsup} and \cref{assumption}, assume that
$\theta_{\rm ms}\le\frac{\sqrt{2}}{4}$.
Then the sesquilinear form $a_h(\cdot,\cdot)$ satisfies the inf--sup condition as follows:
\begin{align*}
\beta_{\rm ms}
:=
\inf_{0\neq w_{\rm ms}\in W_{\rm ms}}
\sup_{0\neq v_{\rm ms}\in W_{\rm ms}}
\frac{
|a_h(w_{\rm ms},v_{\rm ms})|
}{
\|w_{\rm ms}\|_D\|v_{\rm ms}\|_D
}\ge \frac{\sqrt{2}}{4(7(1+C_{T,k h}k)+\theta_{\rm ms})}.
\end{align*}
\end{theorem}

\begin{proof}
For a fixed \(0\neq w_{\rm ms}\in W_{\rm ms}\), define
\(\eta_h=\eta_h(w_{\rm ms})\) satisfying \eqref{eq:ideal-adjoint}. 
By \cref{eta_h-eta_ms}, 
Choose the multiscale approximation
\(\eta_{\rm ms}\in W_{\rm ms}\) satisfying
\eqref{ms-approximation-eta}, and set $v_{\rm ms}
:=
w_{\rm ms}+\eta_{\rm ms}
\in W_{\rm ms}. $ 
By \eqref{eq:ideal-adjoint}, we have $a_h(w_{\rm ms},w_{\rm ms}+\eta_h)
=
a_h(w_{\rm ms},w_{\rm ms})
+
2k^2\|w_{\rm ms}\|_{0,\Omega}^2
=
b_h(w_{\rm ms},w_{\rm ms}).$ 
Therefore, the continuity of $a_h$ and \cref{eta_h-eta_ms} give
\begin{align}
|a_h(w_{\rm ms},v_{\rm ms})|
\geq
|b_h(w_{\rm ms},w_{\rm ms})|
-
|a_h(w_{\rm ms},\eta_h-\eta_{\rm ms})|
\geq
\left(
\frac{\sqrt{2}}{2}-\theta_{\rm ms}
\right)
\|w_{\rm ms}\|_D^2.
\label{eq:ms-numerator-lower}
\end{align}
Furthermore, using \cref{thm:main-infsup} and \eqref{eq:ideal-adjoint},
\begin{align*}
\frac{\sqrt{2}}{4(1+C_{T,k h}k)}\|\eta_h\|_D
&\le
\sup_{v_h\neq0}
\frac{|a_h(v_h,\eta_h)|}{\|v_h\|_D}
=
\sup_{v_h\neq0}
\frac{2k^2|(v_h,w_{\rm ms})_\Omega|}{\|v_h\|_D}.
\end{align*}
Since $k\|v_h\|_{0,\Omega}\le\|v_h\|_D$ and $k\|w_{\rm ms}\|_{0,\Omega}\le\|w_{\rm ms}\|_D$,
\begin{align*}
2k^2|(v_h,w_{\rm ms})_\Omega|
&\le2k^2\|v_h\|_{0,\Omega}\|w_{\rm ms}\|_{0,\Omega}
\le2\|v_h\|_D\|w_{\rm ms}\|_D.
\end{align*}
This proves $\|\eta_h\|_D
\leq
4\sqrt{2}(1+C_{T,k h}k)\|w_{\rm ms}\|_D.$  Hence
\begin{align}
\|v_{\rm ms}\|_D
&\leq
\|w_{\rm ms}\|_D
+
\|\eta_{\rm ms}\|_D
\nonumber\\
&\leq
\|w_{\rm ms}\|_D
+
\|\eta_h\|_D
+
\|\eta_h-\eta_{\rm ms}\|_D
\nonumber\\
&\leq
\left(
7(1+C_{T,k h}k)+\theta_{\rm ms}
\right)
\|w_{\rm ms}\|_D.
\label{eq:ms-test-upper}
\end{align}
Consequently, when $\theta_{\rm ms}\le\frac{\sqrt{2}}{4},$ combining \eqref{eq:ms-numerator-lower} and \eqref{eq:ms-test-upper}, 
we have
\begin{align*}
\frac{\sqrt{2}}{4(7(1+C_{T,k h}k)+\theta_{\rm ms})}\le
\inf_{0\neq w_{\rm ms}\in W_{\rm ms}}
\sup_{0\neq v_{\rm ms}\in W_{\rm ms}}
\frac{
|a_h(w_{\rm ms},v_{\rm ms})|
}{
\|w_{\rm ms}\|_D\|v_{\rm ms}\|_D
}=:\beta_{\rm ms}.
\end{align*}
\end{proof}
It is a direct consequence of this inf-sup condition and Lemma \ref{a_and_b} that the multiscale discrete problem \eqref{GMs} is well-posed.
\begin{corollary}
\label[corollary]{multiscale-solvable}
Under the assumptions of \cref{thm:ms-infsup}, the multiscale discrete problem \eqref{GMs} admits a unique solution
$u_{\rm ms}\in W_{\rm ms}$.
\end{corollary}
	\section{Convergence of mixed GMsFEM for Helmholtz equation}
    \label{sec-5}
In this section, we establish the convergence of the proposed mixed GMsFEM with velocity elimination in the oversampling setting. We first derive a local decomposition of the fine-grid solution and show that the boundary-driven component admits a representation in the corresponding local snapshot space. Based on this representation and the local spectral approximation, we then obtain a best-approximation estimate of the fine-grid solution in the corrected multiscale space. Finally, combining this approximation property with the stability of the multiscale discretization yields the a priori error estimate for the multiscale solution.

For each $T_i\in\mathcal T_H$, we consider a local problem on the
oversampled region $T_i^m$. We use the numerical interface trace
introduced in \eqref{eq:numerical-interface-traces}; the restriction
argument is the same as that used to derive
\eqref{eq:adjoint-local-restriction}. Let
$(\widehat{\Bq},\widehat{u})
\in \BV_h(T_i^m)\times W_h(T_i^m)$
satisfy
\begin{subequations}
\label{eq:primal-local-recovery}
\begin{align}
(\widehat{\Bq},\Bv_h)_{Q,T_i^m}
-(\widehat{u},\nabla\cdot\Bv_h)_{T_i^m}
+\frac{i}{k}
\left\langle
\widehat{\Bq}\cdot\Bn_i+\widetilde{g},
\Bv_h\cdot\Bn_i
\right\rangle_{\partial T_i^m}
&=0,
\label{eq:primal-local-recovery-a}
\\
(\nabla\cdot\widehat{\Bq},p_h)_{T_i^m}
-(k^2\widehat{u},p_h)_{T_i^m}
&=
(f,p_h)_{T_i^m},
\label{eq:primal-local-recovery-b}
\end{align}
\end{subequations}
for all
$(\Bv_h,p_h)\in
\BV_h(T_i^m)\times W_h(T_i^m)$,
where the boundary datum $\widetilde{g}$ is defined by
\begin{equation*}
\widetilde{g}
=
\begin{cases}
g,
& \text{on } \partial T_i^m\cap\partial\Omega,
\\[1mm]
g_{i,h}:=-\Bq_h\cdot\Bn_i-ik \{u_h\},
& \text{on } \partial T_i^m\setminus\partial\Omega.
\end{cases}
\end{equation*}
Here, the boundary data on
$\partial T_i^m\setminus\partial\Omega$ are inherited from the global
fine-grid solution $(\Bq_h,u_h)$. By the same edgewise argument as in
the derivation of \eqref{eq:adjoint-local-restriction}, using
\eqref{eq:numerical-interface-traces},
$(\Bq_h,u_h)|_{T_i^m}$ satisfies
\eqref{eq:primal-local-recovery}, and the uniqueness of the local
problem yields
\[
(\widehat{\Bq},\widehat{u})
=
(\Bq_h,u_h)|_{T_i^m}.
\]

Following the same source--boundary decomposition used for the
adjoint problem above, we decompose the local solution into a
source-driven component and a boundary-driven component. More precisely,
$(\Bq_{h,T_i^m}^{\mathrm{I}},u_{h,T_i^m}^{\mathrm{I}})$
is defined by
\begin{subequations}
\label{decom_1}
\begin{align}
(\Bq_{h,T_i^m}^{\mathrm{I}},\Bv_h)_{Q,T_i^m}
-(u_{h,T_i^m}^{\mathrm{I}},\nabla\cdot\Bv_h)_{T_i^m}
+\frac{i}{k}
\left\langle
\Bq_{h,T_i^m}^{\mathrm{I}}\cdot\Bn_i,
\Bv_h\cdot\Bn_i
\right\rangle_{\partial T_i^m}
&=0,
\label{decom_1.1}
\\
(\nabla\cdot\Bq_{h,T_i^m}^{\mathrm{I}},p_h)_{T_i^m}
-(k^2u_{h,T_i^m}^{\mathrm{I}},p_h)_{T_i^m}
&=
(f,p_h)_{T_i^m},
\label{decom_1.2}
\end{align}
\end{subequations}
for all
$(\Bv_h,p_h)\in
\BV_h(T_i^m)\times W_h(T_i^m)$. 
The second component
$(\Bq_{h,T_i^m}^{\mathrm{II}},u_{h,T_i^m}^{\mathrm{II}})$
is defined by
\begin{subequations}
\label{decom_2}
\begin{align}
(\Bq_{h,T_i^m}^{\mathrm{II}},\Bv_h)_{Q,T_i^m}
-(u_{h,T_i^m}^{\mathrm{II}},\nabla\cdot\Bv_h)_{T_i^m}
+\frac{i}{k}
\left\langle
\Bq_{h,T_i^m}^{\mathrm{II}}\cdot\Bn_i+\widetilde{g},
\Bv_h\cdot\Bn_i
\right\rangle_{\partial T_i^m}
&=0,
\label{decom_2.1}
\\
(\nabla\cdot\Bq_{h,T_i^m}^{\mathrm{II}},p_h)_{T_i^m}
-(k^2u_{h,T_i^m}^{\mathrm{II}},p_h)_{T_i^m}
&=0,
\label{decom_2.2}
\end{align}
\end{subequations}
for all
$(\Bv_h,p_h)\in
\BV_h(T_i^m)\times W_h(T_i^m)$.

By construction and the well-posedness of
\eqref{eq:primal-local-recovery}, we have
\begin{equation*}
(\Bq_h,u_h)|_{T_i^m}
=
(\Bq_{h,T_i^m}^{\mathrm{I}},u_{h,T_i^m}^{\mathrm{I}})
+
(\Bq_{h,T_i^m}^{\mathrm{II}},u_{h,T_i^m}^{\mathrm{II}}).
\end{equation*}
The corresponding piecewise components on the coarse grid  we
define, for $s\in\{\mathrm I,\mathrm{II}\}$,
\begin{align*}
    \Bq_h^s|_{T_i}
:=
\Bq_{h,T_i^m}^s|_{T_i},
\qquad
u_h^s|_{T_i}
:=
u_{h,T_i^m}^s|_{T_i}.
\end{align*}
Consequently,
\begin{equation*}
\Bq_h
=
\Bq_h^{\mathrm{I}}+\Bq_h^{\mathrm{II}},
\qquad
u_h
=
u_h^{\mathrm{I}}+u_h^{\mathrm{II}}.
\end{equation*}

We first estimate the contribution of the source-driven component
$u_{h,T_i^m}^{\mathrm{I}}$.

\begin{lemma}
\label{lem:abc-source-energy}
Let Assumption \ref{assumption} be valid. Let
$u_{h,T_i^m}^{\mathrm{I}}$ be defined in \eqref{decom_1}. Then
\begin{equation}
a_i\!\left(u_{h,T_i^m}^{\mathrm{I}},u_{h,T_i^m}^{\mathrm{I}}\right)
\le \frac{2(H+2mh)h}{k}\|f\|_{0,T_i^m}^2.
\label{eq:abc-source-energy}
\end{equation}
\end{lemma}

The proof of \Cref{lem:abc-source-energy} is deferred to
Appendix~\ref{proof_F}. We then turn to the boundary-driven component
$u_{h,T_i^m}^{\mathrm{II}}$ and show that it admits an exact representation in the local
snapshot space introduced in \Cref{sec-3}.

\begin{lemma}
\label{decom_proof}
Let $(\boldsymbol q_{h,T_i^m}^{\rm I},u_{h,T_i^m}^{\rm I})
\in \BV_h(T_i^m)\times W_h(T_i^m)$ and  $(\boldsymbol q_{h,T_i^m}^{\rm II},u_{h,T_i^m}^{\rm II})
\in \BV_h(T_i^m)\times W_h(T_i^m)$ be the solutions of \eqref{decom_1} and \eqref{decom_2}, then $u_{h,T_i^m}^{\rm I}=\psi_{f,i}^m$ and $u_{h,T_i^m}^{\rm II}\in W_{i,m}^{\rm snap},$
 respectively. Then there exist coefficients
$c_{i,1},\ldots,c_{i,J_i}\in\mathbb C$ such that $u_{h,T_i^m}^{\rm II}
 =\sum_{j=1}^{J_i}c_{i,j}\psi_i^j.$
\end{lemma}

\begin{proof}
Let $g_h$ denote the edgewise constant projection of $g$ on
$\partial T_i^m\cap\partial\Omega$ . Since $\boldsymbol v_h\cdot\boldsymbol n$ is constant
on each boundary fine edge,
\[
\langle g_h,\boldsymbol v_h\cdot\boldsymbol n\rangle_{\partial T_i^m\cap\partial\Omega}
=
\langle g,\boldsymbol v_h\cdot\boldsymbol n\rangle_{\partial T_i^m\cap\partial\Omega}.
\]
Define the discrete impedance datum $\eta_i$ on $\partial T_i^m$ by
\[
\eta_i:=
\begin{cases}
g_{i,h},
&\text{on }\partial T_i^m\setminus \partial\Omega,\\[1mm]
g_h,
&\text{on }\partial T_i^m\cap\partial\Omega.
\end{cases}
\]
Then \eqref{decom_2.1} can be written equivalently as
\begin{equation}
(\boldsymbol q_{h,T_i^m}^{\rm II},\boldsymbol v_h)_{Q,T_i^m}
-(u_{h,T_i^m}^{\rm II},\nabla\cdot\boldsymbol v_h)_{T_i^m}
+\frac{i}{k}
 \langle
 \boldsymbol q_{h,T_i^m}^{\rm II}\cdot\boldsymbol n+\eta_i,
 \boldsymbol v_h\cdot\boldsymbol n
 \rangle_{\partial T_i^m}
=0.
\label{eq:full-impedance-II}
\end{equation}

Decompose $\eta_i$ edgewise as $\eta_i=\sum_{j=1}^{J_i}\eta_{i,j},$ $
\operatorname{supp}\eta_{i,j}\subset e_i^j.$
For each $j$, let
$(\boldsymbol q_{i,j},u_{i,j})
\in \BV_h(T_i^m)\times W_h(T_i^m)$ satisfy
\begin{subequations}
\label{edge-impedance}
    \begin{align}
(\boldsymbol q_{i,j},\boldsymbol v_h)_{Q,T_i^m}
-(u_{i,j},\nabla\cdot\boldsymbol v_h)_{T_i^m}
+\frac{i}{k}
 \langle
 \boldsymbol q_{i,j}\cdot\boldsymbol n+\eta_{i,j},
 \boldsymbol v_h\cdot\boldsymbol n
 \rangle_{\partial T_i^m}
&=0,\label{edge-impedance_1}\\
(\nabla\cdot\boldsymbol q_{i,j},p_h)_{T_i^m}
-(k^2u_{i,j},p_h)_{T_i^m}
&=0.\label{edge-impedance_2}
\end{align}
\end{subequations}
Summing \eqref{edge-impedance}
over $j=1,\ldots,J_i$ shows that $\left(
\sum_{j=1}^{J_i}\boldsymbol q_{i,j},
\sum_{j=1}^{J_i}u_{i,j}
\right) $
satisfies \eqref{eq:full-impedance-II} and \eqref{decom_2.2}.
By uniqueness,
\begin{equation}
\boldsymbol q_{h,T_i^m}^{\rm II}
=\sum_{j=1}^{J_i}\boldsymbol q_{i,j},
\qquad
u_{h,T_i^m}^{\rm II}
=\sum_{j=1}^{J_i}u_{i,j}.
\label{eq:sum-edge-solutions}
\end{equation}

For each $j$, define $c_{i,j}
:=
\left.
\frac{i}{k}
\left(
\boldsymbol q_{i,j}\cdot\boldsymbol n+\eta_{i,j}
\right)
\right|_{e_i^j}. $
Since both $\boldsymbol q_{i,j}\cdot\boldsymbol n$ and
$\eta_{i,j}$ are constant on $e_i^j$, $c_{i,j}$ is a scalar.
Moreover, since $\eta_{i,j}=0$ on
$\partial T_i^m\setminus e_i^j$, equation
\eqref{edge-impedance_1} becomes
\begin{align*}
(\boldsymbol q_{i,j},\boldsymbol v_h)_{Q,T_i^m}
-(u_{i,j},\nabla\cdot\boldsymbol v_h)_{T_i^m}
+\langle c_{i,j},\boldsymbol v_h\cdot\boldsymbol n\rangle_{e_i^j}
+\frac{i}{k}
 \langle\boldsymbol q_{i,j}\cdot\boldsymbol n,
        \boldsymbol v_h\cdot\boldsymbol n
 \rangle_{\partial T_i^m\setminus e_i^j}
=0.
\end{align*}
Together with \eqref{edge-impedance_2}, this is exactly the
discrete snapshot problem with boundary value $c_{i,j}$ on $e_i^j$.
By linearity and uniqueness, we have $(\boldsymbol q_{i,j},u_{i,j})
=
c_{i,j}(\boldsymbol\Phi_i^j,\psi_i^j).$
Substitution into \eqref{eq:sum-edge-solutions} yields
$u_{h,T_i^m}^{\rm II}
=
\sum_{j=1}^{J_i}c_{i,j}\psi_i^j, $ 
and hence $u_{h,T_i^m}^{\rm II}\in W_{i,m}^{\rm snap}.$

\end{proof}  
 
In the following, we present the approximation properties of the multiscale space $W_{\rm ms}$, which plays a key role in our analysis. 

\begin{lemma}
\label{lemma40}
Under Assumption~\ref{assumption}, 
 let $u_h\in W_h$ be the solution of (\ref{bf}). There exists $v^{\rm ms}\in W_{\rm ms}$ such that
\begin{align*}
\|u_h-(u_f+v^{\rm ms})\|_{D}
\le 
\frac{H}{\sqrt{\Lambda_{2,m}}}
\left(
\sqrt{2C_{R,m}N_{\rm ov}}(\sqrt{h}+k^{-\frac{1}{2}})\|u_h\|_D
+2\sqrt{\frac{C_{R,m}N_{\rm ov}(H+2mh)h}{k}}\|f\|_{0,\Omega}\right),
\end{align*}
where $\Lambda_{2,m}
=
\min\{\lambda_{l_i+1,2},\ i=1,\cdots,N_T\}. $  
\end{lemma}
\begin{proof}
    For each $T_i\in \mathcal{T}_H$ , the local decomposition holds $u_h=u_{h}^{\mathrm{I}}+u_{h}^{\mathrm{II}}.$  By \cref{decom_proof}, we obtain $u_{h,T_i^m}^{\rm I}=\psi_{f,i}^m, $ $u_{h,T_i^m}^{\mathrm{II}}\in W_{i,m}^{\mathrm{snap}}. $
    So, we can define $v^{\mathrm{snap}}$ by $v^{\mathrm{snap}}|_{T_i}=u_{h,T_i^m}^{\mathrm{II}}|_{T_i}$. 
 By definition of $v^{\mathrm{snap}}$, we have
 \begin{align*}
     \|u_h-(u_f+v^{\mathrm{snap}})\|_{D}
     &=\|(u_h^{\mathrm{I}}+u_h^{\mathrm{II}})-(u_f+v^{\mathrm{snap}})\|_{D}\\
     &\leq  \|u_h^{\mathrm{I}}-u_f\|_{D}+\|u_h^{\mathrm{II}}-v^{\mathrm{snap}}\|_{D}\\
     &\leq \|\sum_{T_i\in \mathcal{T}_H}(u_{h}^{\mathrm{I}}|_{T_i}-u_f|_{T_i})\|_{D}+\|\sum_{T_i\in \mathcal{T}_H}(u_{h}^{\mathrm{II}}|_{T_i}-u_{h,T_i^m}^{\mathrm{II}}|_{T_i})\|_{D}\\
     &\leq\|\sum_{T_i\in \mathcal{T}_H}(u_{h,T_i^m}^{\mathrm{I}}|_{T_i}-u_f|_{T_i})\|_{D}
     +\|\sum_{T_i\in \mathcal{T}_H}(u_{h,T_i^m}^{\mathrm{II}}|_{T_i}-u_{h,T_i^m}^{\mathrm{II}}|_{T_i})\|_{D}  \\
     &=0.
 \end{align*}
Expand the oversampled component in the $s_i$-orthonormal eigenbasis:
\[
v_{i,T_i^m}^{\mathrm{snap}}:=u_{h,T_i^m}^{\mathrm{II}}
=\sum_{j=1}^{r_i}\widehat{v}_{i,j}\tilde{\psi_i^j}.
\]
Its restriction gives
$v^{\mathrm{snap}}=\sum_{i=1}^{N_T}\sum_{j=1}^{r_i}
\widehat{v}_{i,j}\tilde{\psi_i^j}|_{T_i}$, with each restriction
extended by zero outside $T_i$. We then define
$v^{\mathrm{ms}}\in W_{\mathrm{ms}}$ by
\[
v^{\mathrm{ms}}:=\sum_{i=1}^{N_T}v^{ms}_i:=\sum_{i=1}^{N_T}\sum_{j=1}^{l_i}\widehat{v}_{i,j}\tilde{\psi_i^j}|_{T_i}
\]
and $v_{i,T_i^m}^{\mathrm{ms}}=\sum_{j=1}^{l_i}\widehat{v}_{i,j}\tilde{\psi_i^j}$,
 where $l_i \leq r_i$ is the number of eigenfunctions selected for $T_i$. 
Set $e_i^m=v_{i,T_i^m}^{\mathrm{snap}}
-v_{i,T_i^m}^{\mathrm{ms}}$.
The restriction estimate \eqref{eq:restriction-assembly} and spectral
orthogonality give
\begin{align*}
\|v^{\mathrm{snap}}-v^{\mathrm{ms}}\|_D^2
&\le C_{R,m}H^2\sum_i s_i(e_i^m,e_i^m)\\
&\le \frac{C_{R,m}H^2}{\Lambda_{2,m}}\sum_i a_i(e_i^m,e_i^m)\\
&\le \frac{C_{R,m}H^2}{\Lambda_{2,m}}
\sum_i a_i(u_{h,T_i^m}^{\mathrm{II}},u_{h,T_i^m}^{\mathrm{II}})\\
&\le \frac{2C_{R,m}H^2}{\Lambda_{2,m}}
\left(\sum_i a_i(u_h,u_h)
+\sum_i a_i(u_{h,T_i^m}^{\mathrm I},u_{h,T_i^m}^{\mathrm I})\right),
\end{align*}
where $\Lambda_{2,m}
=
\min\{\lambda_{l_i+1,2},\ i=1,\cdots,N_T\}. $
On an artificial patch-boundary edge,
$kh^2|u_h|_{t_e}|^2=k\|u_h\|_{0,t_e}^2$.
Counting these boundary-cell incidences in $N_{\rm ov}$ therefore yields
\[
\sum_i a_i(u_h,u_h)
\le N_{\rm ov}(h+k^{-1})\|u_h\|_D^2.
\]
Furthermore, \cref{lem:abc-source-energy} and finite overlap give
\[
\sum_i a_i(u_{h,T_i^m}^{\mathrm I},u_{h,T_i^m}^{\mathrm I})
\le \frac{2N_{\rm ov}(H+2mh)h}{k}\|f\|_{0,\Omega}^2.
\]
Thus we have
\begin{align*}
\|v^{\mathrm{snap}}-v^{\mathrm{ms}}\|_{D}
\le
\frac{H}{\sqrt{\Lambda_{2,m}}}
\left(
\sqrt{2C_{R,m}N_{\rm ov}}
(\sqrt{h}+k^{-\frac{1}{2}})\|u_h\|_D
+2\sqrt{\frac{C_{R,m}N_{\rm ov}(H+2mh)h}{k}}\|f\|_{0,\Omega}
\right).
\end{align*}
By the following estimates
$$
\|u_h-(u_f+v_{\rm ms})\|_{D}
\leq
\|u_h-(u_f+v^{\mathrm{snap}})\|_{D}
+\|v^{\mathrm{snap}}-v^{ms}\|_{D},
$$ 
we finish the proof.
\end{proof}

With the above preparation, we can derive the final result.
\begin{theorem}
\label{thm:quasi-optimal-ms}
Under the hypotheses of \cref{thm:ms-infsup}, let $u_h\in W_h$
be the solution of \eqref{bf}, and let
$u_{\rm ms}\in W_{\rm ms}$ be the multiscale solution of \eqref{GMs}. Then the corrected multiscale approximation $u_f+u_{\rm ms}$ satisfies
the following error estimate:
\begin{align*}
    \|u_h-(u_f+u_{\rm ms})\|_D
    &\leq C_{\rm app}\frac{H}{\sqrt{\Lambda_{2,m}}}
\Bigg[
\left(
\frac{1+\sqrt{k h}}{\sqrt{k}}
(C_{T,k h}+k^{-1})
+
\sqrt{\frac{(H+2mh)h}{k}}
\right)
\|f\|_{0,\Omega}
\\
&\hspace{4.8cm}
+
(1+\sqrt{k h})(C_{T,k h}+k^{-1})
\|g\|_{0,\partial\Omega}
\Bigg],
\end{align*}
where $\Lambda_{2,m}
=
\min\{\lambda_{l_i+1,2},\ i=1,\cdots,N_T\}$ and $C_{\rm app}$ is independent of $k$ and $h$.
\end{theorem}

\begin{proof}
Define the $b_h$-projection
$R_{\rm ms}^b:W_h\rightarrow W_{\rm ms}$ by
\[
b_h(R_{\rm ms}^bw_h,v_{\rm ms})
=
b_h(w_h,v_{\rm ms}),
\qquad
\forall v_{\rm ms}\in W_{\rm ms}.
\]
By the coercivity and continuity of $b_h$, for any
$v_{\rm ms}\in W_{\rm ms}$,
\begin{align*}
\frac{1}{\sqrt{2}}\|w_h-R_{\rm ms}^bw_h\|_D^2
&\leq
|b_h(w_h-R_{\rm ms}^bw_h,w_h-R_{\rm ms}^bw_h)|\\
&= 
|b_h(w_h-R_{\rm ms}^bw_h,w_h-v_{\rm ms})|\\
&\leq
\|w_h-R_{\rm ms}^bw_h\|_D
\|w_h-v_{\rm ms}\|_D.
\end{align*}
Hence,
\begin{align*}
\|w_h-R_{\rm ms}^bw_h\|_D
\le
\sqrt{2}
\inf_{v_{\rm ms}\in W_{\rm ms}}
\|w_h-v_{\rm ms}\|_D.
\label{eq:b-projection-best}
\end{align*}

Set
$\widetilde u_h:=u_h-u_f$ and
$e_h:=\widetilde u_h-u_{\rm ms}$, and write
\begin{equation}
e_h
=
\underbrace{\widetilde u_h-R_{\rm ms}^b\widetilde u_h}_{\xi_h}
+
\underbrace{R_{\rm ms}^b\widetilde u_h-u_{\rm ms}}_{\xi_{\rm ms}}.
\label{eq:error-b-decomposition}
\end{equation}
By the definition of $R_{\rm ms}^b$,
$b_h(\xi_h,\xi_{\rm ms})=0$. Moreover, the Galerkin orthogonality gives
$a_h(e_h,v_{\rm ms})=0$ for all $v_{\rm ms}\in W_{\rm ms}$. Therefore,
using
$b_h(v,w)=a_h(v,w)+2k^2(v,w)_\Omega$,
\begin{align*}
\frac{1}{\sqrt{2}}\|\xi_{\rm ms}\|_D^2
&\leq
|b_h(\xi_{\rm ms},\xi_{\rm ms})|
=
|b_h(e_h,\xi_{\rm ms})|
=
2k^2|(e_h,\xi_{\rm ms})_\Omega|.
\end{align*}
For $\xi_{\rm ms}\in W_{\rm ms}$, let
$T(\xi_{\rm ms})\in W_h$ be the solution 
\[
a_h(v_h,T(\xi_{\rm ms}))
=
2k^2(v_h,\xi_{\rm ms})_\Omega,
\qquad
\forall v_h\in W_h,
\]
By \cref{eta_h-eta_ms}, there exists
$T_{\rm ms}(\xi_{\rm ms})\in W_{\rm ms}$ such that
\begin{align*}
2k^2|(e_h,\xi_{\rm ms})_\Omega|
=
|a_h\!\left(
e_h,
T(\xi_{\rm ms})-T_{\rm ms}(\xi_{\rm ms})
\right)|
\leq
\theta_{\rm ms}
\|e_h\|_D\|\xi_{\rm ms}\|_D.
\end{align*}
Consequently, $\|\xi_{\rm ms}\|_D
\leq
\sqrt{2}\theta_{\rm ms}\|e_h\|_D.$ 
Combining this estimate with
\eqref{eq:error-b-decomposition}, we obtain $\|e_h\|_D
\leq
\|\xi_h\|_D
+
\sqrt{2}\theta_{\rm ms}\|e_h\|_D.$ 
When $\theta_{\rm ms}\le \frac{\sqrt{2}}{4}$, using \cref{lemma40} and \cref{fine-grid-stability}, we obtain 
\begin{align*}
\|e_h\|_D
&\leq
\frac{1}{1-\sqrt{2}\theta_{\rm ms}}
\|\widetilde u_h-R_{\rm ms}^b\widetilde u_h\|_D
\nonumber\\
&\le
\frac{\sqrt{2}}
{\bigl(1-\sqrt{2}\theta_{\rm ms}\bigr)}
\inf_{v_{\rm ms}\in W_{\rm ms}}
\|\widetilde u_h-v_{\rm ms}\|_D\\
&\le 4\frac{H}{\sqrt{\Lambda_{2,m}}}
\left(\sqrt{C_{R,m}N_{\rm ov}}
(\sqrt{h}+k^{-\frac{1}{2}})\|u_h\|_D
+\sqrt{\frac{2C_{R,m}N_{\rm ov}(H+2mh)h}{k}}\|f\|_{0,\Omega}\right)\\
&\le
C_{\rm app}\frac{H}{\sqrt{\Lambda_{2,m}}}
\Bigg[
\left(
\frac{1+\sqrt{k h}}{\sqrt{k}}
(C_{T,k h}+k^{-1})
+
\sqrt{\frac{(H+2mh)h}{k}}
\right)
\|f\|_{0,\Omega}
\nonumber\\
&\hspace{4.8cm}
+
(1+\sqrt{k h})(C_{T,k h}+k^{-1})
\|g\|_{0,\partial\Omega}
\Bigg],
\end{align*}
where $C_{\rm app}$ is independent of $k$ and $h$. This yields the desired estimate and completes the proof.
\end{proof}

	\section{Numerical examples}
    \label{sec-6}
	The purpose of this section is to validate the proposed method as an accurate and
efficient reduced model for high-frequency Helmholtz problems.  To this end, we use
homogeneous and strongly heterogeneous examples in two and three dimensions.  We
first determine how the oversampling size and local spectral dimension control the
approximation error.  Having identified representative low-error spaces, we then
compare the reduced and fine-grid models in terms of global dimension, matrix
construction time, and direct solution time.

To make the accuracy comparison precise, let $u_h\in W_h$ be the fine-grid
reference solution and let
$u_f+u_{\mathrm{ms}}$ be the total multiscale
approximation.  We report
\[
 e_{L^2}
 =\frac{\lVert u_h-(u_f+u_{\mathrm{ms}})\rVert_{0,\Omega}}
        {\lVert u_h\rVert_{0,\Omega}},
 \qquad
 e_D
 =\frac{\lVert u_h-(u_f+u_{\mathrm{ms}})\rVert_D}
        {\lVert u_h\rVert_D},
\]
where $\lVert\cdot\rVert_D$ is the discrete energy norm defined in Section~\ref{sec-4};
in heterogeneous media, its wavenumber and boundary weights are evaluated
locally.
Since the wavelength varies with the medium, we characterize the worst resolved
part of each model by
\[
 \mathrm{ppw}:=\min_{T\in\mathcal{T}_h}
 \frac{2\pi}{k_T h_T},
\]
where $h_T$ is the local fine-grid spacing and
$k_T=\max_{\boldsymbol{x}\in T}k(\boldsymbol{x})$.  Thus, for the
homogeneous media and uniform grids considered here,
$\mathrm{ppw}=2\pi/(k h)$.

For the computational comparison, matrix construction and linear-system solution
are timed separately.
$T_{\mathrm{mat}}^h$ denotes the distributed generation of matrix entries and
global assembly of the fine-grid matrix.  The time
$T_{\mathrm{mat}}^{\mathrm{ms}}$ includes all source-independent work needed
to construct the reduced matrix: local snapshot solves, local spectral
problems, basis construction, and Galerkin assembly.  The fine-grid solution
time $T_{\mathrm{solve}}^h$ includes the direct solver's symbolic analysis,
numerical LU factorization, and forward and backward substitutions for the
reported right-hand side.  The reduced solution time
$T_{\mathrm{solve}}^{\mathrm{ms}}$ includes the same three direct-solver
stages for the reduced system. The source-dependent local particular solves
and final reconstruction are measured separately and are not included in
$T_{\mathrm{solve}}^{\mathrm{ms}}$. Each compute node
has two 32-core Intel Xeon Platinum 8358P processors, for 64 CPU cores per
node, together with 512 GB of memory; the processor base frequency is
2.60 GHz.  All computations use single precision.  The global fine and reduced
systems are solved by
full-rank CMUMPS, whereas the independent local
snapshot problems use full-rank CMUMPS in two dimensions and CMUMPS-BLR with
tolerance $10^{-5}$ in three dimensions.

\subsection{Two-dimensional examples}
\label{sec:example-2d}

We first examine oversampling in a homogeneous medium and then consider
the strongly heterogeneous Marmousi model. The first problem is posed on
$\Omega=(0,1)^2$ with unit wave
speed, a point source at $\boldsymbol{x}_s=(0.5,0.5)$, and a first-order
absorbing boundary condition.  The $RT_0$--$P_0$ method with velocity
elimination uses a $2000\times2000$ fine grid and a $50\times50$ coarse grid.
Consequently,
\[
 h=1/2000,\qquad H=1/50,\qquad H/h=40,\qquad
 \dim(W_h)=4{,}000{,}000.
\]
We set $\mathrm{ppw}=10$, or $k=400\pi$, and retain
$l_i=42,44,46,48$ local modes to examine the interplay between spectral
dimension and oversampling at the fixed ratio $H/h=40$.

We then retain \(H/h=40\) and introduce strong heterogeneity using the
Marmousi velocity model \cite{Versteeg1994}, resampled onto a
$2400\times760$ fine grid of spacing $h=4\,\mathrm m$, with a
$60\times19$ coarse partition. In horizontal-distance and depth
coordinates, the domain is $(0,9600)\times(0,3040)\,\mathrm m^2$,
and the source is at $(4800,40)\,\mathrm m$, with depth measured downward
from the top boundary. Thus
$H/h=40$ and $\dim(W_h)=1{,}824{,}000$.  The minimum resolution is
$\mathrm{ppw}=10$. In both examples, fine-edge boundary traces are used, a local
particular solution accounts for the source, and $m=0,\ldots,6$ denotes the
number of oversampling layers.

Figure~\ref{fig:2d-over} shows a steady error reduction for the homogeneous
problem as the local regions grow.  At $m=6$, increasing $l_i$ from $42$ to
$46$ lowers $e_{L^2}$ from $0.00878$ to $0.00443$. The Marmousi example
likewise benefits from oversampling: at $m=6$ and $l_i=46$,
the relative $L^2$ and $D$-norm errors are $0.00695$ and $0.00798$.
Thus, oversampling yields sub-percent $L^2$ error in both the homogeneous
and strongly heterogeneous examples at the same fine-to-coarse grid ratio.

\begin{figure}[!htbp]
    \centering
    \includegraphics[width=0.96\textwidth]
    {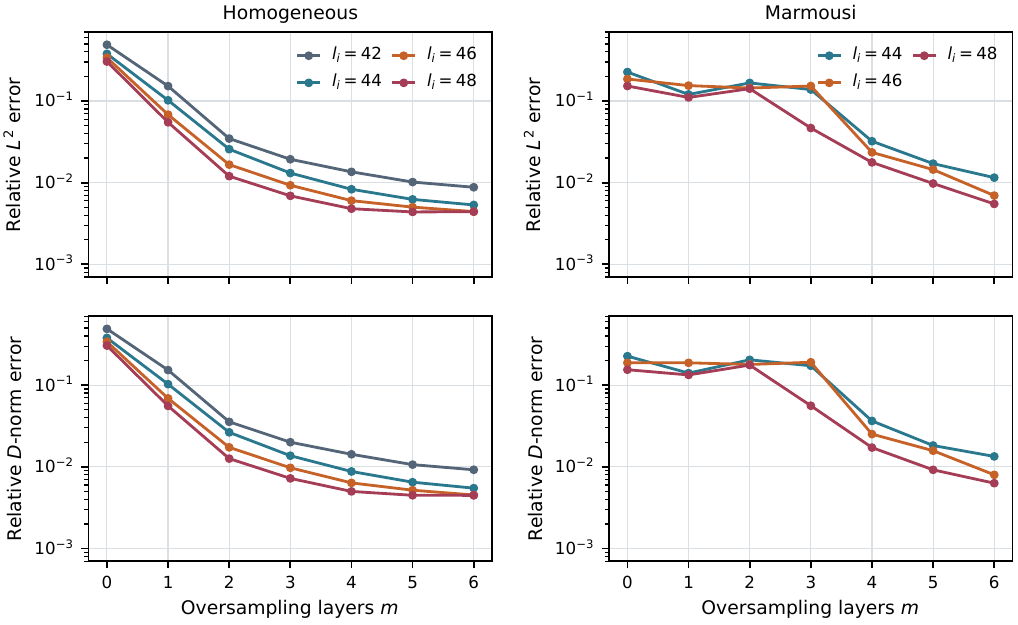}
    \caption{Relative $L^2$ (top) and $D$-norm (bottom) errors versus
    oversampling layers $m$ for the homogeneous and Marmousi problems.
    Both have $\mathrm{ppw}=10$.}
    \label{fig:2d-over}
\end{figure}

Table~\ref{tab:2d-results} reports representative low-error configurations
and compares their costs with the corresponding fine-grid solves.

\begin{table}[htbp]
    \centering
    \small
    \setlength{\tabcolsep}{2.8pt}
    \caption{Two-dimensional fine and multiscale solutions: dimensions,
    timings in seconds, and relative errors at $\mathrm{ppw}=10$ using
    64 CPU cores. The multiscale rows use $m=6$ and $l_i=46$;
    Homo. and Marm. denote the homogeneous and Marmousi problems.}
    \label{tab:2d-results}
    \begin{tabular}{ccccccc}
        \toprule
        case & model & dimension & $T_{\mathrm{mat}}$ &
        $T_{\mathrm{solve}}$ & $e_{L^2}$ & $e_D$ \\
        \midrule
        \multirow{2}{*}{Homo.} & fine & $4{,}000{,}000$ & 0.07 & 27.16 & / & / \\
        & multiscale & $115{,}000$ & 2.59 & 3.69 & 0.00443 & 0.00452 \\
        \midrule
        \multirow{2}{*}{Marm.} & fine & $1{,}824{,}000$ & 0.03 & 12.18 & / & / \\
        & multiscale & $52{,}440$ & 22.70 & 1.22 & 0.00695 & 0.00798 \\
        \bottomrule
    \end{tabular}
\end{table}

\begin{figure}[htbp]
    \centering
    \includegraphics[width=0.96\textwidth]
    {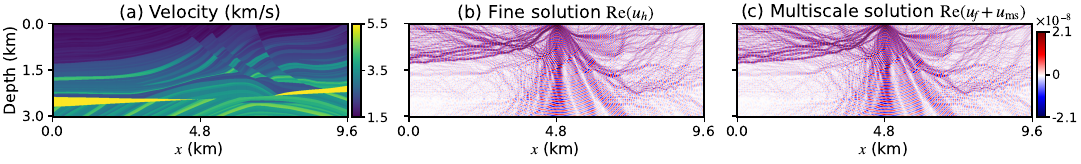}
    \caption{Marmousi example with $\mathrm{ppw}=10$: (a) velocity,
    (b) $\operatorname{Re}(u_h)$, and (c) $\operatorname{Re}(u_f+u_{\mathrm{ms}})$
    for $m=6$ and $l_i=46$.}
    \label{fig:marmousi-fields}
\end{figure}

Figure~\ref{fig:marmousi-fields} confirms that the reduced field reproduces
the principal wavefronts in the heterogeneous medium.  In both examples, the
reduced space has $2.875\%$ of the fine-grid dimension.  The direct solve of
the homogeneous reduced system takes $3.69$ seconds, compared with $27.16$
seconds on the fine grid; the corresponding Marmousi times are $1.22$ and
$12.18$ seconds.  Local construction is source independent and can be reused
for new right-hand sides at fixed medium and frequency.
\FloatBarrier

\subsection{Three-dimensional homogeneous problem}
\label{sec:homo-3d}

We next assess the accuracy and efficiency of the reduced model in three
dimensions. We take $\Omega=(0,1)^3$, unit wave speed, a point source at
$\boldsymbol{x}_s=(0.5,0.5,0.5)$, and the same absorbing boundary condition.
The fine grid contains $100^3$ cells, while the coarse grid contains $10^3$
elements.  Hence
\[
 h=\frac{1}{100},\qquad H=\frac{1}{10},\qquad
 \dim(W_h)=1{,}000{,}000,\qquad H/h=10.
\]
We set $\mathrm{ppw}=10$, corresponding to $k=20\pi$, and vary
$m=0,\ldots,3$ with $l_i=54$. We first use fine-grid boundary traces
and then assess a polynomial trace space for more efficient local
construction.

The curves in Figure~\ref{fig:homo-3d-over} fall rapidly as
the oversampled region expands.  At $m=3$, the relative $L^2$ and $D$-norm
errors are $0.0452$ and $0.0554$. This accuracy provides a baseline for
assessing the smaller polynomial trace space.

\begin{figure}[H]
    \centering
    \includegraphics[width=0.70\textwidth]
    {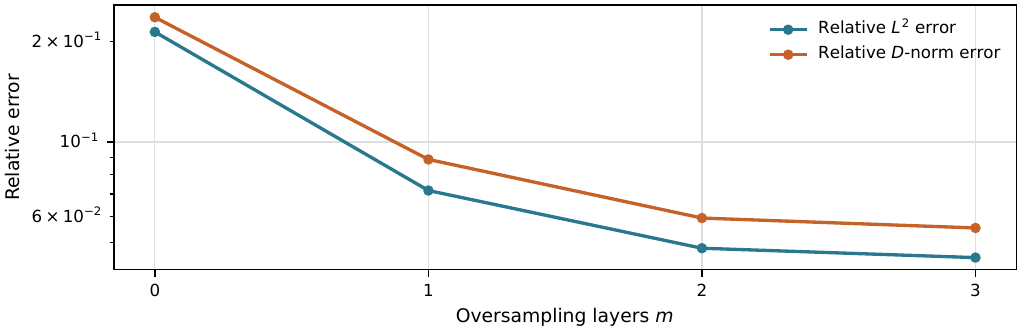}
    \caption{Three-dimensional homogeneous problem at $\mathrm{ppw}=10$:
    relative errors versus $m$ with $l_i=54$.}
    \label{fig:homo-3d-over}
\end{figure}
\FloatBarrier

Table~\ref{tab:homo-3d-results} compares fine-grid boundary traces
with a polynomial-face construction at the same resolution and $m=3$.

\begin{table}[htbp]
    \centering
    \footnotesize
    \setlength{\tabcolsep}{2.9pt}
    \caption{Three-dimensional homogeneous problem: fine and multiscale
    dimensions, timings in seconds, and errors at $\mathrm{ppw}=10$ using
    64 CPU cores. The multiscale rows use $m=3$ and $l_i=54$.}
    \label{tab:homo-3d-results}
    \begin{tabular}{cccccc}
        \toprule
        model & dimension & $T_{\mathrm{mat}}$ &
        $T_{\mathrm{solve}}$ & $e_{L^2}$ & $e_D$ \\
        \midrule
        fine & $1{,}000{,}000$ & 0.01 & 23.41 & / & / \\
        fine-grid traces & $54{,}000$ & 616.67 & 3.24 & 0.04518 & 0.05540 \\
        $\mathbb Q_3$ snapshot & $54{,}000$ & 3.78 & 4.10 & 0.02957 & 0.03761 \\
        \bottomrule
    \end{tabular}
\end{table}

Both reduced systems retain $5.4\%$ of the fine-grid unknowns. The
$\mathbb Q_3$ traces reduce source-independent construction from $616.67$
to $3.78$ seconds in this example while lowering the relative $L^2$ error
from $0.04518$ to $0.02957$. The reduced direct solves take $3.24$ and
$4.10$ seconds, respectively, compared with $23.41$ seconds for the fine
system. Polynomial traces therefore reduce the construction cost while
retaining an accurate reduced solution in this test.

\subsection{Global PML truncation}
\label{sec:overthrust-pml-3d}

The preceding examples truncate the open domain with a first-order
absorbing boundary condition (ABC). We now consider a global perfectly
matched layer (PML) to test the reduction with a different treatment of
outgoing waves. The PML modifies the fine-grid operator through complex
stretching, and its effect is retained in the reduced problem by Galerkin
projection of that operator.

For this comparison we use the heterogeneous Overthrust velocity model
\cite{AminzadehBracKunz1997}, in which wave speed varies from $2.184$ to
$6.000\,\mathrm{km/s}$. The physical domain is discretized into
$504\times504\times120$ cells of width $h=40\,\mathrm{m}$, with a source
at $(2.48,2.48,0.48)\,\mathrm{km}$. With ABC, this domain is partitioned
into $42\times42\times10$ coarse elements. For the PML experiment, the
same physical domain is surrounded by 24 fine cells on the lateral sides
and bottom and 12 on the top. The extended fine and coarse grids are then
$552\times552\times156$ and $46\times46\times13$, respectively. Both
experiments use $H/h=12$, $\mathrm{ppw}=10$, and $m=3$.

We construct the local spaces from the polynomial boundary traces in
\eqref{polynomial-snapshot} with $p=3$; in the PML case, the oversampled
local problems have an additional four-cell absorbing layer. The damping
profile is $\sigma(d)=C\delta^{-1}(d/\delta)^2$, where $d\in[0,\delta]$
is the distance into a layer of width $\delta$; we use $C=20$ globally
and $C=8$ in the auxiliary local layers. Because
the local wavelength changes with velocity, we also compare uniform
mode counts with a source-independent allocation proportional to the
coarse-element mean wavenumber
$\overline{k}_i=|T_i|^{-1}\int_{T_i}k(\boldsymbol{x})\,d\boldsymbol{x}$.
The bounds and target average dimension are fixed for each allocation.
Table~\ref{tab:overthrust-adaptive} reports errors on the physical domain
together with the dimensions and costs of the fine and reduced systems.
For PML, the reported $D$-error uses the interior jump and wavenumber-weighted
mass terms on that domain.

\enlargethispage{5\baselineskip}
\begin{table}[H]
    \centering
    \small
    \setlength{\tabcolsep}{3.2pt}
    \caption{Overthrust model: fine-grid baselines and reduced models with
    uniform or wavenumber-adaptive basis allocation for ABC and PML
    truncations. Here $\overline{l}$ denotes the average per coarse element,
    and all times are in seconds.}
    \label{tab:overthrust-adaptive}
    \begin{tabular}{cccccccccc}
        \toprule
        boundary & model & $\overline{l}$ & $l_i$ range &
        dimension & cores & $T_{\mathrm{mat}}$ &
        $T_{\mathrm{solve}}$ & $e_{L^2}$ & $e_D$ \\
        \midrule
        \multirow{3}{*}{ABC}
          & fine     & \multicolumn{1}{c}{/} & \multicolumn{1}{c}{/} & $30{,}481{,}920$ & 1024 & 0.06 & 1421.12 & \multicolumn{1}{c}{/} & \multicolumn{1}{c}{/} \\
          & uniform  & 54 & 48--54 & $952{,}512$ & 128 & 54.38 & 73.81 & 0.01537 & 0.02732 \\
          & adaptive & 46 & 32--70 & $811{,}440$ & 128 & 53.40 & 63.40 & 0.02207 & 0.02856 \\
        \noalign{\vskip 1pt}
        \cdashline{1-10}
        \noalign{\vskip 2pt}
        \multirow{3}{*}{PML}
          & fine     & \multicolumn{1}{c}{/} & \multicolumn{1}{c}{/} & $47{,}533{,}824$ & 1024 & 0.07 & 5462.37 & \multicolumn{1}{c}{/} & \multicolumn{1}{c}{/} \\
          & uniform  & 54 & 48--54 & $1{,}485{,}384$ & 256 & 140.36 & 144.54 & 0.01754 & 0.03101 \\
          & adaptive & 46 & 33--70 & $1{,}265{,}368$ & 256 & 139.46 & 116.37 & 0.01923 & 0.03054 \\
        \bottomrule
    \end{tabular}
\end{table}

The reduction in global dimension and direct-solve time persists for both
boundary treatments (Table~\ref{tab:overthrust-adaptive}). With adaptive
allocation, the globally coupled dimension falls from $30.5$ to $0.81$
million for ABC and from $47.5$ to $1.27$ million for PML. Using 128 cores
for ABC and 256 for PML, rather than the 1024 cores used for each fine
problem, the reduced systems cut direct-solve times from $1421.12$ to
$63.40$ seconds and from $5462.37$ to $116.37$ seconds, respectively.
The adaptive models retain relative $L^2$ errors near $2\%$. The PML sections in
Figure~\ref{fig:overthrust-pml-fields} show that the reduced solution
retains the main wavefronts of the fine solution.

\begin{figure}[!htbp]
    \centering
    \makebox[\textwidth][c]{\includegraphics[width=1.06\textwidth]
    {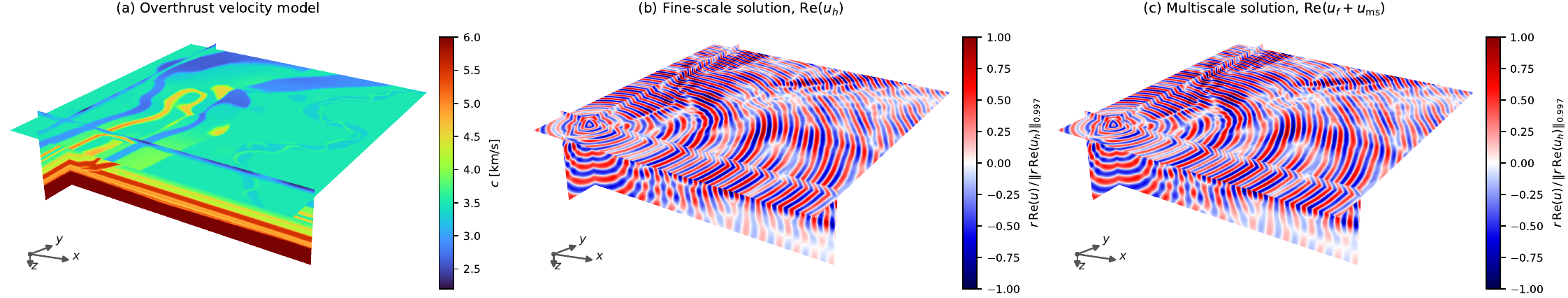}}
    \caption{Overthrust PML problem: source-crossing sections of (a) velocity,
    (b) the direct solution, and (c) the adaptive multiscale solution for
    $m=3$ and $33\leq l_i\leq70$.  The wavefields share the same normalization
    and color scale.}
    \label{fig:overthrust-pml-fields}
\end{figure}
\FloatBarrier
	
	\section{Conclusion}
    \label{sec-7}
We have developed and analyzed a mixed multiscale method for high-frequency
Helmholtz problems as a reusable reduced model for repeated-source
simulations. Mass lumping permits local elimination of the
Raviart--Thomas velocity, leaving a cell-centered pressure system. Local
Dirichlet--impedance solves provide wave-adapted snapshots, which are
compressed spectrally and restricted to individual coarse elements. Thus,
the enlarged regions used to construct the basis do not enlarge its
support in the reduced system, whose coupling remains sparse. Local
particular solutions account for changes in the source without changing
the reduced operator.

The analysis establishes inf--sup stability of the fine-grid and
multiscale problems and bounds the error of the
source-corrected solution in terms of oversampling and the first omitted
local eigenvalue. Numerical experiments in two and three dimensions show
accurate reduced wavefields in homogeneous and heterogeneous media, with
substantially smaller global systems and direct-solve costs, including for
a globally PML-truncated problem. Since the reduced operator is independent
of the source, this reduction is particularly valuable for repeated-source
Helmholtz modeling.

\appendix

\section*{Appendix}

\refstepcounter{section}
\label{proof_A}

\noindent
\textbf{Appendix \thesection. Proof of \Cref{well_posed_of_snap}.}
\par\medskip
We first prove uniqueness. Let \(\zeta=0\) and suppose that
\((\boldsymbol\Phi_h,\psi_h)\) solves the corresponding problem. Taking $\boldsymbol v_h=\boldsymbol\Phi_h$ and $p_h=\psi_h, $ we obtain
\[
\|\boldsymbol\Phi_h\|_{Q,T_i^m}^{2}
-k^2\|\psi_h\|_{0,T_i^m}^{2}
+\frac{i}{k}
\|\boldsymbol\Phi_h\cdot\boldsymbol n\|_{0,\partial T_i^m\setminus e_i^j}^{2}
=0.
\]
Taking the imaginary part yields $\boldsymbol\Phi_h\cdot\boldsymbol n=0,$ on $\partial T_i^m \setminus e_i^j.$ Let \(\{\boldsymbol v_e\}_{e\in\mathcal E_h(T_i^m)}\) be the standard
\(RT_0\) edge basis normalized by $\int_e\boldsymbol v_e\cdot\boldsymbol n_e\,ds=1.$ We can have $\boldsymbol\Phi_h
=
\sum_{e\in\mathcal E_h(T_i^m)}q_e\boldsymbol v_e.$

Without loss of generality, after a rotation of the coordinate
system, assume that \(e_i^j\) is contained in the upper side of
\(T_i^m\). Denote the fine-grid cells by
\[
t_{r,s},
\quad
1\le r\le N_x,\quad 1\le s\le N_y,
\]
where \(s=1\) is the bottom row. 
For every fine-grid edge \(e\) on the bottom side of \(T_i^m\), we
have \(e\in \partial T_i^m\setminus e_i^j\), and hence \(q_e=0\). Taking
\(\boldsymbol v_h=\boldsymbol v_e\) in the first discrete equation
gives
\[
\frac{1}{2}q_e-(\psi_h,\nabla\!\cdot\boldsymbol v_e)_{t_{r,1}}=0.
\]
it follows
that $\psi_h|_{t_{r,1}}=0,$ $r=1,...,N_x$

We now proceed by induction over the rows. Assume that
\[
\psi_h|_{t_{r,s}}=0,
\quad r=1,\ldots,N_x,
\]
and that all fluxes across the lower horizontal edges of the
\(s\)-th row vanish. For every vertical interior edge in the
\(s\)-th row, the two adjacent cell values are zero; hence, the
corresponding edge equation implies that its flux is zero. The
fluxes on the left and right boundaries also vanish since these
edges belong to \(\partial T_i^m\setminus e_i^j\).
So taking \(p_h=\chi_{t_{r,s}}\) in the second discrete equation
yields
\[
\sum_{e\subset\partial t_{r,s}}
q_e
=k^2\psi_h|_{t_{r,s}}h^2=0.
\]
All terms except the flux through the upper horizontal edge are
zero. Therefore, the upper flux is also zero. Applying the first
discrete equation to this horizontal edge gives $\psi_h|_{t_{r,s+1}}=0.$ 
By induction, we obtain 
\[
\psi_h|_t=0,
\quad
\forall t\in\mathcal T_h(T_i^m).
\]The edge equations then imply that all interior fluxes vanish.
The flux on \(\partial T_i^m\setminus e_i^j\) is already zero. Finally, applying the
cell conservation equation to the cell adjacent to \(e_i^j\)
shows that the flux through \(e_i^j\) also vanishes. Thus $\boldsymbol\Phi_h=0.$ Consequently, the homogeneous problem has only the trivial
solution.

It remains to prove existence. Set
\[
X_h:=\boldsymbol V_h(T_i^m)\times W_h(T_i^m)
\]
and define the discrete operator
\(\mathcal A_h:X_h\to X_h'\) by
\[
\langle\mathcal A_hx_h,y_h\rangle
:=
\mathcal B_h(x_h,y_h),
\qquad
x_h,y_h\in X_h,
\]
where \(\mathcal B_h\) denotes the discrete bilinear form above.
The uniqueness result for the homogeneous problem implies
\[
\ker\mathcal A_h=\{0\},
\]
so that \(\mathcal A_h\) is injective. Since \(X_h\) is
finite-dimensional and $\dim X_h=\dim X_h',$ it follows that \(\mathcal A_h\) is also
surjective. Hence, for every linear functional in \(X_h'\), and
in particular for
\[
F_\zeta(\boldsymbol v_h,p_h)
:=
-\langle\zeta,\boldsymbol v_h\cdot\boldsymbol n\rangle_{e_i^j},
\]
there exists a solution. Thus the solution is unique.

\refstepcounter{section}
\label{proof_B}
\par\bigskip
\noindent
\textbf{Appendix \thesection. Proof of \Cref{cell-edge-average-H1}.}
\par\medskip

Without loss of generality, we work on
$t=(0,h)^2$ and take
$e=(0,h)\times\{0\}$ as the bottom edge. For smooth functions,
\begin{align*}
v_t^{\rm av}-v_e^{\rm av}
=\frac1{h^2}\int_0^h\int_0^h\bigl(v(x,y)-v(x,0)\bigr)\,dy\,dx
=\frac1{h^2}\int_0^h\int_0^h(h-t)\,\partial_yv(x,t)\,dt\,dx.
\end{align*}
Hence, by Cauchy--Schwarz inequality,
\begin{align*}
|v_t^{\rm av}-v_e^{\rm av}|
\leq \frac1{h^2}
\left(\int_0^h\int_0^h(h-t)^2\,dt\,dx\right)^{1/2}
\|\partial_yv\|_{0,t}
=\frac1{\sqrt3}\|\partial_yv\|_{0,t}
\leq \frac1{\sqrt3}\|\nabla v\|_{0,t}.
\end{align*}
Density of $C^\infty(\overline t)$ in $H^1(t)$ extends the estimate to all $v\in H^1(t)$.

\refstepcounter{section}
\label{proof_C}
\par\bigskip
\noindent
\textbf{Appendix \thesection. Proof of \Cref{lem:flux}.}
\par\medskip
We present the argument for a vertical interior edge; the horizontal
case follows in exactly the same manner. Let
$e=\partial t^-\cap\partial t^+$ and choose the unit normal
$\Bn_e$ pointing from $t^-$ to $t^+$, so that
$\llbracket P_h^0v\rrbracket_e
=(P_h^0v)|_{t^-}-(P_h^0v)|_{t^+}$.
After a translation, we may assume that
\[
t^-=(-h,0)\times(0,h),\qquad
t^+=(0,h)\times(0,h),\qquad
e=\{0\}\times(0,h),
\]
with $\Bn_e=(1,0)^{\top}$. Set
$\omega_e:=t^-\cup t^+$.

Consider the reference patch
$\widehat\omega=(-1,1)\times(0,1)$, with
$\widehat t^-=(-1,0)\times(0,1)$,
$\widehat t^+=(0,1)\times(0,1)$, and
$\widehat e=\{0\}\times(0,1)$. For
$\widehat v\in H^2(\widehat\omega)$, define
\begin{equation*}
\widehat{\mathcal L}(\widehat v)
:=
\int_{\widehat t^-}\widehat v\,d\widehat x
-
\int_{\widehat t^+}\widehat v\,d\widehat x
+
\int_{\widehat e}
\partial_{\widehat x_1}\widehat v\,d\widehat s .
\end{equation*}
By the trace theorem, $\widehat{\mathcal L}$ is a bounded linear
functional on $H^2(\widehat\omega)$. Moreover, for any
$\widehat p(\widehat x)=a+b\widehat x_1+c\widehat x_2
\in\mathbb P_1(\widehat\omega)$, a direct calculation gives
\[
\int_{\widehat t^-}\widehat p\,d\widehat x
-
\int_{\widehat t^+}\widehat p\,d\widehat x
=-b
=-
\int_{\widehat e}
\partial_{\widehat x_1}\widehat p\,d\widehat s,
\]
and hence $\widehat{\mathcal L}(\widehat p)=0$ for all
$\widehat p\in\mathbb P_1(\widehat\omega)$. The Bramble--Hilbert lemma \cite{BrambleHilbert1970}
therefore yields
\begin{equation}
|\widehat{\mathcal L}(\widehat v)|
\le C_{\rm bh}
|\widehat v|_{2,\widehat\omega}.
\label{eq:BH-flux-local}
\end{equation}

Let $x=x_e+h\widehat x$ and set
$\widehat v(\widehat x):=v(x_e+h\widehat x)$. Since
$|t^\pm|=h^2$, the cell averages are invariant under this scaling:
$(P_h^0v)|_{t^\pm}
=\int_{\widehat t^\pm}\widehat v\,d\widehat x$.
Furthermore,
$\partial_{n_e}v=h^{-1}\partial_{\widehat x_1}\widehat v$
and $ds=h\,d\widehat s$. Hence
\begin{align}
\llbracket P_h^0v\rrbracket_e
+
\int_e\partial_{n_e}v\,ds
&=
\widehat{\mathcal L}(\widehat v),
\label{eq:flux-scaling-1}\\
|\widehat v|_{2,\widehat\omega}
&=
h\,|v|_{2,\omega_e}.
\label{eq:flux-scaling-2}
\end{align}
Combining \eqref{eq:BH-flux-local}--\eqref{eq:flux-scaling-2} gives the
local estimate
\begin{equation}
\left|
\llbracket P_h^0v\rrbracket_e
+
\int_e\partial_{n_e}v\,ds
\right|
\le
C_{\rm bh}h\,|v|_{2,\omega_e}.
\label{eq:local-flux-estimate}
\end{equation}

The same argument applies to horizontal interior edges. Squaring
\eqref{eq:local-flux-estimate} and summing over
$e\in\mathcal E_h^0$, we obtain
\begin{align*}
\sum_{e\in\mathcal E_h^0}
\left|
\llbracket P_h^0z\rrbracket_e
+
\int_e\partial_{n_e}z\,ds
\right|^2
\le
C_{\rm bh}^2h^2
\sum_{e\in\mathcal E_h^0}
|z|_{2,\omega_e}^2
\le
4C_{\rm bh}^2h^2|z|_{2,\Omega}^2,
\end{align*}
Taking the square
root completes the proof.

\refstepcounter{section}
\label{proof_D}
\par\bigskip
\noindent
\textbf{Appendix \thesection. Proof of \Cref{lem:boundary-average}.}
\par\medskip

Consider a boundary cell $t=(0,h)\times(0,h)$ and its bottom edge $e=(0,h)\times\{0\}$. Then by proof of \cref{cell-edge-average-H1} and $\partial_yv(x,t)=\partial_yv(x,0)+\int_0^t\partial_{yy}v(x,s)\,ds,$ we obtain
\begin{align}
v_t^{\rm av}-v_e^{\rm av}
&=\frac12\int_0^h\partial_yv(x,0)\,dx+R_t,
\label{eq:boundary-avg-decomp}
\end{align}
where $R_t
=\frac1{2h^2}\int_0^h\int_0^h(h-s)^2\partial_{yy}v(x,s)\,ds\,dx.$ Cauchy--Schwarz inequality gives
\begin{align*}
\left|\frac12\int_0^h\partial_yv(x,0)\,dx\right|^2
\leq \frac h4\|\partial_yv\|_{0,e}^2\text{ and }
|R_t|\le \frac{h}{2\sqrt{5}} \|\partial_{yy}v\|_{0,t}.
\end{align*}
Therefore $|v_t^{\rm av}-v_e^{\rm av}|^2
\le h\|\partial_nv\|_{0,e}^2+h^2|v|_{2,t}^2.$ Repeating the argument for the four possible boundary orientations and summing gives \eqref{eq:boundary-average}. \eqref{eq:boundary-average-imp} follows from $\partial_nv=-ik v$.


\refstepcounter{section}
\label{proof_E}
\par\bigskip
\noindent
\textbf{Appendix \thesection. Proof of \Cref{eta_h-eta_ms}.}
\par\medskip

We first establish a restriction estimate that will be used for both
components of the approximation error. Let $v_i\in W_h(T_i^m)$ and
define the piecewise restriction $v$ by
\[
v|_{T_i}:=v_i|_{T_i}.
\]
The only contributions to the global $D$-norm that are not directly
controlled by the restrictions of the local energies are the jumps
across coarse interfaces. We therefore fix a fine edge
$e\subset\partial T_i\setminus\partial\Omega$ and consider the chain of
fine cells obtained by starting from the cell of $T_i$ adjacent to $e$
and moving along the outward normal until reaching $\partial T_i^m$.
Let $z_0,\ldots,z_r$ denote the values of $v_i$ on the cells along this
chain, where $r=r_{i,e}\le m$. Set
\[
\beta_{i,e}:=
\begin{cases}
1,
&\text{if the terminal edge belongs to }
\partial T_i^m\setminus\partial\Omega,\\[2mm]
|\alpha|,
&\text{if the terminal edge belongs to }\partial\Omega.
\end{cases}
\]
Writing $\beta=\beta_{i,e}$, the telescoping identity $z_0=\sum_{\ell=0}^{r-1}(z_\ell-z_{\ell+1})+z_r$ 
and the weighted Cauchy--Schwarz inequality give
\begin{equation}
|z_0|^2
\le
(r+\beta^{-1})
\left(
\sum_{\ell=0}^{r-1}|z_\ell-z_{\ell+1}|^2
+\beta|z_r|^2
\right).
\label{eq:chain-with-endpoint}
\end{equation}
The same estimate is valid when $r=0$.

For a fine edge on a coarse interface shared by $T_i$ and $T_j$, let
$z_0^{(i)}$ and $z_0^{(j)}$ be the values on the two adjacent fine
cells. Then
\[
|z_0^{(i)}-z_0^{(j)}|^2
\le
2\left(
|z_0^{(i)}|^2+|z_0^{(j)}|^2
\right).
\]
For each fixed patch, the chains emanating from parallel sides occupy
disjoint outward strips, while chains emanating from perpendicular sides
use fine edges of different orientations. Hence every chain edge and
every terminal edge is counted at most once within that patch.
Furthermore, these chain edges are disjoint from the fine edges lying
strictly inside $T_i$ and from the physical-boundary edges of $T_i$.
Therefore, summing \eqref{eq:chain-with-endpoint} over all coarse
interfaces and adding the interior-edge, physical-boundary, and mass
contributions yields
\begin{equation}
\|v\|_D^2
\le
C_{R,m}(m+1)
\sum_i\|v_i\|_{D,T_i^m}^2
=
C_{R,m}H^2
\sum_i s_i(v_i,v_i),
\label{eq:restriction-assembly}
\end{equation}
where
\begin{equation}
C_{R,m}
:=
2\max\left\{
1,\,
\max_{i,e}
\frac{r_{i,e}+\beta_{i,e}^{-1}}{m+1}
\right\},
\label{eq:restriction-constant}
\end{equation}
with the inner maximum omitted when there are no coarse interfaces.

We shall also use a slightly different form of this estimate for the
first local component. Define
\[
\mathcal D_i(v_i)
:=
\sum_{e\in\mathcal E_{T_i^m}^0}
|\llbracket v_i\rrbracket_e|^2
+
2\sum_{e\in\mathcal E_{T_i^m}^B}
|v_i|_{t_e}|^2
+
k^2\|v_i\|_{0,T_i^m}^2.
\]
Applying \eqref{eq:chain-with-endpoint} with $\beta=2$ gives
\[
r+\beta^{-1}
=
r+\frac12
\le m+1.
\]
Since $|\alpha|\le2$, the physical-boundary contribution is also
controlled by $\mathcal D_i(v_i)$. Repeating the preceding counting
argument therefore gives
\begin{equation}
\|v\|_D^2
\le
2(m+1)
\sum_i\mathcal D_i(v_i).
\label{eq:restriction-dirichlet-energy}
\end{equation}
Notice that both \eqref{eq:restriction-assembly} and
\eqref{eq:restriction-dirichlet-energy} remain valid for $m=0$.

We now apply these estimates to the decomposition of $\eta_h$.
By \eqref{q_u_condition} and \eqref{eq:first-comp-energy},
\begin{align}
\mathcal D_i(\eta_{h,T_i^m}^{\rm I})
&=
\|\Br_{h,T_i^m}^{\rm I}\|_{Q,T_i^m}^2
+
k^2\|\eta_{h,T_i^m}^{\rm I}\|_{0,T_i^m}^2
\nonumber\\
&\le
\frac{20}{9}
k^4(H+2mh)^2
\|w_{\rm ms}\|_{0,T_i^m}^2.
\label{eq:local-adjoint-source-energy}
\end{align}
Applying \eqref{eq:restriction-dirichlet-energy} to the piecewise
restriction in \eqref{decom1.2}, and using the finite overlap of the
patches, we obtain
\begin{align}
\|\eta_h^{\rm I}\|_D^2
&\le
2(m+1)
\sum_{T_i\in\mathcal T_H}
\mathcal D_i(\eta_{h,T_i^m}^{\rm I})
\nonumber\\
&\le
\frac{40}{9}
N_{\rm ov}(m+1)
k^4(H+2mh)^2
\|w_{\rm ms}\|_{0,\Omega}^2
\nonumber\\
&\le
\frac{40}{9}
N_{\rm ov}(m+1)
k^2(H+2mh)^2
\|w_{\rm ms}\|_D^2.
\label{eq:assembled-adjoint-source}
\end{align}
Here $N_{\rm ov}$ denotes a uniform bound for the multiplicities of
patch cells, patch edges, and artificial-boundary cell incidences.

By \cref{thm:adjoint-II-snapshot-representation}, the function
$\eta_{\rm snap}\in W_{\rm snap}$ defined by $\eta_{\rm snap}|_{T_i}
:=
\eta_{h,T_i^m}^{\rm II}|_{T_i}$ is well defined. Since $\eta_h|_{T_i}
=
\eta_{h,T_i^m}^{\rm I}|_{T_i}
+
\eta_{h,T_i^m}^{\rm II}|_{T_i},$ we have $\eta_h-\eta_{\rm snap}
=
\eta_h^{\rm I}.$

Consequently, \eqref{eq:assembled-adjoint-source} yields
\begin{equation}
\|\eta_h-\eta_{\rm snap}\|_D
\le
\frac{2}{3}
\sqrt{10N_{\rm ov}(m+1)}
\,k(H+2mh)
\|w_{\rm ms}\|_D.
\label{eq:eta-snapshot-error}
\end{equation}

Expand the known oversampled component
$\eta_{i,T_i^m}^{\mathrm{snap}}:=\eta_{h,T_i^m}^{\mathrm{II}}$
in the $s_i$-orthonormal eigenbasis:
$\eta_{i,T_i^m}^{\mathrm{snap}}
=\sum_{j=1}^{r_i}\widehat{\eta}_{i,j}\tilde{\psi_i^j}$.
Restricting this expansion to $T_i$ and extending by zero gives
\begin{align*}
\eta_{\mathrm{snap}}
:=\sum_{i=1}^{N_T}\eta_i^{\mathrm{snap}}
=\sum_{i=1}^{N_T}\sum_{j=1}^{r_i}
\widehat{\eta}_{i,j}\psi_i^{j,\mathrm{ms}},
\qquad
\eta_i^{\mathrm{snap}}
:=\sum_{j=1}^{r_i}
\widehat{\eta}_{i,j}\psi_i^{j,\mathrm{ms}}.
\end{align*}
The multiscale approximation is then obtained by retaining the first $l_i$
local eigenfunctions:
\begin{align*}
\eta_{\mathrm{ms}}
:=\sum_{i=1}^{N_T}\eta_i^{\mathrm{ms}}
=\sum_{i=1}^{N_T}\sum_{j=1}^{l_i}
\widehat{\eta}_{i,j}\psi_i^{j,\mathrm{ms}},
\qquad
\eta_i^{\mathrm{ms}}
:=\sum_{j=1}^{l_i}
\widehat{\eta}_{i,j}\psi_i^{j,\mathrm{ms}},
\end{align*}
with
$\eta_{i,T_i^m}^{\mathrm{ms}}
:=\sum_{j=1}^{l_i}\widehat{\eta}_{i,j}\tilde{\psi_i^j}$,
where $l_i\leq r_i$ denotes the number of local eigenfunctions retained for
the coarse region $T_i$. 

Set $e_i^m
:=
\eta_{i,T_i^m}^{\rm snap}
-
\eta_{i,T_i^m}^{\rm ms}
=
\sum_{j=l_i+1}^{r_i}
\widehat{\eta}_{i,j}\tilde{\psi}_i^j.$ 
Applying the restriction estimate \eqref{eq:restriction-assembly} to
these spectral tails gives, for every $m\ge0$,
\begin{equation}
\|\eta_{\rm snap}-\eta_{\rm ms}\|_D^2
\le
C_{R,m}H^2
\sum_{T_i\in\mathcal T_H}
s_i(e_i^m,e_i^m).
\label{eq:eta-spectral-common}
\end{equation}

By the $s_i$-orthogonality of the local eigenfunctions and the local
spectral problem,
\begin{align}
s_i(e_i^m,e_i^m)
&=
\sum_{j=l_i+1}^{r_i}
|\widehat{\eta}_{i,j}|^2
s_i(\tilde{\psi}_i^j,\tilde{\psi}_i^j)
\nonumber\\
&=
\sum_{j=l_i+1}^{r_i}
\frac{|\widehat{\eta}_{i,j}|^2}{\lambda_{j,1}}
a_i(\tilde{\psi}_i^j,\tilde{\psi}_i^j)
\nonumber\\
&\le
\frac{1}{\lambda_{l_i+1,1}}
\sum_{j=l_i+1}^{r_i}
|\widehat{\eta}_{i,j}|^2
a_i(\tilde{\psi}_i^j,\tilde{\psi}_i^j)
\nonumber\\
&=
\frac{1}{\lambda_{l_i+1,1}}
a_i(e_i^m,e_i^m).
\label{eq:eta-local-spectral-estimate}
\end{align}
Hence,
\begin{equation}
\|\eta_{\rm snap}-\eta_{\rm ms}\|_D^2
\le
\frac{C_{R,m}H^2}{\Lambda_{1,m}}
\sum_{T_i\in\mathcal T_H}
a_i(e_i^m,e_i^m),
\label{part2}
\end{equation}
where $\Lambda_{1,m}
:=
\min_{1\le i\le N_T}
\lambda_{l_i+1,1}.$

It remains to estimate the last sum. Since the local eigenfunctions are
$s_i$-orthogonal and satisfy the generalized eigenvalue problem, they
are also $a_i$-orthogonal. Therefore, the retained and discarded parts
of $\eta_{i,T_i^m}^{\rm snap}$ are $a_i$-orthogonal, and thus
\begin{equation}
a_i(e_i^m,e_i^m)
\le
a_i(
\eta_{i,T_i^m}^{\rm snap},
\eta_{i,T_i^m}^{\rm snap}
).
\label{eq:tail-ai-projection}
\end{equation}
Recalling that
\[
\eta_{i,T_i^m}^{\rm snap}
=
\eta_{h,T_i^m}^{\rm II}
=
\eta_h|_{T_i^m}
-
\eta_{i,T_i^m}^{\rm I},
\]
and using the Hermitian positive definiteness of $a_i(\cdot,\cdot)$,
we obtain
\begin{align}
\sum_{T_i\in\mathcal T_H}
a_i(e_i^m,e_i^m)
&\le
2\sum_{T_i\in\mathcal T_H}
a_i(\eta_h,\eta_h)
+
2\sum_{T_i\in\mathcal T_H}
a_i(
\eta_{i,T_i^m}^{\rm I},
\eta_{i,T_i^m}^{\rm I}
).
\label{eq:tail-ai-decomposition}
\end{align}

We first estimate the contribution of $\eta_h$. On an artificial
boundary edge
$e\subset\partial T_i^m\setminus\partial\Omega$, we have
\[
kh^2|(\eta_h)_{t_e}|^2
=
k\|\eta_h\|_{0,t_e}^2.
\]
Hence, by the definition of $a_i(\cdot,\cdot)$, the finite-overlap
property, the definition of the $D$-norm, and
\cref{thm:main-infsup},
\begin{align}
\sum_{T_i\in\mathcal T_H}
a_i(\eta_h,\eta_h)
&=
\sum_{T_i\in\mathcal T_H}
\Bigg(
\sum_{e\in\mathcal E_{T_i^m}^0}
h|\llbracket\eta_h\rrbracket_e|^2
+
\sum_{e\in
\mathcal E_{T_i^m}^B\setminus\mathcal E_h^B}
kh^2|(\eta_h)_{t_e}|^2
+
\sum_{e\in
\mathcal E_{T_i^m}^B\cap\mathcal E_h^B}
h|\alpha|
|\llbracket\eta_h\rrbracket_e|^2
\Bigg)
\nonumber\\
&\le
N_{\rm ov}
\left(
h+\frac1k
\right)
\|\eta_h\|_D^2
\nonumber\\
&\le
N_{\rm ov}
C_{Q,kh}(k^2h+k)
\|w_{\rm ms}\|_D^2,
\label{eq:ai-global-adjoint}
\end{align}
where $C_{Q,kh}
=
32(C_{T,kh}+k^{-1})^2.$

On the other hand, by the definition of $a_i(\cdot,\cdot)$,
\[
a_i(v_h,v_h)
\le
C_{\alpha,kh}\,
h\,
\|v_h\|_{D,T_i^m}^2,
\qquad
C_{\alpha,kh}
:=
\max\{1,kh\}.
\]
Using \cref{first_comp} and the finite overlap of the patches gives
\begin{align}
\sum_{T_i\in\mathcal T_H}
a_i(
\eta_{i,T_i^m}^{\rm I},
\eta_{i,T_i^m}^{\rm I}
)
&\le
\frac{20}{9}
C_{\alpha,kh}
k^4(H+2mh)^2h
\sum_{T_i\in\mathcal T_H}
\|w_{\rm ms}\|_{0,T_i^m}^2
\nonumber\\
&\le
\frac{20}{9}
N_{\rm ov}
C_{\alpha,kh}
k^4(H+2mh)^2h
\|w_{\rm ms}\|_{0,\Omega}^2
\nonumber\\
&\le
\frac{20}{9}
N_{\rm ov}
C_{\alpha,kh}
k^2(H+2mh)^2h
\|w_{\rm ms}\|_D^2.
\label{eq:ai-adjoint-source}
\end{align}
Combining
\eqref{eq:tail-ai-decomposition},
\eqref{eq:ai-global-adjoint}, and
\eqref{eq:ai-adjoint-source}, we conclude that
\begin{equation}
\sum_{T_i\in\mathcal T_H}
a_i(e_i^m,e_i^m)
\le
2N_{\rm ov}
\left(
C_{Q,kh}(k^2h+k)
+
\frac{20}{9}
C_{\alpha,kh}
k^2(H+2mh)^2h
\right)
\|w_{\rm ms}\|_D^2.
\label{eq:tail-ai-final}
\end{equation}

Finally,
\[
\eta_h-\eta_{\rm ms}
=
(\eta_h-\eta_{\rm snap})
+
(\eta_{\rm snap}-\eta_{\rm ms}).
\]
Thus, by the triangle inequality,
\eqref{eq:eta-snapshot-error}, \eqref{part2}, and
\eqref{eq:tail-ai-final},
\begin{equation}
\|\eta_h-\eta_{\rm ms}\|_D
\le
\theta_{\rm ms}\|w_{\rm ms}\|_D,
\label{ms-approximation-eta}
\end{equation}
where
\begin{align*}
\theta_{\rm ms}
={}&
\frac{2}{3}
\sqrt{10N_{\rm ov}(m+1)}
\,k(H+2mh)
\\
&+
\left[
\frac{
2C_{R,m}N_{\rm ov}H^2
}{
\Lambda_{1,m}
}
\left(
C_{Q,kh}(k^2h+k)
+
\frac{20}{9}
C_{\alpha,kh}
k^2(H+2mh)^2h
\right)
\right]^{1/2}.
\end{align*}
This completes the proof.

\refstepcounter{section}
\label{proof_F}
\par\bigskip
\noindent
\textbf{Appendix \thesection. Proof of \Cref{lem:abc-source-energy}.}
\par\medskip

Fix $T_i\in\mathcal T_H$. Let $q_e^{\mathrm{I}}$ denote the coefficient
associated with the basis function $\Bv_e$ in the expansion of
$\Bq_{h,T_i^m}^{\mathrm{I}}$, i.e.,
\[
\Bq_{h,T_i^m}^{\mathrm{I}}
=\sum_{e\in\mathcal E_{T_i^m}}q_e^{\mathrm{I}}\Bv_e.
\]
For the problem \eqref{decom_1}, the same calculation as in
\eqref{q_u_jump} and \eqref{q_u_bjump} gives
\[
q_e^{\mathrm{I}}=\llbracket u_{h,T_i^m}^{\mathrm{I}}\rrbracket_e,
\quad e\in\mathcal E_{T_i^m}^{0},
\qquad
q_e^{\mathrm{I}}=\alpha\llbracket u_{h,T_i^m}^{\mathrm{I}}\rrbracket_e,
\quad e\in\mathcal E_{T_i^m}^{B},
\]
where $\alpha=2kh/(kh+2i)$. Recall that, for
$e\in\mathcal E_{T_i^m}^{B}$,
$\llbracket u_{h,T_i^m}^{\mathrm{I}}\rrbracket_e$ denotes the value of
$u_{h,T_i^m}^{\mathrm{I}}$ on the unique fine element in $T_i^m$
adjacent to $e$.

Taking $p_h=u_{h,T_i^m}^{\mathrm{I}}$ in \eqref{decom_1} and using
the above flux relations yields
\begin{equation}
\sum_{e\in\mathcal E_{T_i^m}^{0}}
\left|\llbracket u_{h,T_i^m}^{\mathrm{I}}\rrbracket_e\right|^2
+\alpha\sum_{e\in\mathcal E_{T_i^m}^{B}}
\left|\llbracket u_{h,T_i^m}^{\mathrm{I}}\rrbracket_e\right|^2
-k^2\|u_{h,T_i^m}^{\mathrm{I}}\|_{0,T_i^m}^2
=(f,u_{h,T_i^m}^{\mathrm{I}})_{T_i^m}.
\label{eq:abc-source-identity}
\end{equation}
Since $\operatorname{Re}\alpha=2k^2h^2/(k^2h^2+4)$ and
$-\operatorname{Im}\alpha=4kh/(k^2h^2+4)>0$, the real and imaginary
parts of \eqref{eq:abc-source-identity} give
\begin{align}
\sum_{e\in\mathcal E_{T_i^m}^{0}}
\left|\llbracket u_{h,T_i^m}^{\mathrm{I}}\rrbracket_e\right|^2
+\operatorname{Re}\alpha\sum_{e\in\mathcal E_{T_i^m}^{B}}
\left|\llbracket u_{h,T_i^m}^{\mathrm{I}}\rrbracket_e\right|^2
-k^2\|u_{h,T_i^m}^{\mathrm{I}}\|_{0,T_i^m}^2
&=\operatorname{Re}(f,u_{h,T_i^m}^{\mathrm{I}})_{T_i^m},
\label{eq:abc-source-real}\\
(-\operatorname{Im}\alpha)\sum_{e\in\mathcal E_{T_i^m}^{B}}
\left|\llbracket u_{h,T_i^m}^{\mathrm{I}}\rrbracket_e\right|^2
&=-\operatorname{Im}(f,u_{h,T_i^m}^{\mathrm{I}})_{T_i^m}\ge0.
\label{eq:abc-source-imag}
\end{align}

We next estimate $\|u_{h,T_i^m}^{\mathrm{I}}\|_{0,T_i^m}$.
For any $v_h\in W_h(T_i^m)$, let $v_{j,r}$ denote the value of $v_h$
on the $(j,r)$-th fine element. Along each row,
\[
2v_{j,r}=v_{1,r}+v_{n_x,r}
+\sum_{\ell=1}^{j-1}(v_{\ell+1,r}-v_{\ell,r})
-\sum_{\ell=j}^{n_x-1}(v_{\ell+1,r}-v_{\ell,r}),
\]
and hence, by the Cauchy--Schwarz inequality,
\[
|v_{j,r}|^2\le |v_{1,r}|^2+|v_{n_x,r}|^2
+\frac{n_x-1}{2}\sum_{\ell=1}^{n_x-1}|v_{\ell+1,r}-v_{\ell,r}|^2.
\]
Summing this estimate over all cells and repeating the same argument in
the vertical direction gives
\begin{equation}
\|v_h\|_{0,T_i^m}^2
\le \frac{(H+2mh)^2}{4}
\sum_{e\in\mathcal E_{T_i^m}^{0}}|\llbracket v_h\rrbracket_e|^2
+\frac{(H+2mh)h}{2}
\sum_{e\in\mathcal E_{T_i^m}^{B}}|\llbracket v_h\rrbracket_e|^2.
\label{eq:abc-source-trace-poincare}
\end{equation}
Applying \eqref{eq:abc-source-trace-poincare} to
$u_{h,T_i^m}^{\mathrm{I}}$ and using \eqref{eq:abc-source-real}, we obtain
\begin{align*}
\left(1-\frac{k^2(H+2mh)^2}{4}\right)
\|u_{h,T_i^m}^{\mathrm{I}}\|_{0,T_i^m}^2
&\le \frac{(H+2mh)^2}{4}
\operatorname{Re}(f,u_{h,T_i^m}^{\mathrm{I}})_{T_i^m}\\
&\quad+\left[\frac{(H+2mh)h}{2}
-\frac{(H+2mh)^2}{4}\operatorname{Re}\alpha\right]
\sum_{e\in\mathcal E_{T_i^m}^{B}}
\left|\llbracket u_{h,T_i^m}^{\mathrm{I}}\rrbracket_e\right|^2.
\end{align*}
Using \eqref{eq:abc-source-imag} together with
\begin{align*}
\frac{(H+2mh)h/2-(H+2mh)^2\operatorname{Re}\alpha/4}
{-\operatorname{Im}\alpha}
&=\frac{H+2mh}{2k}
\left(1-\frac{k^2h(H+2mh-h)}{4}\right)\\
&\le\frac{H+2mh}{2k},
\end{align*}
we further obtain
\begin{align*}
\left(1-\frac{k^2(H+2mh)^2}{4}\right)
\|u_{h,T_i^m}^{\mathrm{I}}\|_{0,T_i^m}^2
&\le \frac{(H+2mh)^2}{4}
\operatorname{Re}(f,u_{h,T_i^m}^{\mathrm{I}})_{T_i^m}\\
&\quad-\frac{H+2mh}{2k}
\operatorname{Im}(f,u_{h,T_i^m}^{\mathrm{I}})_{T_i^m}\\
&\le \frac{H+2mh}{2k}
\sqrt{1+\frac{k^2(H+2mh)^2}{4}}
\|f\|_{0,T_i^m}\|u_{h,T_i^m}^{\mathrm{I}}\|_{0,T_i^m}.
\end{align*}
Here the last inequality follows by applying the Cauchy--Schwarz
inequality to the real and imaginary parts of
$(f,u_{h,T_i^m}^{\mathrm{I}})_{T_i^m}$.

By Assumption \ref{assumption}, $k(H+2mh)<1$, and therefore
$1-k^2(H+2mh)^2/4\ge3/4$. Hence
\begin{equation}
\|u_{h,T_i^m}^{\mathrm{I}}\|_{0,T_i^m}
\le\frac{\sqrt{5}}{3}\frac{H+2mh}{k}\|f\|_{0,T_i^m}.
\label{eq:abc-source-l2}
\end{equation}
It remains to estimate
$a_i(u_{h,T_i^m}^{\mathrm{I}},u_{h,T_i^m}^{\mathrm{I}})$.
By the definition of $a_i$ and $|\alpha|\le kh$,
\[
\frac1h a_i(u_{h,T_i^m}^{\mathrm{I}},u_{h,T_i^m}^{\mathrm{I}})
\le\sum_{e\in\mathcal E_{T_i^m}^{0}}
\left|\llbracket u_{h,T_i^m}^{\mathrm{I}}\rrbracket_e\right|^2
+kh\sum_{e\in\mathcal E_{T_i^m}^{B}}
\left|\llbracket u_{h,T_i^m}^{\mathrm{I}}\rrbracket_e\right|^2.
\]
Using again the real and imaginary parts of \eqref{eq:abc-source-identity},
we have
\begin{align*}
\frac1h a_i(u_{h,T_i^m}^{\mathrm{I}},u_{h,T_i^m}^{\mathrm{I}})
&\le k^2\|u_{h,T_i^m}^{\mathrm{I}}\|_{0,T_i^m}^2
+\operatorname{Re}(f,u_{h,T_i^m}^{\mathrm{I}})_{T_i^m}\\
&\quad-\left(1-\frac{kh}{2}+\frac{k^2h^2}{4}\right)
\operatorname{Im}(f,u_{h,T_i^m}^{\mathrm{I}})_{T_i^m}.
\end{align*}
Since $0<kh\le k(H+2mh)<1$,
\[
0<1-\frac{kh}{2}+\frac{k^2h^2}{4}\le1.
\]
Together with $-\operatorname{Im}(f,u_{h,T_i^m}^{\mathrm{I}})_{T_i^m}\ge0$,
this gives
\[
\frac1h a_i(u_{h,T_i^m}^{\mathrm{I}},u_{h,T_i^m}^{\mathrm{I}})
\le k^2\|u_{h,T_i^m}^{\mathrm{I}}\|_{0,T_i^m}^2
+\sqrt{2}\|f\|_{0,T_i^m}\|u_{h,T_i^m}^{\mathrm{I}}\|_{0,T_i^m}.
\]
Finally, inserting \eqref{eq:abc-source-l2}, we obtain
\begin{align*}
a_i(u_{h,T_i^m}^{\mathrm{I}},u_{h,T_i^m}^{\mathrm{I}})
&\le h\left[\frac59(H+2mh)^2
+\frac{\sqrt{10}}{3}\frac{H+2mh}{k}\right]\|f\|_{0,T_i^m}^2\\
&\le\frac{2(H+2mh)h}{k}\|f\|_{0,T_i^m}^2.
\end{align*}
This completes the proof.

\bibliographystyle{siamplain}
\bibliography{bibliography}

\end{document}